\documentclass[11pt,a4paper]{article}
\usepackage[utf8]{inputenc}
\usepackage[english]{babel}
\usepackage[T1]{fontenc}

\usepackage{amsmath}
\usepackage{amsthm}
\usepackage{amsfonts}
\usepackage{amssymb}
\usepackage{scrextend}
\usepackage[colorlinks=true, allcolors=blue]{hyperref}
\usepackage{cleveref}
\usepackage{graphicx}
\usepackage{enumerate}
\usepackage{lmodern}
\usepackage{wrapfig}
\usepackage{mathrsfs}
\usepackage{subcaption}
\usepackage[onelanguage,ruled,vlined,linesnumbered]{algorithm2e}
\usepackage{siunitx}
\usepackage{pdfpages}
\usepackage{float}
\usepackage[colorinlistoftodos]{todonotes}
\usepackage{comment}
\usepackage[parfill]{parskip}
\usepackage{bm}
\usepackage{xcolor}
\usepackage{makecell}
\usepackage{colortbl}
\usepackage{tablefootnote}
\usepackage{cancel}

\usepackage{diagbox}

\usepackage{tikz}
\usetikzlibrary{shapes,graphs,graphs.standard}

\newcommand{\ru}[1]{\textbf{R#1}}

\newtheorem{theorem}{Theorem}

\newtheorem{corollary}[theorem]{Corollary}
\newtheorem{definition}[theorem]{Definition}

\crefname{proposition}{Proposition}{Propositions}
\newtheorem{lemma}[theorem]{Lemma}
\newtheorem{conjecture}[theorem]{Conjecture}

\newtheorem{observation}[theorem]{Observation}

\crefname{claim}{claim}{claims}
\Crefname{claim}{Claim}{Claims}

\usepackage[left=2cm,right=2cm,top=2cm,bottom=2cm]{geometry}

\DeclareMathOperator{\mad}{mad}

\usepackage{authblk}

\title{\bf $2$-distance, injective, and exact square list-coloring of planar graphs with maximum degree 4}

\author[1]{Hoang La\thanks{xuan-hoang.la@lirmm.fr}}
\author[2]{Kenny \v{S}torgel\thanks{kennystorgel.research@gmail.com}}
\affil[1]{LIRMM, Université de Montpellier, CNRS, Montpellier, France}
\affil[2]{Faculty of Information Studies in Novo mesto, Slovenia}
\begin{document}
  \maketitle

\begin{abstract}
In the past various distance based colorings on planar graphs were introduced. We turn our focus to three of them, namely $2$-distance coloring, injective coloring, and exact square coloring. A $2$-distance coloring is a proper coloring of the vertices in which no two vertices at distance $2$ receive the same color, an injective coloring is a coloring of the vertices in which no two vertices with a common neighbor receive the same color, and an exact square coloring is a coloring of the vertices in which no two vertices at distance exactly $2$ receive the same color. We prove that planar graphs with maximum degree $\Delta = 4$ and girth at least $4$ are $2$-distance list $(\Delta + 7)$-colorable and injectively list $(\Delta + 5)$-colorable. Additionally, we prove that planar graphs with $\Delta = 4$ are injectively list $(\Delta + 7)$-colorable and exact square list $(\Delta + 6)$-colorable.
\end{abstract}

\section{Introduction}

A \emph{$2$-distance coloring} of a graph $G$ is a proper coloring of the vertices of $G$ such that no pair of vertices at distance at most $2$ receive the same color. The \emph{$2$-distance chromatic number} of a graph $G$, denoted by $\chi^2(G)$, is the smallest integer $k$ such that there exists a $2$-distance coloring of $G$ with $k$ colors. An \emph{injective coloring} of a graph $G$ is a coloring of the vertices of $G$ in which every pair of vertices with a common neighbor receive distinct colors. The \emph{injective chromatic number}, denoted by $\chi^i(G)$, is the smallest integer $k$ such that there exists an injective coloring of $G$ with $k$ colors. An \emph{exact square coloring} of a graph $G$ is a coloring of the vertices of $G$ in which every pair of vertices at distance exactly $2$ receive distinct colors. The \emph{exact square chromatic number}, denoted by $\chi^{\#2}(G)$, is the smallest integer $k$ such that there exists an exact square coloring of $G$ with $k$ colors. Given a list assignment $L$ of $G$, a $2$-distance (injective, exact square) coloring $\phi$ of $G$ is called a \emph{$2$-distance (injective, exact square) list-coloring} if $\phi(v)\in L(v)$ for every $v\in V(G)$. A graph $G$ is \emph{$2$-distance (injective, exact square) $k$-choosable} if $G$ admits a $2$-distance list-coloring for any list assignment $L$ with $|L(v)| \ge k$ for each $v\in V(G)$. The \emph{$2$-distance (injective, exact square) choosability number} of $G$, denoted by $\chi_\ell^2(G)$ ($\chi_\ell^i(G)$, $\chi_\ell^{\#2}(G)$), is the smallest integer $k$ such that $G$ is $2$-distance (injective, exact square) $k$-choosable. 

Unlike the $2$-distance coloring, both the injective coloring and the exact square coloring are not necessarily proper, i.e., adjacent vertices can receive the same color, provided that they satisfy certain conditions. For instance, in the exact square coloring two vertices can be colored with the same color if they are adjacent, and in the injective coloring two vertices can be colored with the same color if they are adjacent and do not share a common neighbor. See \Cref{fig:coloring_example} for a comparison of these colorings.

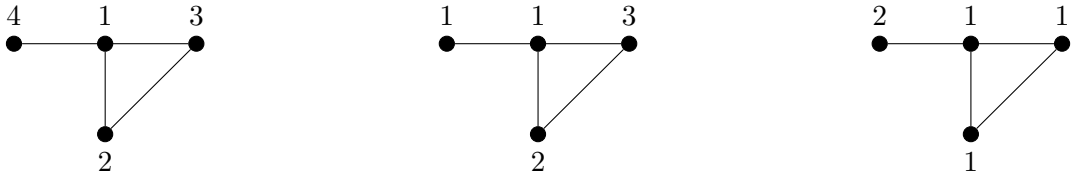
\begin{figure}[H]
\begin{subfigure}[b]{0.33\textwidth}
\centering
\begin{tikzpicture}[scale=0.6]{thick}
\begin{scope}[every node/.style={fill,circle,draw,minimum size=1pt,inner sep=2}]
    \node[label={above:4}] (1') at (8,0) {};
    \node[label={above:1}] (2') at (10,0) {};
    \node[label={below:2}] (20') at (10,-2) {};
    \node[label={above:3}] (3') at (12,0) {};
\end{scope}

\begin{scope}[every edge/.style={draw=black}]
    \path (1') edge (3');
    \path (2') edge (20');
    \path (3') edge (20');
\end{scope}
\end{tikzpicture}
\caption{A 2-distance coloring.}
\end{subfigure}
\begin{subfigure}[b]{0.33\textwidth}
\centering
\begin{tikzpicture}[scale=0.6]{thick}
\begin{scope}[every node/.style={fill,circle,draw,minimum size=1pt,inner sep=2}]
    \node[label={above:1}] (1') at (8,0) {};
    \node[label={above:1}] (2') at (10,0) {};
    \node[label={below:2}] (20') at (10,-2) {};
    \node[label={above:3}] (3') at (12,0) {};
\end{scope}

\begin{scope}[every edge/.style={draw=black}]
    \path (1') edge (3');
    \path (2') edge (20');
    \path (3') edge (20');
\end{scope}
\end{tikzpicture}
\caption{An injective coloring.}
\end{subfigure}
\begin{subfigure}[b]{0.33\textwidth}
\centering
\begin{tikzpicture}[scale=0.6]{thick}
\begin{scope}[every node/.style={fill,circle,draw,minimum size=1pt,inner sep=2}]
    \node[label={above:2}] (1') at (8,0) {};
    \node[label={above:1}] (2') at (10,0) {};
    \node[label={below:1}] (20') at (10,-2) {};
    \node[label={above:1}] (3') at (12,0) {};
\end{scope}

\begin{scope}[every edge/.style={draw=black}]
    \path (1') edge (3');
    \path (2') edge (20');
    \path (3') edge (20');
\end{scope}
\end{tikzpicture}
\caption{An exact square coloring.}
\end{subfigure}
\caption{A 2-distance coloring, injective coloring, and exact square coloring of the same graph.\label{fig:coloring_example}}
\end{figure}

It is therefore easy to observe that every $2$-distance coloring is an injective coloring and every injective coloring is an exact square coloring. Thus, for every graph $G$ we have the following chain of inequalities:

$$\chi^{\#2}(G)\le \chi^i(G)\le \chi^2(G).$$

Moreover, $\chi^{\#2}(G) = \chi^i(G)$ in the case of triangle-free graphs, i.e., graphs in which no pair of adjacent vertices share a common neighbor.

The notion of distance based colorings was first introduced in 1969 by Kramer and Kramer~\cite{kramer2,kramer1} when they introduced the notion of a $p$-distance coloring. In this type of coloring, we require that vertices at distance at most $p$ receive distinct colors. When $p=1$ we obtain the familiar proper coloring. Thus, a $p$-distance coloring is a generalization of the the classical proper coloring. Throughout the years, $p$-distance colorings, in particular the case when $p=2$, became a focus for many researchers, see~\cite{bi12,havet2,havet,tho18}. Only in recent years, several results were investigated with regards to the $2$-distance coloring of planar graphs, see~\cite{BouDesMeyPie2021,CheMiaZho22,FedHelSub21}. For more of the recent results see also~\cite{la21,lm21g21d3,lm21}. It is easy to observe that for every graph $G$, where $\Delta(G)$ (or simply $\Delta$ when $G$ is clear from the context) is the maximum degree of $G$, $\Delta(G) + 1\le \chi^2(G)\le \Delta^2(G) + 1$. Moreover, the upper bound, which follows from a greedy algorithm, is known to be tight for the family of Moore graphs (see, e.g., \cite{lm21g7d10g8d6}). A famous conjecture of Wegner from 1977~\cite{wegner} states that for planar graphs $\chi^2(G)$ is linear in terms of $\Delta$.

\begin{conjecture}[Wegner \cite{wegner}]
\label{conj:Wegner}
Let $G$ be a planar graph with maximum degree $\Delta$. Then,
$$
\chi^2(G) \leq \left\{
    \begin{array}{ll}
        7, & \mbox{if } \Delta\leq 3, \\
        \Delta + 5, & \mbox{if } 4\leq \Delta\leq 7,\\
        \left\lfloor\frac{3}{2}\Delta\right\rfloor + 1, & \mbox{if } \Delta\geq 8.
    \end{array}
\right.
$$
\end{conjecture}

If true, then these upper bounds are tight, as there exist graphs that attain them (see~\cite{wegner}). In 2018, the case when $\Delta\le 3$ was proved independently by Thomassen~\cite{tho18} and by Hartke \textit{et al.}~\cite{har16}. Additionally, for $\Delta\geq 8$, Havet \textit{et al.}~\cite{havet} proved that the bound is $\frac{3}{2}\Delta(1+o(1))$. Moreover, \Cref{conj:Wegner} is known to be true for some subfamilies of planar graphs (e.g., $K_4$-minor free graphs~\cite{lwz03}). 

In~\cite{lm21g7d10g8d6}, La and Montassier presented a summary of the latest known results regarding the $2$-distance coloring of planar graphs for different girth values, where \emph{girth} of a graph $G$, denoted by $g(G)$, is defined as the length of the shortest cycle. 
An additional more recent result is due to Bousquet \textit{et al.}~\cite{BouDesMeyPie2021}. They improved a general result, in the case of the $2$-distance coloring, stating that $2\Delta + 7$ colors are sufficient for all planar graphs with maximum degree between $6$ and $31$. Additionally, in~\cite{BouMeyDesPie22}, the same authors proved that 12 colors are sufficient when $\Delta = 4$, the case that we are particularly interested in.

The injective coloring was first introduced in 2002 by Hahn, Kratochv\'{i}l, \v{S}ir\'{a}\v{n}, and Sotteau~\cite{HahnKraSirSott02}. The authors proved that for every graph $G$, $\Delta\le \chi^i(G)\le \Delta^2 - \Delta + 1$. They also characterized the regular graphs which achieve the lower bound, and the graphs which attain the upper bound. In 2005, Doyon, Hahn, and Raspaud~\cite{DoyHahnRas10}\footnote{The manuscript was presented in 2005 and the paper appeared in journal in 2010.} presented the first results on injective colorings of planar graphs and later Chen \textit{et al.}~\cite{ChenHahnRasWang12} proved that for every $K_4$-minor free graph $G$, $\chi^i(G)\le \lceil \frac{3}{2}\Delta\rceil$. In the same paper they also posed the first conjecture which was proven to be incorrect by Lu\v{z}ar and \v{S}krekovski~\cite{LuzSkre15} who provided an infinite family of planar graphs with small maximum degree (between $4$ and $7$), or of even maximum degree, for which the original conjecture is false.
Although, the original conjecture was supported by several results proving that in the case of planar graphs with girth at least $5$, $\Delta + C$ colors are sufficient, where $C$ is a small constant (see, e.g.,~\cite{BorIva11,BuChenRasWang09,CraKimYu11,DongLin13,LuzSkreTan09}). 
Moreover, Lu\v{z}ar and \v{S}krekovski~\cite{LuzSkre15} proposed a new Wegner type conjecture.

\begin{conjecture}[Lu\v{z}ar and \v{S}krekovski~\cite{LuzSkre15}]
\label{conj:injective2}
Let $G$ be a planar graph with maximum degree $\Delta$. Then,
$$
\chi^i(G) \leq \left\{
    \begin{array}{ll}
        5, & \mbox{if } \Delta\leq 3, \\
        \Delta + 5, & \mbox{if } 4\leq \Delta\leq 7,\\
        \left\lfloor\frac{3}{2}\Delta\right\rfloor + 1, & \mbox{if } \Delta\geq 8.
    \end{array}
\right.
$$
\end{conjecture}

Note that since injective coloring is a relaxation of the $2$-distance coloring, proving Wegner's conjecture would prove \Cref{conj:injective2}, except in the case of subcubic graphs, i.e., the class of graphs with maximum degree $3$. Brimkov \textit{et al.}~\cite{BriEdmLazLidMesWal17} proved that $5$ colors suffice for subcubic planar graphs with girth at least $6$, but in general that case is still open. If true, then the conjectured upper bound for subcubic graphs is also tight (see, e.g.,~\cite{LuzSkre15}). For the sake of completeness we present a table summarizing the latest known results regarding the injective chromatic number of planar graphs for different girth values. A somewhat different table was presented in 2017 by Brimkov \textit{et al.}~\cite{BriEdmLazLidMesWal17}.

\begin{table}[!htb]
\begin{center}
\scalebox{0.8}{%
\begin{tabular}{|c||c|c|c|c|c|c|c|c|}
\hline
\backslashbox{$g_0$ \kern-1em}{\kern-1em $\chi^i$} & $\Delta$ & $\Delta+1$ & $\Delta+2$ & $\Delta+3$ & $\Delta+4$ & $\Delta+5$ & $\Delta+6$ & $\Delta+7$\\
\hline \hline
$3$ & \slashbox{\phantom{}}{} & \slashbox{\phantom{}}{}~\cite{ChenHahnRasWang12} & \cancel{$\Delta\geq 4$}~\cite{LuzSkre15} & \makecell{\cancel{$\Delta\geq 4$}~\cite{LuzSkre15}\\$\Delta = 3$~\cite{ChenHahnRasWang12}} & \cancel{$\Delta\geq 4$}~\cite{LuzSkre15} & \cancel{$\Delta\geq 10$}~\cite{LuzSkre15} & \cancel{$\Delta\geq 12$}~\cite{LuzSkre15} & \cellcolor{gray!50}\colorbox{gray!50}{\makecell{\cancel{$\Delta\geq 14$}~\cite{LuzSkre15}\\$\Delta = 4$}}\\
\hline
$4$ & \slashbox{\phantom{}}{} & \cancel{$\Delta\geq 4$}~\cite{LuzSkreTan09} & \cancel{$\Delta\geq 6$}~\cite{LuzSkreTan09} & \makecell{\cancel{$\Delta\geq 8$}~\cite{LuzSkreTan09}\\$\Delta = 3$~\cite{ChenHahnRasWang12}} & \cancel{$\Delta\geq 10$}~\cite{LuzSkreTan09} & \cellcolor{gray!50}\colorbox{gray!50}{\makecell{\cancel{$\Delta\geq 12$}~\cite{LuzSkreTan09}\\$\Delta = 4$}} & \cancel{$\Delta\geq 14$}~\cite{LuzSkreTan09} & \cancel{$\Delta\geq 16$}~\cite{LuzSkreTan09}\\
\hline
$5$ & \slashbox{\phantom{}}{} & & & \makecell{$\Delta = 3$~\cite{ChenHahnRasWang12}\\$\Delta\geq 35$~\cite{DongLin14}} & $\Delta\geq 11$~\cite{BuHua18} & \cellcolor{gray!50}\colorbox{gray!50}{\makecell{$\Delta = 4$\\$\Delta\geq 11$~\cite{BuHua18}}} & $\Delta\geq 3$~\cite{DongLin14} & \\
\hline
$6$ & \slashbox{\phantom{}}{}~\cite{LuzSkreTan09} & $\Delta\geq 17$~\cite{DongLin13} & \makecell{$\Delta = 3$~\cite{BriEdmLazLidMesWal17}\\$\Delta\geq 8$~\cite{BuLu13}} & $\Delta\geq 3$~\cite{DongLin13} & & & & \\
\hline
$7$ & $\Delta\geq 16$~\cite{BorIva11} & $\Delta\geq 7$~\cite{BuLu12} & \makecell{$\Delta = 3$~\cite{LuzSkreTan09}\\$\Delta\geq 4$~\cite{CraKimYu11}} & & & & & \\
\hline
$8$ & $\Delta\geq 10$~\cite{BorIva11} & $\Delta\geq 5$~\cite{BuLu13} & & & & & & \\
\hline
$9$ & $\Delta\geq 9$~\cite{BuLuYang15} & $\Delta\geq 4$~\cite{CraKimYu2010} & & & & & & \\
\hline
$10$ & \makecell{$\Delta\geq 6$~\cite{BorIva11}\\ \cancel{$\Delta=3$}~\cite{LuzSkreTan09}} & $\Delta\geq 3$~\cite{LuzSkreTan09} & & & & & & \\
\hline
$11$ & & & & & & & & \\
\hline
$12$ & $\Delta\geq 5$~\cite{BorIva11} & & & & & & & \\
\hline
$13$ & $\Delta\geq 4$~\cite{CraKimYu2010} & & & & & & & \\
\hline
$19$ & $\Delta\geq 3$~\cite{LuzSkreTan09} & & & & & & & \\
\hline
\end{tabular}}
\caption{Summary of the latest results with a coefficient 1 before $\Delta$ in the upper bound of $\chi^i$ for different girth values while for the crossed out cases there exist counterexamples.}
\label{recap table inj}
\end{center}
\end{table}

\Cref{recap table inj} reads as follows. For example, the result from line ``7'' and column ``$\Delta$'' states that every planar graph $G$ of girth at least 7 and maximum degree $\Delta(G)\ge 16$ satisfies $\chi^i(G)\leq \Delta(G)$. In the first column (``$\Delta$''), the first four crossed out cases follow from the graph in \Cref{fig:injCE}. In this same column, row ``$10$'' with a crossed out value of $\Delta = 3$ corresponds to a construction of a planar graph $G$ with $\Delta(G) = 3$, girth $10$, and $\chi^i(G)\geq \Delta+1$~\cite{LuzSkreTan09}. Similar results in the same vein in columns ``$\Delta+1$'' to ``$\Delta+4$'' in row ``3'' are presented in~\cite{ChenHahnRasWang12} and~\cite{LuzSkre15}. There exists a planar graph $G$ with girth 3 and $\chi^i(G)\geq \lfloor\frac{3}{2}\Delta\rfloor+1$ for all $\Delta\geq 8$~\cite{LuzSkre15}, which justifies the remaining crossed out cases in row ``3''. Similarly, there also exists a planar graph $G$ with girth 4 and $\chi^i(G)\geq \frac{3}{2}\Delta$ for all $\Delta\geq 3$~\cite{LuzSkreTan09}, which justifies the crossed out cases in row ``4''. Every highlighted result without a reference in \Cref{recap table inj} is a part of our contribution in this paper. Note that the highlighted result in row ``5'' follows from the result in row ``4'',

\begin{figure}[H]
\centering
\begin{tikzpicture}
\begin{scope}[every node/.style={circle,thick,draw,minimum size=1pt,inner sep=2}]
    \node[fill] (y3) at (3,2) {};
    
    \node[fill] (y'3) at (3,3) {};
    \node[fill] (y'32) at (2,3) {};
    \node[fill] (y'33) at (4,3) {};
    
    \node[fill] (x'3) at (3,4) {};
    \node[fill] (x'32) at (2,4) {};
    \node[fill] (x'33) at (4,4) {};
    
    \node[fill] (x3) at (3,5) {};
    
    \node[draw=none] at (5.5,2.4) {$\Delta$ vertices};
\end{scope}

\begin{scope}[every edge/.style={draw=black,thick}]
	\path (x3) edge (y3);
	\path (x3) edge (x'32);
	\path (x3) edge (x'33);
	\path (y3) edge (y'32);
	\path (y3) edge (y'33);
	\path (x'32) edge (y'32);
	\path (x'33) edge (y'33);
\end{scope}
\draw (3,3) ellipse (1.5cm and 0.5cm);
\end{tikzpicture}
\caption{A graph with girth 6 and $\chi^i\geq \Delta+1$ (drawn for $\Delta=3$)~\cite{LuzSkreTan09}.}
\label{fig:injCE}
\end{figure}
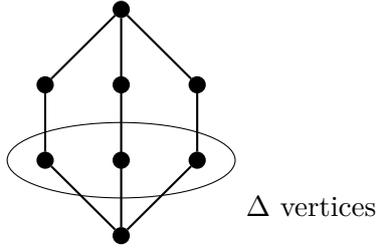

Finally, the study of the exact distance $p$-powers of graphs was started by Simi\'{c}~\cite{Sim83} and exact $p$-distance colorings have first been studied for graphs of bounded expansion~\cite{NesMen12}, see also~\cite{HeuKieQui18}. This parameter received an increasing attention in the last decade, see~\cite{BouEspHarJoa19,FouHocMisNarNasSopVal20,HeuKieQui18,Qui20}. In~\cite{FouHocMisNarNasSopVal20}, Foucaud \textit{et al.} began the first systematic study of the exact square coloring (i.e., exact $2$-distance coloring) with respect to the maximum degree. In their paper, they considered the exact square coloring for some specific classes of subcubic graphs. Both in the case of subcubic $K_4$-minor free graphs and subcubic planar bipartite graphs, they prove that $4$ colors suffice in any exact square coloring. Moreover, they provide examples attaining this bound. Furthermore, they prove that \Cref{conj:injective2} holds for fullerene graphs, i.e., cubic planar graphs in which every face has size $5$ or $6$. Since for every graph $G$ with girth at least $4$ we have equality between $\chi^{\#2}(G)$ and $\chi^i(G)$, all the results in \Cref{recap table inj}, except for the row corresponding to girth at least $3$, hold also for the exact square coloring.

In this paper we present some results regarding the $2$-distance, injective, and exact square coloring of planar graphs of small girth ($3$ or $4$), thus filling some gaps from the literature. We focus on graphs which have maximum degree $4$, also known in the literature as \emph{subquartic graphs}. 

With respect to the $2$-distance coloring, we consider planar graphs with maximum degree $4$ and girth at least $4$ to get the following result for $2$-distance choosability number.

\begin{theorem} \label{thm:2dg4}
If $G$ is a planar graph with $g(G)\geq 4$ and $\Delta(G)=4$, then $\chi^2_\ell(G)\leq \Delta(G) + 7$.
\end{theorem}

In the case of the injective coloring, we prove that the same bound as in \Cref{thm:2dg4} also holds, but without any restrictions on the girth.

\begin{theorem} \label{thm:injg3}
If $G$ is a planar graph with $\Delta(G)=4$, then $\chi^i_\ell(G)\leq \Delta(G) + 7$.
\end{theorem}

Furthermore, in the case when girth is at least $4$, we improve the bound implied by \Cref{thm:2dg4} to $\Delta + 5$ colors.

\begin{theorem} \label{thm:injg4}
If $G$ is a planar graph with $g(G)\geq 4$ and $\Delta(G)=4$, then $\chi^i_\ell(G)\leq \Delta(G) + 5$.
\end{theorem}

Finally, in the case of the exact square coloring, we improve the bound implied by \Cref{thm:injg3} to $\Delta + 6$ colors.

\begin{theorem} \label{thm:exact}
If $G$ is a planar graph with $\Delta(G)=4$, then $\chi^{\#2}_\ell(G)\leq \Delta(G) + 6$.
\end{theorem}

The structure of the paper is organized as follows. In \Cref{prelim}, we present notations and auxiliary results. The proofs of \Cref{thm:2dg4,thm:injg3,thm:injg4,thm:exact} are then provided in \Cref{2dg4,injg3,injg4,exact}. We conclude the paper with some additional remarks in \Cref{conclusion}.

\section{Preliminaries}\label{prelim}

We denote by $F(G)$ the set of faces of a planar graph $G$. Given two vertices $u$ and $v$ of a graph $G$, and a set $S\subseteq V(G)$, we denote by $d_G(u,v)$ the distance between $u$ and $v$ in $G$, and define the distance between the vertex $v$ and the set $S$ as $d_G(v,S) :=\min\{d_G(v,w)\mid w\in S\}$. For a vertex $v$ of $G$, we define the \emph{neighborhood}, \emph{$2$-distance neighborhood}, and \emph{exact $2$-distance neighborhood} respectively as follows:
\begin{itemize}
    \item $N_G(v):=\{w\in V(G)\mid d_G(v,w)=1\}$,
    \item $N^*_G(v):=\{w\in V(G)\mid 1\leq d_G(v,w)\leq 2\}$,
    \item $N^{\#2}_G(v):=\{w\in V(G)\mid d_G(v,w) = 2\}$.
\end{itemize}
Additionally, we define the \emph{$2$-distance degree} and the \emph{exact $2$-distance degree} respectively as follows:
\begin{itemize}
    \item $d^*_G(v):= |N_G^*(v)|$,
    \item $d^{\#2}_G(v):= |N_G^{\#2}(v)|$.
\end{itemize}
Furthermore, we define by $d$-vertex ($d^+$-vertex, $d^-$-vertex) a vertex of degree $d$ (at least $d$, at most $d$) and by $d$-face ($d^+$-face, $d^-$-face) a face of size $d$ (at least $d$, at most $d$). For $S\subseteq V(G)\cup E(G)$, we denote by $G-S$ the graph obtained from $G$ by removing the elements from $S$. Similarly, we denote by $G+S$ the graph obtained from $G$ by adding the elements from $S$.

We will drop the subscript in the notations when the graph is clear from the context. Also for conciseness, from now on, when we say ``to color'' a vertex, it means to color such a vertex under the constraint of the current $2$-distance, injective, or exact square list-coloring. We also say that a vertex $u$ ``sees'' a vertex $v$ if $u$ and $v$ are at distance at most $2$ from each other (resp. share a neighbor, or are at distance exactly 2 from each other) for the $2$-distance (resp. injective, or exact square) list-coloring.
We abuse this vocabulary to say that $u$ ``sees a color'' $c$ when $u$ sees a vertex colored $c$. 

As a drawing convention for the rest of the figures, black vertices have a fixed degree, all of their edges are drawn, and white vertices may have a higher degree than what is drawn. We often show the reducibility of a configuration by extending a certain precoloring. We denote by $L(u)$ the list of remaining colors for an uncolored vertex $u$. We will also indicate the lower bound on the size of $|L(u)|$ next to the relevant vertices on the figures. 

A useful tool in proving coloring results is also Hall's Theorem, which guarantees distinct colors for a set of vertices.
\begin{theorem}[Hall~\cite{Hal35}]
	\label{thm:Hall}
	A bipartite graph with partition sets $A$ and $B$ admits a matching that covers every vertex
	of $A$ if and only if for every set $S \subseteq A$ the number of vertices of $B$ with a neighbor in $S$ is at least $|S|$.
\end{theorem}

If we consider $A$ to be the set of vertices and $B$ to be the set of colors, then we get the following result for list coloring.

\begin{corollary}[Hall]
\label{cor:Hall}
Let $L$ be a list-assignment on a set of vertices $V$, for every subset $S\subseteq V$ of size $k$, if $|\bigcup_{u\in S}L(u)|\geq k$, then $V$ is $L$-list-colorable where every vertex has a distinct color.
\end{corollary}

The proof of each theorem will be structured as follows. First we will consider a counterexample $G$ minimizing the number of vertices plus the number of edges. Such counterexamples are always connected, otherwise a component of $G$ would be a smaller counterexample. Thus, the minimum degree is always at least 1.
After studying the structural properties of $G$, we will proceed with a discharging procedure where we always assign the initial charge $\mu(u)=d(u)-4$ to each vertex $u\in V(G)$, and $\mu(f)=d(f)-4$ to each face $f\in F(G)$. By Euler's formula, we must have 
\begin{equation}\label{equation}
\sum_{u\in V(G)} (d(u)-4) + \sum_{f\in F(G)} (d(f)-4)  < 0.
\end{equation}
Finally, to prove that $G$ does not exist, we will redistribute the charges while preserving the total sum and prove that the final charge of each vertex and face is non-negative, which will be a contradiction to \Cref{equation}.

\section{2-distance list-coloring of triangle-free planar graphs}\label{2dg4}

In this section, we provide the proof of \Cref{thm:2dg4}. Let $G$ be a minimal counterexample to \Cref{thm:2dg4}, namely $G$ has maximum degree 4, girth at least 4, and $\chi^2_\ell(G)\geq 12$. 

\subsection{Structural properties of $G$}\label{2dg4:properties}

We start by proving that $G$ cannot be too sparse. More precisely, we have a lower bound on the minimum degree of $G$.

\begin{lemma}\label{2dg4:minimumDegree}
The minimum degree of $G$ is at least 3.
\end{lemma}

\begin{proof}
If $G$ contains a 1-vertex $v$, then we can simply remove $v$ and color the resulting graph (with 11 colors), which is possible by minimality of $G$. Then, we add $v$ back and extend the coloring (at most $4$ constraints and $11$ colors).

If $G$ contains a vertex $u$ of degree $2$ with neighbors $v$ and $w$ (see \Cref{fig:2dg4:minimumDegree}), then let $H=G-\{u\}$. We color $H+\{vw\}$ (resp. $H$) by minimality if $d_H(v,w)\geq 3$ (resp. $d_H(v,w)\leq 2$). Observe that in both cases, the girth (at least 4) and maximum degree (4) of the resulting graph are preserved. We extend such coloring to $G$ by coloring $u$ which sees only at most $8$ different colors.
\end{proof}

\begin{figure}[H]
\centering
\begin{subfigure}[b]{0.49\textwidth}
\centering
\begin{tikzpicture}[scale=0.6]{thick}
\begin{scope}[every node/.style={circle,draw,minimum size=1pt,inner sep=2}]
    \node[label={above:$v$}] (1) at (0,0) {};
    \node[fill,label={above:$u$},label={below:$3$}] (2) at (2,0) {};
    \node[label={above:$w$}] (3) at (4,0) {};
\end{scope}

\begin{scope}[every edge/.style={draw=black}]
    \path (1) edge (3);
    \path (1) edge[bend left = 60,dashed] (3);
\end{scope}
\end{tikzpicture}
\caption{If $d_{G-\{u\}}(v,w)\geq 3$, then we add the dashed edge.}
\end{subfigure}
\begin{subfigure}[b]{0.49\textwidth}
\centering
\begin{tikzpicture}[scale=0.6]{thick}
\begin{scope}[every node/.style={circle,draw,minimum size=1pt,inner sep=2}]
    \node[label={above:$v$}] (1) at (0,0) {};
    \node[fill,label={above:$u$},label={below:$3$}] (2) at (2,0) {};
    \node[label={above:$w$}] (3) at (4,0) {};
    \node (20) at (2,2) {};
\end{scope}

\begin{scope}[every edge/.style={draw=black}]
    \path (1) edge (3);
    \path (1) edge (20);
    \path (20) edge (3);
\end{scope}
\end{tikzpicture}
\caption{If $d_{G-\{u\}}(v,w)\leq 2$, then $v$ and $w$ already have different colors in a coloring of $G-\{u\}$.}
\end{subfigure}

\caption{Reducible configurations in \Cref{2dg4:minimumDegree}.}
\label{fig:2dg4:minimumDegree}
\end{figure}

Along the same line, we prove that objects will a smaller neighborhood cannot be close together. Otherwise, $G$ would be colorable.

\begin{lemma}\label{2dg4:config1}
Graph $G$ does not contain the following configurations:
\begin{itemize}
\item[(i)] A $3$-vertex incident to two $4$-cycles.
\item[(ii)] A $3$-vertex incident to a $4$-cycle and adjacent to a $3$-vertex. 
\item[(iii)] A $3$-vertex adjacent with two $3$-vertices.
\end{itemize}
\end{lemma}

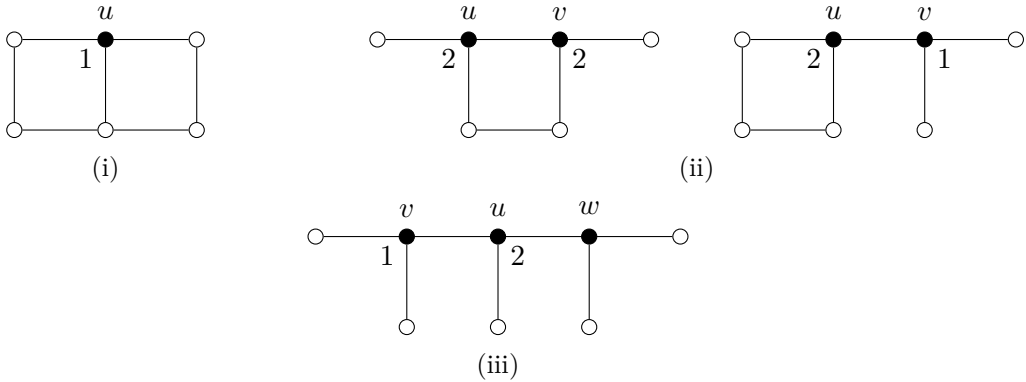
\begin{figure}[H]
\centering
\begin{subfigure}[b]{0.3\textwidth}
\centering
\begin{tikzpicture}[scale=0.6]{thick}
\begin{scope}[every node/.style={circle,draw,minimum size=1pt,inner sep=2}]
    \node (1) at (0,0) {};
    \node (10) at (0,-2) {};
    \node[fill,label={above:$u$},label={below left:$1$}] (2) at (2,0) {};
    \node (20) at (2,-2) {};
    \node (3) at (4,0) {};
    \node (30) at (4,-2) {};
\end{scope}

\begin{scope}[every edge/.style={draw=black}]
    \path (1) edge (3);
    \path (1) edge (10);
    \path (2) edge (20);
    \path (3) edge (30);
    \path (10) edge (20);
    \path (20) edge (30);
\end{scope}
\end{tikzpicture}
\caption{}
\end{subfigure}
\begin{subfigure}[b]{0.6\textwidth}
\centering
\begin{tikzpicture}[scale=0.6]{thick}
\begin{scope}[every node/.style={circle,draw,minimum size=1pt,inner sep=2}]
    \node (1) at (0,0) {};
    \node[fill,label={above:$u$},label={below left:$2$}] (2) at (2,0) {};
    \node (20) at (2,-2) {};
    \node[fill,label={above:$v$},label={below right:$2$}] (3) at (4,0) {};
    \node (30) at (4,-2) {};
    \node (4) at (6,0) {};
    
    \node (1') at (8,0) {};
    \node (10') at (8,-2) {};
    \node[fill,label={above:$u$},label={below left:$2$}] (2') at (10,0) {};
    \node (20') at (10,-2) {};
    \node[fill,label={above:$v$},label={below right:$1$}] (3') at (12,0) {};
    \node (30') at (12,-2) {};
    \node (4') at (14,0) {};
\end{scope}

\begin{scope}[every edge/.style={draw=black}]
    \path (1) edge (4);
    \path (2) edge (20);
    \path (3) edge (30);
    \path (20) edge (30);
    
    \path (1') edge (4');
    \path (1') edge (10');
    \path (2') edge (20');
    \path (3') edge (30');
    \path (10') edge (20');
\end{scope}
\end{tikzpicture}
\caption{}
\end{subfigure}
\begin{subfigure}[b]{0.6\textwidth}
\centering
\begin{tikzpicture}[scale=0.6]{thick}
\begin{scope}[every node/.style={circle,draw,minimum size=1pt,inner sep=2}]
    \node (0) at (-2,0) {};
    \node[fill,label={above:$v$},label={below left:$1$}] (1) at (0,0) {};
    \node (10) at (0,-2) {};
    \node[fill,label={above:$u$},label={below right:$2$}] (2) at (2,0) {};
    \node (20) at (2,-2) {};
    \node[fill,label={above:$w$}] (3) at (4,0) {};
    \node (30) at (4,-2) {};
    \node (4) at (6,0) {};
\end{scope}

\begin{scope}[every edge/.style={draw=black}]
    \path (0) edge (4);
    \path (1) edge (10);
    \path (2) edge (20);
    \path (3) edge (30);
\end{scope}
\end{tikzpicture}
\caption{}
\end{subfigure}
\caption{Reducible configurations in \Cref{2dg4:config1}.}
\label{fig:2dg4:config1}
\end{figure}

\begin{proof} 
We reduce each configuration separately (see \Cref{fig:2dg4:config1}) by precoloring a subgraph of $G$ and extending the coloring to $G$.
\begin{itemize}
    \item[(i)] Let $u$ be a $3$-vertex incident to two $4$-cycles. Let $e$ be the incident edge to $u$ that is also incident to both cycles. Color $G-\{e\}$ by minimality and uncolor $u$. Observe that $d^*_G(u)\leq 10$ so $u$ has at least one available color.
    \item[(ii)] Let $u$ and $v$ be two adjacent $3$-vertices and suppose that $u$ is incident to a $4$-cycle. Color $G-\{uv\}$ and uncolor $u$ and $v$. Observe that $v$ has at least one and $u$ has at least two available colors.
    \item[(iii)] Let $u$ be a $3$-vertex with two $3$-neighbors $v$ and $w$. Color $G-\{uv\}$ and uncolor $u$ and $v$. Observe that $v$ has at least one and $u$ has at least two available colors.
\end{itemize}
This concludes the proof.
\end{proof}

In a $2$-distance coloring, $5$-faces play an important role as they are right at the limit of being ``big'' enough objects that are not directly colorable, but ``small'' enough that a planar graph can contain only those. Thus, we turn our attention to configurations surrounding $5$-faces.

\begin{lemma}\label{2dg4:config2i}
Let $f=v_1v_2v_3v_4v_5$ be a $5$-face in $G$ such that $d(v_1)=d(v_3)=d(v_5)=3$, and $v_2v_3$ is incident to a $4$-face. Let $f'=v'_1v'_2v'_3v'_4v'_5$ be another $5$-face in $G$ such that $d(v'_2)=d(v'_3)=d(v'_5)=3$. Then, $v'_4\neq v_4$ or $v'_5\neq v_5$.
\end{lemma}

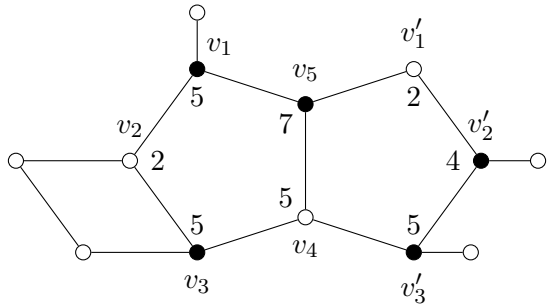
\begin{figure}[H]
\centering
\begin{tikzpicture}[scale=1.5]{thick}
\begin{scope}[every node/.style={circle,draw,minimum size=1pt,inner sep=2}]
    \node[fill,label={above right:$v_1$},label={below:$5$}] (v1) at (0.05,1.31) {};
    \node[label={above:$v_2$},label={right:$2$}] (v2) at (-0.54,0.5) {};
    \node[fill,label={below:$v_3$},label={above:$5$}] (v3) at (0.05,-0.31) {};
    \node[label={below:$v_4$},label={above left:$5$}] (v4) at (1,0) {};
    \node[fill,label={above:$v_5$},label={below left:$7$}] (v5) at (1,1) {};
    \node[label={above:$v'_1$},label={below:$2$}] (v'1) at (1.95,1.31) {};
    \node[fill,label={above:$v'_2$},label={left:$4$}] (v'2) at (2.54,0.5) {};
    \node[fill,label={below:$v'_3$},label={above:$5$}] (v'3) at (1.95,-0.31) {};

    \node (1) at (0.05,1.81) {};
    \node (2) at (-1.54,0.5) {};
    \node (3) at (-0.95,-0.31) {};
    \node (2') at (3.04,0.5) {};
    \node (3') at (2.45,-0.31) {};
\end{scope}

\begin{scope}[every edge/.style={draw=black}]
    \path (v1) edge (v2);
    \path (v2) edge (v3);
    \path (v3) edge (v4);
    \path (v4) edge (v5);
    \path (v5) edge (v1);
    \path (v5) edge (v'1);
    \path (v'1) edge (v'2);
    \path (v'3) edge (v'2);
    \path (v'3) edge (v4);
    \path (v1) edge (1);
    \path (v2) edge (2);
    \path (v3) edge (3);
    \path (2) edge (3);
    \path (v'2) edge (2');
    \path (v'3) edge (3');
\end{scope}
\end{tikzpicture}
\caption{Reducible configuration in \Cref{2dg4:config2i}.}
\label{fig:2dg4:config2i}
\end{figure}

\begin{proof}
Suppose by contradiction that $v'_4=v_4$ and $v'_5=v_5$ (see \Cref{fig:2dg4:config2i}).
First, observe that the vertices of $H=G[\{v_1,v_2,v_3,v_4,v_5,v'_1,v'_2,v'_3\}]$ are all distinct since $g(G)\geq 4$ and due to \Cref{2dg4:config1}(ii). Color $G-\{v_4,v_5\}$ by minimality and uncolor the other vertices of $f$ and $f'$. Observe that the remaining list of colors for these vertices have size: $|L(v_1)|\geq 5$, $|L(v_2)|\geq 2$, $|L(v_3)|\geq 5$, $|L(v_4)|\geq 5$, $|L(v_5)|\geq 7$, $|L(v'_1)|\geq 2$, $|L(v'_2)|\geq 4$, and $|L(v'_3)|\geq 5$. Note that if for any pair of vertices $(u,v)$ in $V(H)$, if $d_G(u,v)\leq 2$ while $d_H(u,v)\geq 3$, then both $u$ and $v$ have at least one more remaining color.

We claim that $L(v_1)\cap L(v'_3)=\emptyset$ or $L(v_3)\cap L(v'_2)=\emptyset$. Suppose by contradiction that $L(v_1)\cap L(v'_3)\neq\emptyset$ and $L(v_3)\cap L(v'_2)\neq\emptyset$. 

If $d_G(v_1,v'_3)\geq 3$. Then, we can color $v_1$ and $v'_3$ with the same color $c$, then we finish by coloring $v'_1$, $v_2$, $v'_2$, $v_3$, $v_4$, and $v_5$ in this order. This is a contradiction so $d_G(v_1,v'_3)\leq 2$. If $d_G(v_1,v'_3)=1$, then $v_1$, $v'_3$, $v_4$, and $v_5$ form the configuration from \Cref{2dg4:config1}(ii). So, $d_G(v_1,v'_3)=2$.

If $d_G(v_3,v'_2)\geq 3$. Then, we can color $v_3$ and $v'_2$ with the same color $c$, then we finish by coloring $v'_1$, $v_2$, $v_1$, $v'_3$, $v_4$, and $v_5$ in this order. This is a contradiction so $d_G(v_3,v'_2)\leq 2$. If $d_G(v_3,v'_2)=1$, then $v'_2$, $v'_3$, $v_3$, and $v_4$ form the configuration from \Cref{2dg4:config1}(ii). So, $d_G(v_3,v'_2)=2$.

As a result, by planarity, $v_1$, $v_3$, $v'_2$, and $v'_3$ must have a common neighbor $u$. However, $uv'_2v'_3$ is a triangle while $g(G)\geq 4$.

Consequently, we have $|L(v_1)\cup L(v'_3)|\geq 10$ or $|L(v'_2)\cup L(v_3)|\geq 9$. Moreover, note that $|L(v_1)\cup L(v'_3)|\geq 6$ and $|L(v_3)\cup L(v'_2)|\geq 6$. Indeed, these inequalities hold when $d_G(v_1,v'_3)\leq 2$ (resp. $d_G(v_3,v'_2)\leq 2$) as $|L(v_1)|\geq 6$ (resp. $|L(v_3)|\geq 6$). And when $d_G(v_1,v'_3)\geq 3$ (resp. $d_G(v_3,v'_2)\geq 3$), we also the same inequalities, otherwise $L(v_1)$ and $L(v'_3)$ (resp. $L(v_3)$ and $L(v'_2)$) will have a common color, in which case $G$ is colorable as seen previously. 
  
Finally, $G$ is $L$-list-colorable by \Cref{cor:Hall} as $|\bigcup_{u\in S}L(u)|\geq k$ for every subset $S$ of size $k$ in $V(H)$.
\end{proof}

\begin{lemma}\label{2dg4:config2ii}
Let $f=v_1v_2v_3v_4v_5$ be a $5$-face in $G$ such that $d(v_1)=d(v_3)=d(v_5)=3$, and $v_2v_3$ is incident to a $4$-face. Let $f'=v'_1v'_2v'_3v'_4v'_5$ be another $5$-face in $G$ such that $d(v'_1)=d(v'_3)=d(v'_5)=3$, and $v'_2v'_3$ is incident to a $4$-face. Then, $v'_1\neq v_1$ or $v'_5\neq v_5$.
\end{lemma}

\begin{figure}[H]
\centering
\begin{tikzpicture}[scale=1.5]{thick}
\begin{scope}[every node/.style={circle,draw,minimum size=1pt,inner sep=2}]
    \node[label={above:$v_2$},label={below:$3$}] (v2) at (0.05,1.31) {};
    \node[fill,label={above:$v_3$},label={right:$4$}] (v3) at (-0.54,0.5) {};
    \node[label={below:$v_4$},label={above:$2$}] (v4) at (0.05,-0.31) {};
    \node[fill,label={below:$v_5$},label={above left:$7$}] (v5) at (1,0) {};
    \node[fill,label={above:$v_1$},label={below left:$7$}] (v1) at (1,1) {};
    \node[label={above:$v'_2$},label={below:$3$}] (v'2) at (1.95,1.31) {};
    \node[fill,label={above:$v'_3$},label={left:$4$}] (v'3) at (2.54,0.5) {};
    \node[label={below:$v'_4$},label={above:$2$}] (v'4) at (1.95,-0.31) {};

    \node (2) at (-0.95,1.31) {};
    \node (3) at (-1.54,0.5) {};
    \node (2') at (2.95,1.31) {};
    \node (3') at (3.54,0.5) {};
\end{scope}

\begin{scope}[every edge/.style={draw=black}]
    \path (v1) edge (v2);
    \path (v2) edge (v3);
    \path (v3) edge (v4);
    \path (v4) edge (v5);
    \path (v5) edge (v1);
    \path (v5) edge (v'4);
    \path (v1) edge (v'2);
    \path (v'3) edge (v'2);
    \path (v'3) edge (v'4);
    \path (v2) edge (2);
    \path (v3) edge (3);
    \path (2) edge (3);
    \path (2') edge (3');
    \path (v'2) edge (2');
    \path (v'3) edge (3');
\end{scope}
\end{tikzpicture}
\caption{Reducible configuration in \Cref{2dg4:config2ii}.}
\label{fig:2dg4:config2ii}
\end{figure}
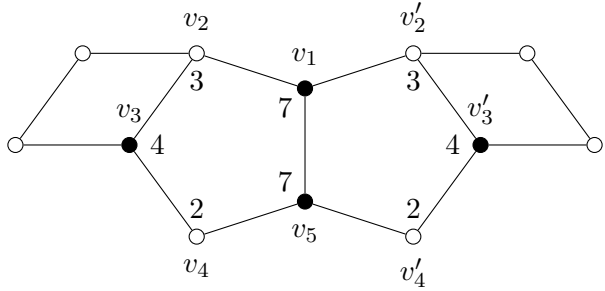

\begin{proof}
Suppose by contradiction that $v'_1=v_1$ and $v'_5=v_5$ (see \Cref{fig:2dg4:config2ii}). First, observe that the vertices of $H=G[\{v_1,v_2,v_3,v_4,v_5,v'_2,v'_3,v'_4\}]$ are all distinct since $g(G)\geq 4$ and due to \Cref{2dg4:config1}(ii). Color $G-\{v_1,v_5\}$ by minimality and uncolor the other vertices of $f$ and $f'$. Observe that the remaining list of colors for these vertices have size: $|L(v_1)|\geq 7$, $|L(v_2)|\geq 3$, $|L(v_3)|\geq 4$, $|L(v_4)|\geq 2$, $|L(v_5)|\geq 7$, $|L(v'_2)|\geq 3$, $|L(v'_3)|\geq 4$, and $|L(v'_4)|\geq 2$. Note that if for any pair of vertices $(u,v)$ in $V(H)$, if $d_G(u,v)\leq 2$ while $d_H(u,v)\geq 3$, then both $u$ and $v$ have at least one more remaining color.  
We start with the following claims:
\begin{itemize}
\item $L(v_4)\cap L(v'_2)=\emptyset$ or $L(v'_4)\cap L(v_2)=\emptyset$.\\
Suppose by contradiction that $L(v_4)\cap L(v'_2)\neq\emptyset$ and $L(v'_4)\cap L(v_2)\neq\emptyset$. Suppose w.l.o.g. that $d_G(v_4,v'_2)\geq 3$. Then, we can color $v_4$ and $v'_2$ with the same color $c$, then we finish by coloring $v'_4$, $v_2$, $v_3$, $v'_3$, $v_1$, and $v_5$ in this order. This is a contradiction so $d_G(v_4,v'_2)\leq 2$ and $d(v'_4,v_2)\leq 2$. If $d_G(v_4,v'_2)=1$, then $v_1$, $v'_2$, $v_4$, and $v_5$ form the configuration from \Cref{2dg4:config1}(ii). By symmetry, we get $d_G(v_4,v'_2)=d_G(v'_4,v_2)=2$. As a result, by planarity, $v_2$, $v'_2$, $v_4$, and $v'_4$ must have a common neighbor $u$. However, $v_1$, $v_2$, $u$, and $v'_2$ form the configuration from \Cref{2dg4:config1}(ii). 

As a consequence, we have $|L(v_4)\cup L(v'_2)|\geq 5$ or $|L(v'_4)\cup L(v_2)|\geq 5$. 

\item $L(v_4)\cap L(v'_3)=\emptyset$ or $L(v'_4)\cap L(v_3)=\emptyset$.\\
Suppose by contradiction that $L(v_4)\cap L(v'_3)\neq\emptyset$ and $L(v'_4)\cap L(v_3)\neq\emptyset$. Suppose w.l.o.g. that $d_G(v_4,v'_3)\geq 3$. Then, we can color $v_4$ and $v'_3$ with the same color $c$, then we finish by coloring $v'_4$, $v'_2$, $v_2$, $v_3$, $v_1$, and $v_5$ in this order. This is a contradiction so $d_G(v_4,v'_3)\leq 2$ and $d_G(v'_4,v_3)\leq 2$. If $d_G(v_4,v'_3)=1$, then $v_4$, $v_5$, $v'_4$, and $v'_3$ form the configuration from \Cref{2dg4:config1}(ii). By symmetry, we get $d_G(v_4,v'_3)=d_G(v'_4,v_3)=2$. As a result, by planarity, $v_3$, $v'_3$, $v_4$, and $v'_4$ must have a common neighbor $u$. However, $v_3v_4u$ is a triangle.

As a consequence, we have $|L(v_4)\cup L(v'_3)|\geq 5$ or $|L(v'_4)\cup L(v_3)|\geq 5$. 

\item $|L(v_2)\cup L(v_3)\cup L(v'_2)\cup L(v'_3)|\geq 8$.\\
Suppose by contradiction that $|L(v_2)\cup L(v_3)\cup L(v'_2)\cup L(v'_3)|\leq 7$. Since $|L(v_3)|\geq 4$ and $|L(v'_3)|\geq 4$, we get $L(v_3)\cap L(v'_3)\neq \emptyset$. Suppose that $d_G(v_3,v'_3)\geq 3$. Then, we can color $v_3$ and $v'_3$ with the same color $c$. Since $L(v_4)\cap L(v'_3)=\emptyset$ or $L(v'_4)\cap L(v_3)=\emptyset$, w.l.o.g. we can color $v_4$ then $v'_4$. Since $L(v_4)\cap L(v'_2)=\emptyset$ or $L(v'_4)\cap L(v_2)=\emptyset$, w.l.o.g. we can color $v_2$ then $v'_2$. Finish by coloring $v_1$ and $v_5$. So, $d_G(v_3,v'_3)\leq 2$. 

If $d_G(v_3,v'_3)=1$, then let $w\notin \{v_2,v_4\}$ be the last neighbor of $v_3$, and we know that $v'_3\in \{w,v_2,v_4\}$. We cannot have $v'_3=v_2$ since $v_1v_2v'_2$ is a triangle. By symmetry, $v'_3\neq v_4$. Finally, if $v'_3=w$, then $v_3$ and $v'_3$ are adjacent $3$-vertices lying on the same $4$-cycle, which is impossible by \Cref{2dg4:config1}(ii).

So, we have $d(v_3,v'_3)=2$. In this case, observe that $|L(v_3)|\geq 5$ and $|L(v'_3)|\geq 5$. Recall that $|L(v_2)\cup L(v_3)\cup L(v'_2)\cup L(v'_3)|\leq 7$. As a result, $L(v_3)\cap L(v'_2)\neq \emptyset$ since $|L(v_3)|\geq 5$ and $|L(v'_2)|\geq 3$. The same holds for $L(v'_3)$ and $L(v_2)$. Suppose that $d_G(v_3,v'_2)\geq 3$, then we can color $v_3$ and $v'_2$ with the same color $c$. Since $L(v_3)\cap L(v'_4)=\emptyset$ or $L(v_4)\cap L(v'_2)=\emptyset$, w.l.o.g. we can color $v_4$ then $v'_4$. Finish by coloring $v_2$, $v'_3$, $v_1$ and $v_5$. So, $d_G(v_3,v'_2)\leq 2$. By symmetry, the same holds for $d_G(v'_3,v_2)$. By \Cref{2dg4:config1}, $d_G(v_3,v'_2)>1$. So, we get $d_G(v_3,v'_2)=d_G(v'_3,v_2)=2$. As a result, by planarity, $v_2$, $v'_2$, $v_3$, and $v'_3$ must have a common neighbor $u$. However, $uv_2v_3$ is a triangle. 
\end{itemize}

Due to the above claims, $G$ is $L$-list-colorable by \Cref{cor:Hall}.
\end{proof}

\subsection{Discharging procedure}\label{2dg4:discharging}

To prove that at least one of reducible configurations in $G$ is unavoidable in a planar graph, we apply the following rules in the discharging procedure:

\begin{itemize}
\item[\ru0] Every $5^+$-face $f$ gives $\frac{1}{3}$ to each incident $3$-vertex that is not incident to a $4$-face.
\item[\ru1] Every $5^+$-face $f$ gives $\frac{1}{2}$ to each incident $3$-vertex that is incident to a $4$-face.
\item[\ru2] Let $f'=u_1u_2u_3u_4u_5$ be a $5$-face where $d(u_1)=d(u_3)=d(u_4)=3$, $u_1u_2$ is incident to a $4$-face and let $f$ be incident to $u_4u_5$. If $f$ is a $5^+$-face, then $f$ gives $\frac{1}{6}$ to $f'$.
\end{itemize}

\begin{figure}[H]
\begin{minipage}[b]{0.24\textwidth}
\centering
\begin{tikzpicture}[scale=1.5]{thick}
\begin{scope}[every node/.style={circle,draw,minimum size=1pt,inner sep=2}]
    \node[draw=none] (f) at (1.54,0.5) {$f$};
    
    \node[draw=none] at (3.04,0) {$5^+$-face};
    \node[draw=none] at (3.04,1) {$5^+$-face};
    
    \node (v'2) at (1.95,1.31) {};
    \node[fill] (v'3) at (2.54,0.5) {};
    \node (v'4) at (1.95,-0.31) {};
    \node (3') at (3.54,0.5) {};
\end{scope}

\begin{scope}[every edge/.style={draw=black}]
    \path (v'3) edge (v'2);
    \path (v'3) edge (v'4);
    \path (v'3) edge (3');
    \path[->] (f) edge node[above] {$\frac13$} (v'3);
\end{scope}
\end{tikzpicture}
\caption{\ru0.}
\label{fig:2dg4:ru0}
\end{minipage}
\begin{minipage}[b]{0.24\textwidth}
\centering
\begin{tikzpicture}[scale=1.5]{thick}
\begin{scope}[every node/.style={circle,draw,minimum size=1pt,inner sep=2}]
    \node[draw=none] (f) at (1.54,0.5) {$f$};
    
    \node[draw=none] at (3.04,0) {$5^+$-face};
    
    \node (v'2) at (1.95,1.31) {};
    \node[fill] (v'3) at (2.54,0.5) {};
    \node (v'4) at (1.95,-0.31) {};

    \node (2') at (2.95,1.31) {};
    \node (3') at (3.54,0.5) {};
\end{scope}

\begin{scope}[every edge/.style={draw=black}]
    \path (v'3) edge (v'2);
    \path (v'3) edge (v'4);
    \path (2') edge (3');
    \path (v'2) edge (2');
    \path (v'3) edge (3');
    \path[->] (f) edge node[above] {$\frac12$} (v'3);
\end{scope}
\end{tikzpicture}
\caption{\ru1.}
\label{fig:2dg4:ru1}
\end{minipage}
\begin{minipage}[b]{0.49\textwidth}
\centering
\begin{tikzpicture}[scale=1.5]{thick}
\begin{scope}[every node/.style={circle,draw,minimum size=1pt,inner sep=2}]
    \node (v2) at (0.05,1.31) {};
    \node (v4) at (0.05,-0.31) {};
    \node[fill,label={below:$u_4$}] (v5) at (1,0) {};
    \node[label={above:$u_5$}] (v1) at (1,1) {};
    \node[fill,label={above:$u_1$}] (v'2) at (1.95,1.31) {};
    \node[label={above:$u_2$}] (v'3) at (2.54,0.5) {};
    \node[fill,label={below:$u_3$}] (v'4) at (1.95,-0.31) {};
 
    \node (2') at (2.95,1.31) {};
    \node (3') at (3.54,0.5) {};
    \node (4') at (2.45,-0.31) {};
    
    \node[draw=none] (f') at (1.69,0.5) {$f'$};
    \node[draw=none] (f) at (0.31,0.5) {$f$};
\end{scope}

\begin{scope}[every edge/.style={draw=black}]
    \path (v1) edge (v2);
    \path (v5) edge (v1);
    \path (v4) edge (v5);
    \path (v5) edge (v'4);
    \path (v1) edge (v'2);
    \path (v'3) edge (v'2);
    \path (v'3) edge (v'4);
    \path (2') edge (3');
    \path (v'2) edge (2');
    \path (v'3) edge (3');
    \path (v'4) edge (4');
    \path[->] (f) edge node[above left] {$\frac16$} (f');
\end{scope}
\end{tikzpicture}
\caption{\ru2.}
\label{fig:2dg4:ru2}
\end{minipage}
\end{figure}

We are now ready to prove \Cref{thm:2dg4} using discharging procedure together with the structural properties of $G$ proven in \Cref{2dg4:properties} and the discharging rules stated above.

\begin{proof}[Proof of \Cref{thm:2dg4}]
Let $G$ be a minimal counterexample to the theorem. Let $\mu(u)$ be the initial charge assignment for the vertices and faces of $G$ with the charge $\mu(u)=d(u)-4$ for each vertex $u\in V(G)$, and $\mu(f)=d(f)-4$ for each face $f\in F(G)$. By \Cref{equation}, we have that the total sum of the charges is negative.

Let $\mu^*$ be the charge assignment after the discharging procedure. In what follows, we prove that: $$\forall x \in V(G)\cup F(G), \mu^*(x)\geq 0.$$

First, we prove that the final charge on each vertex is non-negative. Let $u$ be a vertex in $V(G)$. Recall that $\Delta(G)=4$ and $\mu(u)=d(u)-4$. By \Cref{2dg4:minimumDegree}, vertex $u$ has degree at least $3$, thus we consider the following two cases.

\textbf{Case 1:} $d(u)=4$\\ 
By \ru0-\ru2, $u$ does not give any charge. So,
$$\mu^*(u)=\mu(u)=d(u)-4=0.$$

\textbf{Case 2:} $d(u)=3$\\
Recall that $\mu(u)=d(u)-4=-1$ and we have the following cases:
\begin{itemize}
\item If $u$ is not incident to any $4$-face, then it receives $\frac{1}{3}$ from each of the three incident $5^+$-faces (since $g(G)\geq 4$) by \ru0. So,
$$ \mu^*(u) = -1 + 3\cdot\frac13 = 0.$$
\item If $u$ is incident to a $4$-face, then it has exactly one incident $4$-face due to \Cref{2dg4:config1}(i). Therefore, $u$ is incident to two $5^+$-faces and receives $\frac12$ from each by \ru1. So,
$$ \mu^*(u) = -1 + 2\cdot\frac12 = 0.$$
\end{itemize}

Secondly, we prove that the final charge on each face is non-negative. Let $f$ be a face in $F(G)$ and let $i_0$, $i_1$, and $i_2$ be respectively the number of times $f$ gives charge by \ru0, \ru1, and \ru2. Recall that $\mu(f)=d(f)-4$. Moreover, $d(f)\geq 4$ since $g(G)\geq 4$, thus we consider the following three cases.

\textbf{Case 1:} $d(f)\geq 6$\\
By \Cref{2dg4:config1}(iii), there are at most $\frac23 d(f)$ $3$-vertices incident to $f$. As a result, we get $i_0+i_1\leq \frac{2}{3}d(f)$. Moreover, \ru2 can only be applied whenever \ru0 is applied since $u_4$ (from the statement of \ru2) is a $3$-vertex that is not incident to a $4$-face by \Cref{2dg4:config1}(ii). Additionally, by \Cref{2dg4:config2ii}, for each application of \ru0, \ru2 is applied at most once. In other words, $i_2\leq i_0$. Finally,
\begin{align*}
\mu^*(f)& \geq \mu(f)-\frac{1}{3}i_0-\frac{1}{2}i_1-\frac{1}{6}i_2\\
& \geq \mu(f)-\frac{1}{3}i_0-\frac{1}{2}i_1-\frac{1}{6}i_0\\
& \geq \mu(f)-\frac{1}{2}(i_0+i_1)\\
& \geq \mu(f)-\frac{1}{3}d(f)\\
& \geq \frac23 d(f) - 4\\
& \geq 0
\end{align*}
since $d(f)\geq 6$.

\textbf{Case 2:} $d(f)=5$\\
Recall that $\mu(f)=d(f)-4=1$. Observe that we have the following inequalities:
\begin{itemize}
\item $i_0+i_1\leq 3$ and $i_2\leq i_0$ (as in the previous case).
\item $i_1\leq 2$. Indeed, by \Cref{2dg4:config1}(ii), the $3$-vertex in \ru1 must be adjacent to only $4$-vertices. As a result, if we have equality, then $f$ is incident to exactly two $3$-vertices.
\item $i_2\leq 2$. Indeed, by \Cref{2dg4:config1}(iii), the neighbors of $u_4$ (as in the statement of \ru2) that are incident to $f$ must be $4$-vertices. As a result, if we have equality, then $f$ is incident to exactly two $3$-vertices.
\end{itemize}

Recall that $f$ gives $\frac13 i_0 + \frac12 i_1 + \frac16 i_2$ by \ru0, \ru1, and \ru2.\\
If $i_0\leq 1$, $i_1\leq 1$, and $i_2\leq 1$, then
$$ \mu^*(f)\geq 1 - \frac13- \frac12- \frac16=0.$$ 

If $i_2\geq 2$, then $i_2=2$ since $i_2\leq 2$ and $f$ is incident to exactly two $3$-vertices. Thus, $i_0+i_1\leq 2$. Moreover, since $i_2\leq i_0$, we also have $i_0=2$ and $i_1=0$. So,
$$ \mu^*(f)\geq 1 - 2\cdot\frac13- 2\cdot\frac16=0.$$

If $i_1\geq 2$, then $i_1=2$ since $i_1\leq 2$ and $f$ is incident to exactly two $3$-vertices. Thus, $i_0+i_1\leq 2$. Therefore, $i_0=0$. Moreover, since $i_2\leq i_0$, we also have $i_2=0$. So,
$$ \mu^*(f)\geq 1 - 2\cdot\frac12=0.$$

If $i_0\geq 3$, then $i_0=3$ and $i_1=0$ since $i_0+i_1\leq 3$. Let $f=v_1v_2v_3v_4v_5$. W.l.o.g. we have $d(v_1)=d(v_2)=d(v_4)=3$ and $d(v_3)=d(v_5)=4$. Observe that $f$ cannot give charge through $v_5v_1$, $v_1v_2$, or $v_2v_3$ by \ru2 due to \Cref{2dg4:config1}(i). If it gives through $v_4v_5$, then we have \Cref{2dg4:config2i}, which is a contradiction. The same holds for $v_4v_3$ by symmetry. As a result, $i_2=0$. So,
$$ \mu^*(f)\geq 1 - 3\cdot\frac13=0.$$

If $i_0=2$, then we have the following two cases since $i_0+i_1\leq 3$.
\begin{itemize}
\item If $i_1=0$, then we already know that $i_2\leq i_0 \leq 2$. So,
$$\mu^*(f)\geq 1 - 2\cdot\frac13 - 2\cdot\frac16=0. $$
\item If $i_1=1$, then let $f=u_1u_2u_3u_4u_5$. W.l.o.g. we have $d(u_1)=d(u_3)=d(u_4)=3$ and $d(u_2)=d(u_5)=4$. For $1\leq i\leq 5$, let $f_i\neq f$ be the face incident to $u_{i}u_{i+1(\text{mod }5)}$. Observe that neither $u_3$ nor $u_4$ can be incident to a $4$-face due to \Cref{2dg4:config1}(ii) so $f_2$, $f_3$, and $f_4$ must be $5^+$-faces. Moreover, we can assume w.l.o.g. that $f_1$ is a $4$-face. In such a case, $f$ receives $\frac16$ from $f_4$ by \ru2. Additionally, $f$ cannot give charge by \ru2 to neither $f_1$ nor $f_5$ due to \Cref{2dg4:config1}(ii). Furthermore, $f$ cannot gives to $f_2$, $f_3$, nor $f_4$ by \ru2 due to \Cref{2dg4:config1}(iii). As a result, $i_2=0$. So,
$$ \mu^*(f)\geq 1 +\frac16 - 2\cdot\frac13 - \frac12=0.$$
\end{itemize}
It follows that after the discharging procedure $5$-faces have non-negative charge.

\textbf{Case 3:} $d(f)=4$\\
Recall that $\mu(f)=d(f)-4=0$. Since $f$ does not give any charge, we have
$$ \mu^*(f)=\mu(f)=0.$$

To conclude the proof, after the discharging procedure, which preserved the total sum, we end up with a non-negative total sum, a contradiction to \Cref{equation}.
\end{proof}

\section{Injective list-coloring of planar graphs}\label{injg3}

In this section, we provide the proof to \Cref{thm:injg3}. Let $G$ be a minimal counterexample to \Cref{thm:injg3} with the fewest number of vertices plus edges. More precisely, $G$ has maximum degree $4$ and $\chi^i_\ell(G)\geq 12$.

\subsection{Structural properties of $G$\label{injg3:structure}}

We follow the same ideas as in the previous proof, starting by bounding the minimum degree of $G$ and proving that ``smaller'' objects are far away from each other. Recall that for injective coloring, two vertices see each other only when they share a neighbor.

\begin{lemma}\label{injg3:minimumDegree}
The minimum degree of $G$ is at least 3.
\end{lemma}

\begin{proof}
Let $u$ be a $2^-$-vertex in $G$. If $u$ is a $1$-vertex, then we color $G-\{u\}$, and we extend the coloring to $u$ by using one of the at least $8$ remaining colors. On the other hand, if $u$ is a $2$-vertex and $v$ one of its neighbors (see \Cref{fig:injg3:minimumDegree}), then we color $G-\{u\}$ and uncolor $v$. Observe that $u$ has at least five and $v$ has at least one available color. Thus, we first color $v$ and then $u$ to complete the coloring. 
\end{proof}

\begin{figure}[H]
\centering
\begin{tikzpicture}[scale=0.6]{thick}
\begin{scope}[every node/.style={circle,draw,minimum size=1pt,inner sep=2}]
    \node[label={above:$v$},label={below:$1$}] (1) at (0,0) {};
    \node[fill,label={above:$u$},label={below:$5$}] (2) at (2,0) {};
    \node (3) at (4,0) {};
\end{scope}

\begin{scope}[every edge/.style={draw=black}]
    \path (1) edge (3);
\end{scope}
\end{tikzpicture}
\caption{Reducible configuration in \Cref{injg3:minimumDegree}.}
\label{fig:injg3:minimumDegree}
\end{figure}
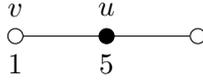

\begin{lemma}\label{injg3:config1}
Graph $G$ does not contain the following configurations:
\begin{itemize}
\item[(i)] Two adjacent $3$-vertices.
\item[(ii)] A $3$-vertex incident to a $4^-$-cycle.
\item[(iii)] A $3$-vertex at distance 1 from two adjacent $3$-cycles. 
\item[(iv)] A $4$-vertex incident to two adjacent $3$-cycles and another $4^-$-cycle.
\end{itemize}
\end{lemma}

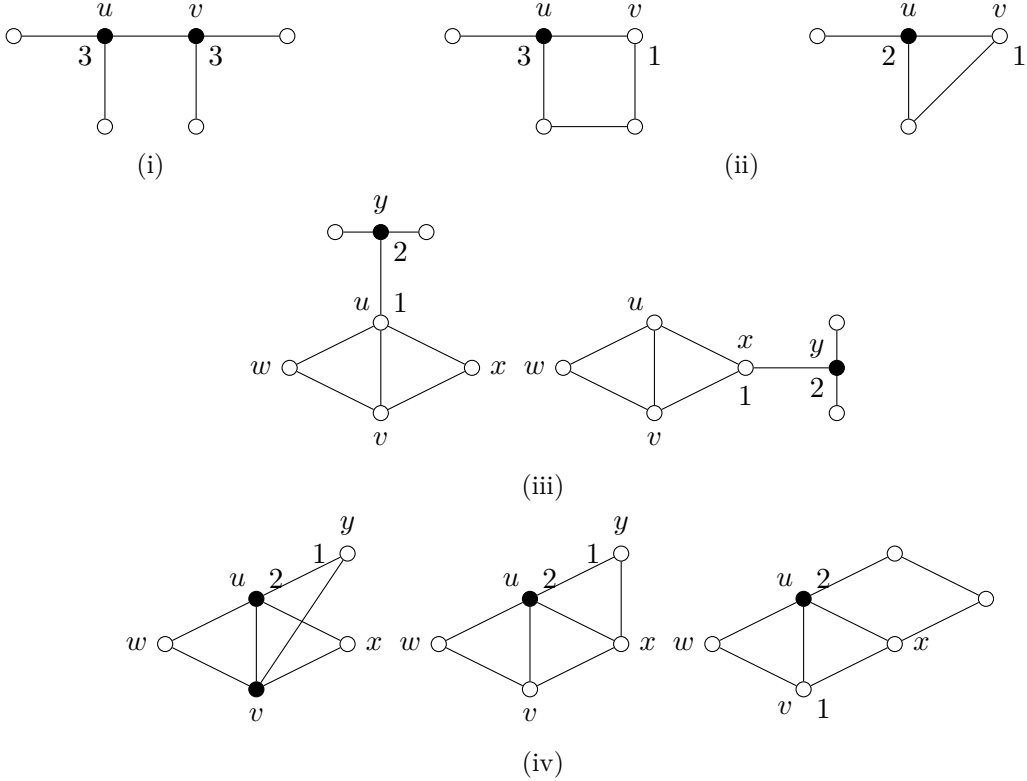
\begin{figure}[!htb]
\centering
\begin{subfigure}[b]{0.3\textwidth}
\centering
\begin{tikzpicture}[scale=0.6]{thick}
\begin{scope}[every node/.style={circle,draw,minimum size=1pt,inner sep=2}]
    \node (1) at (0,0) {};
    \node[fill,label={above:$u$},label={below left:$3$}] (2) at (2,0) {};
    \node (20) at (2,-2) {};
    \node[fill,label={above:$v$},label={below right:$3$}] (3) at (4,0) {};
    \node (30) at (4,-2) {};
    \node (4) at (6,0) {};
\end{scope}

\begin{scope}[every edge/.style={draw=black}]
    \path (1) edge (4);
    \path (2) edge (20);
    \path (3) edge (30);
\end{scope}
\end{tikzpicture}
\caption{}
\end{subfigure}
\begin{subfigure}[b]{0.6\textwidth}
\centering
\begin{tikzpicture}[scale=0.6]{thick}
\begin{scope}[every node/.style={circle,draw,minimum size=1pt,inner sep=2}]
    \node (1) at (0,0) {};
    \node[fill,label={above:$u$},label={below left:$3$}] (2) at (2,0) {};
    \node (20) at (2,-2) {};
    \node[label={above:$v$},label={below right:$1$}] (3) at (4,0) {};
    \node (30) at (4,-2) {};
    
    \node (1') at (8,0) {};
    \node[fill,label={above:$u$},label={below left:$2$}] (2') at (10,0) {};
    \node (20') at (10,-2) {};
    \node[label={above:$v$},label={below right:$1$}] (3') at (12,0) {};
\end{scope}

\begin{scope}[every edge/.style={draw=black}]
    \path (1) edge (3);
    \path (2) edge (20);
    \path (3) edge (30);
    \path (20) edge (30);
    
    \path (1') edge (3');
    \path (2') edge (20');
    \path (3') edge (20');
\end{scope}
\end{tikzpicture}
\caption{}
\end{subfigure}
\begin{subfigure}[b]{0.49\textwidth}
\centering
\begin{tikzpicture}[scale=0.6]{thick}
\begin{scope}[every node/.style={circle,draw,minimum size=1pt,inner sep=2}]
    \node[label={left: $w$}] (1) at (0,-1) {};
    \node[label={above left: $u$},label={above right: $1$}] (2) at (2,0) {};
    \node[label={below:$v$}] (3) at (2,-2) {};
    \node[label={right:$x$}] (4) at (4,-1) {};
    \node[fill,label={above: $y$},label={below right:$2$}] (20) at (2,2) {};
    \node (10) at (1,2) {};
    \node (30) at (3,2) {};
    
    \node[label={left: $w$}] (1') at (6,-1) {};
    \node[label={above left: $u$}] (2') at (8,0) {};
    \node[label={below:$v$}] (3') at (8,-2) {};
    \node[label={above:$x$},label={below: $1$}] (4') at (10,-1) {};
    \node[fill,label={above left: $y$},label={below left:$2$}] (20') at (12,-1) {};
    \node (10') at (12,0) {};
    \node (30') at (12,-2) {};
\end{scope}

\begin{scope}[every edge/.style={draw=black}]
    \path (1) edge (2);
    \path (2) edge (3);
    \path (3) edge (1);
    \path (2) edge (4);
    \path (3) edge (4);
    \path (2) edge (20);
    \path (10) edge (30);
    
    \path (1') edge (2');
    \path (2') edge (3');
    \path (3') edge (1');
    \path (2') edge (4');
    \path (3') edge (4');
    \path (4') edge (20');
    \path (10') edge (30');
\end{scope}
\end{tikzpicture}
\caption{}
\end{subfigure}

\begin{subfigure}[b]{0.66\textwidth}
\centering
\begin{tikzpicture}[scale=0.6]{thick}
\begin{scope}[every node/.style={circle,draw,minimum size=1pt,inner sep=2}]
    \node[label={left: $w$}] (1) at (-6,-1) {};
    \node[fill,label={above left: $u$},label={above right: $2$}] (2) at (-4,0) {};
    \node[fill,label={below:$v$}] (3) at (-4,-2) {};
    \node[label={right:$x$}] (4) at (-2,-1) {};
    \node[label={above:$y$},label={left:$1$}] (20) at (-2,1) {};
    
    \node[label={left: $w$}] (1') at (0,-1) {};
    \node[fill,label={above left: $u$},label={above right: $2$}] (2') at (2,0) {};
    \node[label={below:$v$}] (3') at (2,-2) {};
    \node[label={right:$x$}] (4') at (4,-1) {};
    \node[label={above:$y$},label={left:$1$}] (20') at (4,1) {};
    
    \node[label={left: $w$}] (1'') at (6,-1) {};
    \node[fill,label={above left: $u$},label={above right: $2$}] (2'') at (8,0) {};
    \node[label={below left:$v$},label={below right:$1$}] (3'') at (8,-2) {};
    \node[label={right:$x$}] (4'') at (10,-1) {};
    \node (20'') at (10,1) {};
    \node (30'') at (12,0) {};
\end{scope}

\begin{scope}[every edge/.style={draw=black}]
    \path (1) edge (2);
    \path (2) edge (3);
    \path (3) edge (1);
    \path (2) edge (4);
    \path (3) edge (4);
    \path (2) edge (20);
    \path (3) edge (20);
    
    \path (1') edge (2');
    \path (2') edge (3');
    \path (3') edge (1');
    \path (2') edge (4');
    \path (3') edge (4');
    \path (2') edge (20');
    \path (4') edge (20');
    
    \path (1'') edge (2'');
    \path (2'') edge (3'');
    \path (3'') edge (1'');
    \path (2'') edge (4'');
    \path (3'') edge (4'');
    \path (2'') edge (20'');
    \path (20'') edge (30'');
    \path (4'') edge (30'');
\end{scope}
\end{tikzpicture}
\caption{}
\end{subfigure}
\caption{Reducible configurations in \Cref{injg3:config1}.}
\label{fig:injg3:config1}
\end{figure}

\begin{proof}
We give a proof for each configuration separately (see \Cref{fig:injg3:config1}).
\begin{itemize}
\item[(i)] Suppose by contradiction that there exist two adjacent $3$-vertices $u$ and $v$. Color $G-\{uv\}$ by minimality and uncolor the vertices $u$ and $v$. Observe that each of them can see at most $8$ colors. Thus, $u$ and $v$ are colorable.
\item[(ii)] Suppose by contradiction that there exists a $3$-vertex $u$ that is incident to a $4^-$-cycle $C$. Let $v$ be a vertex incident to $C$ and adjacent with $u$. Color $G-\{uv\}$ by minimality and uncolor the vertices $u$ and $v$. Observe that $|L(u)|\geq 2$ and $|L(v)|\geq 1$. Thus, $u$ and $v$ are colorable.
\item[(iii)] Let $uvw$ and $uvx$ be $3$-cycles where $x\neq w$. First, suppose by contradiction that there exists a $3$-vertex $y$ that is at distance $1$ from $G[\{u,v,w,x\}]$. If $y$ is adjacent to $u$, color $G-\{uy\}$ by minimality and uncolor the vertices $u$ and $y$. Observe that $|L(u)|\geq 1$ and $|L(y)|\geq 2$. Thus, $u$ and $y$ are colorable. Now, suppose that $y$ is adjacent to $x$ instead. In this case, color $G-\{xy\}$ by minimality and uncolor the vertices $x$ and $y$. Observe that $|L(x)|\geq 1$ and $|L(y)|\geq 2$. Thus, $x$ and $y$ are colorable.
\item[(iv)] Let $u$ be a $4$-vertex that is incident to two adjacent $3$-cycles $uvw$ and $uvx$ where $x\neq w$. First, suppose by contradiction that $u$ is incident to another $3$-cycle. If there exists $y$ such that $y\notin \{x,w\}$ and $uvy$ is a $3$-cycle, then color $G-\{uy\}$ by minimality and uncolor the vertices $u$ and $y$. Observe that $|L(u)|\geq 2$ and $|L(y)|\geq 1$. Thus, $u$ and $y$ are colorable. If $uxy$ is a $3$-cycle with $y\notin \{v,w\}$, then the same exact arguments also hold in this case. Now, suppose that $u$ is incident to a $4$-cycle. Color $G-\{uv\}$ by minimality and uncolor the vertices $u$ and $v$. Observe that $|L(u)|\geq 2$ and $|L(v)|\geq 1$. Thus, $u$ and $v$ are colorable.
\end{itemize}
This concludes the proof.
\end{proof}

\begin{lemma}\label{injg3:config2}
A $3$-cycle in $G$ is not adjacent with a $4$-cycle.
\end{lemma}

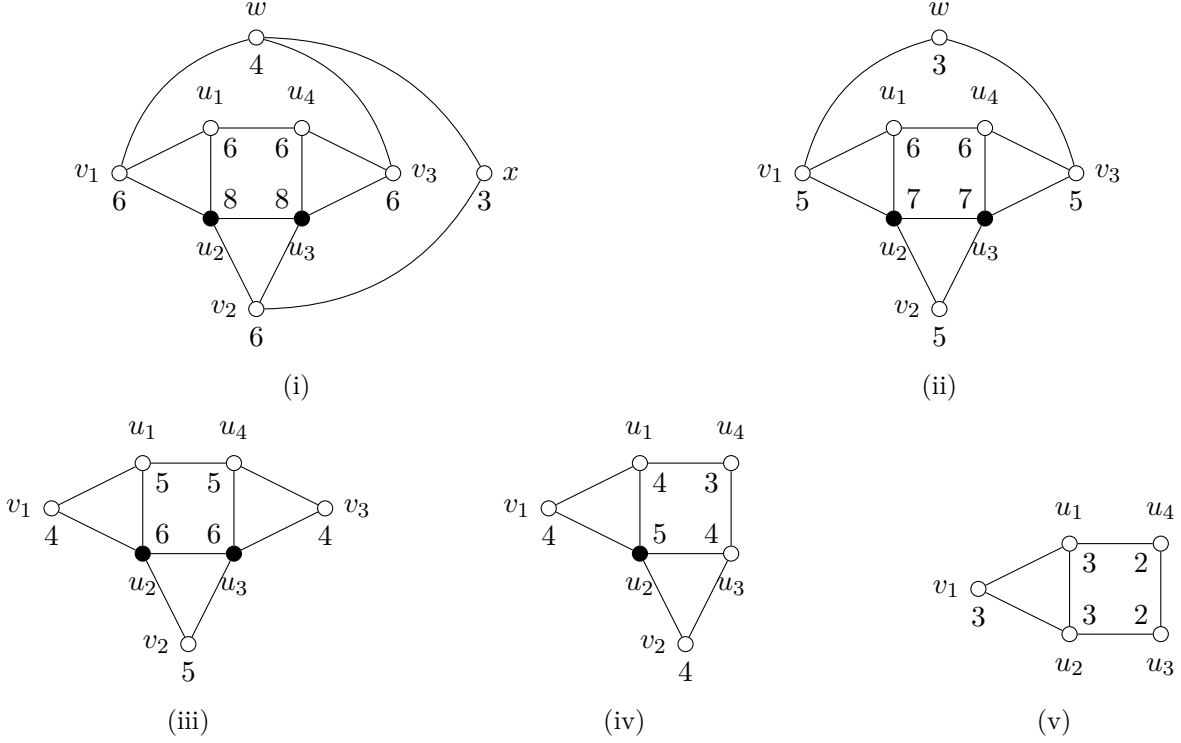
\begin{figure}[H]
\centering
\begin{subfigure}[b]{0.49\textwidth}
\centering
\begin{tikzpicture}[scale=0.6]{thick}
\begin{scope}[every node/.style={circle,draw,minimum size=1pt,inner sep=2}]
    \node[label={above:$u_1$},label={below right:$6$}] (u1) at (2,0) {};
    \node[fill,label={below:$u_2$},label={above right:$8$}] (u2) at (2,-2) {};
    \node[fill,label={below:$u_3$},label={above left:$8$}] (u3) at (4,-2) {};
    \node[label={above:$u_4$},label={below left:$6$}] (u4) at (4,0) {};
    
    \node[label={left:$v_1$},label={below:$6$}] (v1) at (0,-1) {};
    \node[label={left:$v_2$},label={below:$6$}] (v2) at (3,-4) {};
    \node[label={right:$v_3$},label={below:$6$}] (v3) at (6,-1) {};
    \node[label={above:$w$},label={below:$4$}] (w) at (3,2) {};
    \node[label={right:$x$},label={below:$3$}] (x) at (8,-1) {};
\end{scope}

\begin{scope}[every edge/.style={draw=black}]
    \path (u1) edge (u4);
    \path (u1) edge (u2);
    \path (u2) edge (u3);
    \path (u3) edge (u4);
    \path (u1) edge (v1);
    \path (v1) edge (u2);
    \path (v2) edge (u2);
    \path (v2) edge (u3);
    \path (v3) edge (u3);
    \path (v3) edge (u4);
    \path (v3) edge[bend right] (w);
    \path (v1) edge[bend left] (w);
    \path (v2) edge[bend right] (x);
    \path (w) edge[bend left] (x);
\end{scope}
\end{tikzpicture}
\caption{}
\end{subfigure}
\begin{subfigure}[b]{0.49\textwidth}
\centering
\begin{tikzpicture}[scale=0.6]{thick}
\begin{scope}[every node/.style={circle,draw,minimum size=1pt,inner sep=2}]
    \node[label={above:$u_1$},label={below right:$6$}] (u1) at (2,0) {};
    \node[fill,label={below:$u_2$},label={above right:$7$}] (u2) at (2,-2) {};
    \node[fill,label={below:$u_3$},label={above left:$7$}] (u3) at (4,-2) {};
    \node[label={above:$u_4$},label={below left:$6$}] (u4) at (4,0) {};
    
    \node[label={left:$v_1$},label={below:$5$}] (v1) at (0,-1) {};
    \node[label={left:$v_2$},label={below:$5$}] (v2) at (3,-4) {};
    \node[label={right:$v_3$},label={below:$5$}] (v3) at (6,-1) {};
    \node[label={above:$w$},label={below:$3$}] (w) at (3,2) {};
\end{scope}

\begin{scope}[every edge/.style={draw=black}]
    \path (u1) edge (u4);
    \path (u1) edge (u2);
    \path (u2) edge (u3);
    \path (u3) edge (u4);
    \path (u1) edge (v1);
    \path (v1) edge (u2);
    \path (v2) edge (u2);
    \path (v2) edge (u3);
    \path (v3) edge (u3);
    \path (v3) edge (u4);
    \path (v3) edge[bend right] (w);
    \path (v1) edge[bend left] (w);
\end{scope}
\end{tikzpicture}
\caption{}
\end{subfigure}

\begin{subfigure}[b]{0.33\textwidth}
\centering
\begin{tikzpicture}[scale=0.6]{thick}
\begin{scope}[every node/.style={circle,draw,minimum size=1pt,inner sep=2}]
    \node[label={above:$u_1$},label={below right:$5$}] (u1) at (2,0) {};
    \node[fill,label={below:$u_2$},label={above right:$6$}] (u2) at (2,-2) {};
    \node[fill,label={below:$u_3$},label={above left:$6$}] (u3) at (4,-2) {};
    \node[label={above:$u_4$},label={below left:$5$}] (u4) at (4,0) {};
    
    \node[label={left:$v_1$},label={below:$4$}] (v1) at (0,-1) {};
    \node[label={left:$v_2$},label={below:$5$}] (v2) at (3,-4) {};
    \node[label={right:$v_3$},label={below:$4$}] (v3) at (6,-1) {};
\end{scope}

\begin{scope}[every edge/.style={draw=black}]
    \path (u1) edge (u4);
    \path (u1) edge (u2);
    \path (u2) edge (u3);
    \path (u3) edge (u4);
    \path (u1) edge (v1);
    \path (v1) edge (u2);
    \path (v2) edge (u2);
    \path (v2) edge (u3);
    \path (v3) edge (u3);
    \path (v3) edge (u4);
\end{scope}
\end{tikzpicture}
\caption{}
\end{subfigure}
\begin{subfigure}[b]{0.33\textwidth}
\centering
\begin{tikzpicture}[scale=0.6]{thick}
\begin{scope}[every node/.style={circle,draw,minimum size=1pt,inner sep=2}]
    \node[label={above:$u_1$},label={below right:$4$}] (u1) at (2,0) {};
    \node[fill,label={below:$u_2$},label={above right:$5$}] (u2) at (2,-2) {};
    \node[label={below:$u_3$},label={above left:$4$}] (u3) at (4,-2) {};
    \node[label={above:$u_4$},label={below left:$3$}] (u4) at (4,0) {};
    
    \node[label={left:$v_1$},label={below:$4$}] (v1) at (0,-1) {};
    \node[label={left:$v_2$},label={below:$4$}] (v2) at (3,-4) {};
\end{scope}

\begin{scope}[every edge/.style={draw=black}]
    \path (u1) edge (u4);
    \path (u1) edge (u2);
    \path (u2) edge (u3);
    \path (u3) edge (u4);
    \path (u1) edge (v1);
    \path (v1) edge (u2);
    \path (v2) edge (u2);
    \path (v2) edge (u3);
\end{scope}
\end{tikzpicture}
\caption{}
\end{subfigure}
\begin{subfigure}[b]{0.32\textwidth}
\centering
\begin{tikzpicture}[scale=0.6]{thick}
\begin{scope}[every node/.style={circle,draw,minimum size=1pt,inner sep=2}]
    \node[label={above:$u_1$},label={below right:$3$}] (u1) at (2,0) {};
    \node[label={below:$u_2$},label={above right:$3$}] (u2) at (2,-2) {};
    \node[label={below:$u_3$},label={above left:$2$}] (u3) at (4,-2) {};
    \node[label={above:$u_4$},label={below left:$2$}] (u4) at (4,0) {};
    
    \node[label={left:$v_1$},label={below:$3$}] (v1) at (0,-1) {};
\end{scope}

\begin{scope}[every edge/.style={draw=black}]
    \path (u1) edge (u4);
    \path (u1) edge (u2);
    \path (u2) edge (u3);
    \path (u3) edge (u4);
    \path (u1) edge (v1);
    \path (v1) edge (u2);
\end{scope}
\end{tikzpicture}
\caption{}
\end{subfigure}
\caption{Reducible configurations in the proof of \Cref{injg3:config2}.}
\label{fig:injg3:config2}
\end{figure}

\begin{proof}
To prove \Cref{injg3:config2}, we will start by proving that the configurations in \Cref{fig:injg3:config2} are reducible and observe that \Cref{injg3:config2} corresponds to \Cref{fig:injg3:config2}(v). First, observe that the vertices in $S=\{u_1,u_2,u_3,u_4,v_1,v_2,v_3,w,x\}$ are pairwise distinct due to \Cref{injg3:config1}(iv).
\begin{itemize}
\item[(i)] Suppose first that $G$ contains a configuration shown in \Cref{fig:injg3:config2}(i). Color $G-\{u_2,u_3\}$ and uncolor the remaining vertices in $S$. The remaining list of colors for the non-colored vertices have size: $|L(u_1)|\geq 6$, $|L(u_2)|\geq 8$, $|L(u_3)|\geq 8$, $|L(u_4)|\geq 6$, $|L(v_1)|\geq 6$, $|L(v_2)|\geq 6$, $|L(v_3)|\geq 6$, $|L(w)|\geq 4$, and $|L(x)|\geq 3$. Observe that $u_1$ cannot be adjacent to $w$ as $u_1$, $u_2$, $u_3$, $u_4$, $v_1$, and $w$ would form the configuration in \Cref{injg3:config1}(iv). The same holds for $u_4$ and $w$. Moreover, $v_1$ and $v_3$ cannot both be adjacent to $x$ as $v_3$, $u_3$, $v_2$, $x$, $v_1$, and $w$ would form the configuration in \Cref{injg3:config1}(iv). Due to the previous observations and due to planarity, either $u_1$ or $u_4$ do not share a common neighbor with $x$, say $u_1$ does not. In this case, we show that $L(u_1)\cap L(x)=\emptyset$. Otherwise, we can color $u_1$ and $x$ with the same color and finish the coloring by coloring $w$, $v_1$, $v_2$, $v_3$, $u_4$, $u_2$, and $u_3$, in this order. Consequently, $|L(u_1)\cup L(x)|\geq 9$. But then, $G$ is $L$-list-colorable by \Cref{cor:Hall}.

\item[(ii)] Suppose first that $G$ contains a configuration shown in \Cref{fig:injg3:config2}(ii). Color $G-\{u_2,u_3\}$ and uncolor the other vertices. The remaining list of colors for these vertices have size: $|L(u_1)|\geq 6$, $|L(u_2)|\geq 7$, $|L(u_3)|\geq 7$, $|L(u_4)|\geq 6$, $|L(v_1)|\geq 5$, $|L(v_2)|\geq 5$, $|L(v_3)|\geq 5$, and $|L(w)|\geq 3$. Due to (i), $v_2$ and $w$ cannot share a neighbor. We claim that $L(v_2)\cap L(w)=\emptyset$. Otherwise, we can color $v_2$ and $w$ with the same color and finish the coloring in this order: $v_1$, $v_3$, $u_1$, $u_4$, $u_2$, and $u_3$. Consequently, $|L(v_2)\cup L(w)|\geq 8$. By \Cref{cor:Hall}, this configuration is reducible.

\item[(iii)] Suppose first that $G$ contains a configuration shown in \Cref{fig:injg3:config2}(iii). Color $G-\{u_2,u_3\}$ and uncolor the other vertices. The remaining list of colors for these vertices have size: $|L(u_1)|\geq 5$, $|L(u_2)|\geq 6$, $|L(u_3)|\geq 6$, $|L(u_4)|\geq 5$, $|L(v_1)|\geq 4$, $|L(v_2)|\geq 5$, and $|L(v_3)|\geq 4$. Due to (ii), $v_1$ and $v_3$ cannot share a neighbor. We claim that $L(v_1)\cap L(v_3)=\emptyset$. Otherwise, we can color $v_1$ and $v_3$ with the same color and finish the coloring in this order: $u_1$, $u_4$, $v_2$, $u_2$, and $u_3$. Consequently, $|L(v_1)\cup L(v_3)|\geq 8$. By \Cref{cor:Hall}, this configuration is reducible.

\item[(iv)] Suppose first that $G$ contains a configuration shown in \Cref{fig:injg3:config2}(iv). Color $G-\{u_2,u_3\}$ and uncolor the other vertices. The remaining list of colors for these vertices have size: $|L(u_1)|\geq 4$, $|L(u_2)|\geq 5$, $|L(u_3)|\geq 4$, $|L(u_4)|\geq 3$, $|L(v_1)|\geq 4$, and $|L(v_2)|\geq 4$. Due to (iii), $u_3$ and $u_4$ cannot share a neighbor. We claim that $L(u_3)\cap L(u_4)=\emptyset$. Otherwise, we can color $u_3$ and $u_4$ with the same color and finish the coloring in this order: $v_1$, $v_2$, $u_1$ and $u_2$. Consequently, $|L(u_3)\cup L(u_4)|\geq 7$. By \Cref{cor:Hall}, this configuration is reducible.

\item[(v)] Finally, suppose first that $G$ contains a configuration shown in \Cref{fig:injg3:config2}(v). Color $G-\{u_2,u_3\}$ and uncolor the other vertices. The remaining list of colors for these vertices have size: $|L(u_1)|\geq 2$, $|L(u_2)|\geq 3$, $|L(u_3)|\geq 3$, $|L(u_4)|\geq 2$, and $|L(v_2)|\geq 3$. Due to (iv), $u_1$ and $u_2$ cannot share a neighbor. The same holds for $u_3$ and $u_4$. We claim that $L(u_1)\cap L(u_2)=\emptyset$. Otherwise, we can color $u_1$ and $u_2$ with the same color and finish the coloring in this order: $u_4$, $v_2$, and $u_3$. Consequently, $|L(u_1)\cup L(u_2)|\geq 5$. By symmetry, the same holds for $L(u_3)$ and $L(u_4)$. Thus, a $3$-cycle adjacent to a $4$-cycle in $G$ is reducible by \Cref{cor:Hall}.
\end{itemize}
This concludes the proof.
\end{proof}

We now look at adjacent $3$-cycles.

\begin{lemma}\label{injg3:config3}
If $uvw$ and $uvx$ are two adjacent $3$-cycles in $G$ with $x\neq w$, then $x$ is not incident to another $3$-cycle $xyz$ with $y,z\notin \{u,v,w\}$. 
\end{lemma}

\begin{figure}[H]
\centering
\begin{subfigure}[b]{0.34\textwidth}
\centering
\begin{tikzpicture}[scale=0.6]{thick}
\begin{scope}[every node/.style={circle,draw,minimum size=1pt,inner sep=2}]
    \node[label={left: $w$}] (w) at (0,-1) {};
    \node[label={above left: $u$},label={above right: $1$}] (u) at (2,0) {};
    \node[label={below:$v$}] (v) at (2,-2) {};
    \node[fill,label={below:$x$},label={above:$2$}] (x) at (4,-1) {};
    \node[label={above: $y$}] (y) at (6,0) {};
    \node[label={below: $z$}] (z) at (6,-2) {};
    \node[label={above: $t$}] (t) at (8,-1) {};
\end{scope}

\begin{scope}[every edge/.style={draw=black}]
    \path (u) edge (v);
    \path (w) edge (v);
    \path (u) edge (w);
    \path (u) edge (x);
    \path (x) edge (v);
    \path (x) edge (y);
    \path (x) edge (z);
    \path (z) edge (y);
    \path (z) edge (t);
    \path (t) edge (y);
\end{scope}
\end{tikzpicture}
\caption{}
\end{subfigure}
\begin{subfigure}[b]{0.32\textwidth}
\centering
\begin{tikzpicture}[scale=0.6]{thick}
\begin{scope}[every node/.style={circle,draw,minimum size=1pt,inner sep=2}]
    \node[fill,label={above left: $w$},label={below left: $5$}] (w) at (0,-1) {};
    \node[label={above left: $u$},label={above right: $7$}] (u) at (2,0) {};
    \node[label={below left:$v$},label={below right: $7$}] (v) at (2,-2) {};
    \node[fill,label={below:$x$},label={above:$2$}] (x) at (4,-1) {};
    \node[label={above: $y$},label={right: $5$}] (y) at (6,0) {};
    \node[label={below: $z$},label={right: $5$}] (z) at (6,-2) {};
    \node[label={below: $s$},label={above: $4$}] (s) at (3,-4) {};
    \node[label={above: $t$},label={below: $4$}] (t) at (3,2) {};
\end{scope}

\begin{scope}[every edge/.style={draw=black}]
    \path (u) edge (v);
    \path (w) edge (v);
    \path (u) edge (w);
    \path (u) edge (x);
    \path (x) edge (v);
    \path (x) edge (y);
    \path (x) edge (z);
    \path (z) edge (y);
    \path (z) edge[bend left] (s);
    \path (s) edge[bend left] (w);
    \path (t) edge[bend left] (y);
    \path (t) edge[bend right] (w);
\end{scope}
\end{tikzpicture}
\caption{}
\end{subfigure}
\begin{subfigure}[b]{0.32\textwidth}
\centering
\begin{tikzpicture}[scale=0.6]{thick}
\begin{scope}[every node/.style={circle,draw,minimum size=1pt,inner sep=2}]
    \node[label={above left: $w$},label={below left: $3$}] (w) at (0,-1) {};
    \node[label={above left: $u$},label={above right: $5$}] (u) at (2,0) {};
    \node[label={below left:$v$},label={below right: $5$}] (v) at (2,-2) {};
    \node[fill,label={below:$x$},label={above:$5$}] (x) at (4,-1) {};
    \node[label={above: $y$},label={right: $3$}] (y) at (6,0) {};
    \node[label={below: $z$},label={right: $3$}] (z) at (6,-2) {};
\end{scope}

\begin{scope}[every edge/.style={draw=black}]
    \path (u) edge (v);
    \path (w) edge (v);
    \path (u) edge (w);
    \path (u) edge (x);
    \path (x) edge (v);
    \path (x) edge (y);
    \path (x) edge (z);
    \path (z) edge (y);
\end{scope}
\end{tikzpicture}
\caption{}
\end{subfigure}
\caption{Reducible configurations in the proof of \Cref{injg3:config3}.}
\label{fig:injg3:config3}
\end{figure}
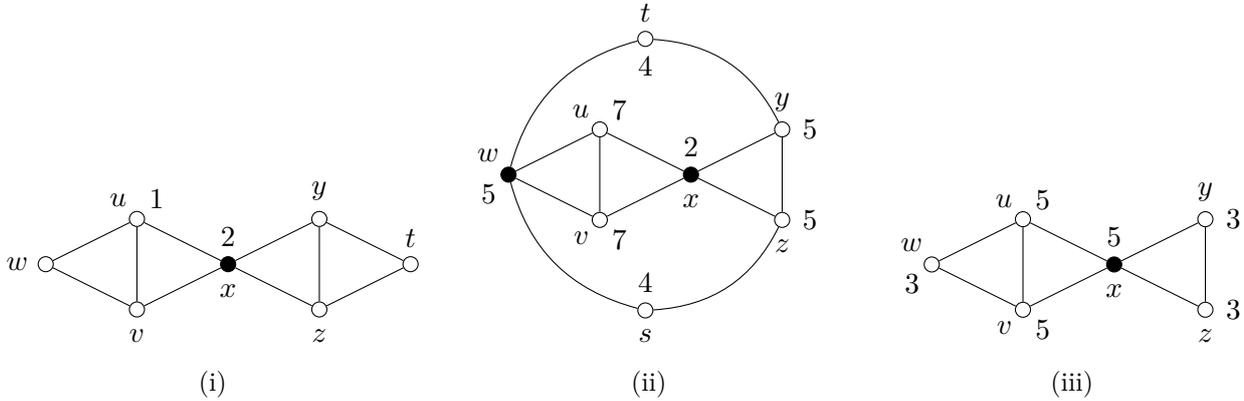

\begin{proof}
To prove \Cref{injg3:config3}, we will start by proving that the configurations in \Cref{fig:injg3:config3} are reducible and observe that \Cref{injg3:config3} corresponds to \Cref{fig:injg3:config3}(iii). First, observe that the vertices in $\{t,u,v,w,x,y,z\}$ are pairwise distinct due to \Cref{injg3:config1}(iv).
\begin{itemize}
\item[(i)] Color $G-\{ux\}$ and uncolor $u$ and $x$. Observe that $|L(u)|\geq 1$ and $|L(x)|\geq 2$. Thus, $u$ and $x$ are colorable.

\item[(ii)] First, observe that $s$ is distinct from all other vertices due to (i) and \Cref{injg3:config1}(iv). Moreover, $v$ cannot be adjacent to $s$ due to \Cref{injg3:config1}(iv). The same holds for $u$ and $s$. Also, $s$ cannot be adjacent to $t$ since $z$, $y$, $t$, $s$, and $x$ would form the configuration in \Cref{injg3:config2}. Now, color $G-\{x\}$ and uncolor the other vertices. The remaining list of colors for these vertices have size: $|L(u)|\geq 7$, $|L(v)|\geq 7$, $|L(w)|\geq 5$, $|L(x)|\geq 7$, $|L(y)|\geq 5$, $|L(z)|\geq 5$, $|L(s)|\geq 4$, and $|L(t)|\geq 4$. Due to the previous observations, $w$ and $s$ cannot share a neighbor. We claim that $L(w)\cap L(s)=\emptyset$. Otherwise, we can color $w$ and $s$ with the same color and finish the coloring in this order: $t$, $y$, $z$, $u$, $v$, and $x$. Consequently, $|L(w)\cup L(s)|\geq 9$. By \Cref{cor:Hall}, this configuration is reducible.

\item[(iii)] Color $G-\{x\}$ and uncolor the other vertices. The remaining list of colors for these vertices have size: $|L(u)|\geq 5$, $|L(v)|\geq 5$, $|L(w)|\geq 3$, $|L(x)|\geq 5$, $|L(y)|\geq 3$, and $|L(z)|\geq 3$. Due to (ii), either $w$ and $y$, or $w$ and $z$ cannot share a neighbor, say $w$ and $y$. We claim that $L(w)\cap L(y)=\emptyset$. Otherwise, we can color $w$ and $y$ with the same color and finish the coloring in this order: $z$, $u$, $v$, and $x$. Consequently, $|L(w)\cup L(y)|\geq 6$. By \Cref{cor:Hall}, this configuration is reducible.
\end{itemize}
This concludes the proof.
\end{proof}

Similar to $2$-distance coloring, the constraint at distance 2 in an injective coloring also naturally pushes us to look at configurations around $5$-faces.

\begin{lemma}\label{injg3:config9}
Let $f=u_1u_2u_3u_4u_5$ be a $5$-face in $G$ and let $u_4u_5v_4$ and $u_5v_4w$ be two distinct $3$-faces. Then, $f$ is adjacent to at most one other $3$-face.
\end{lemma}

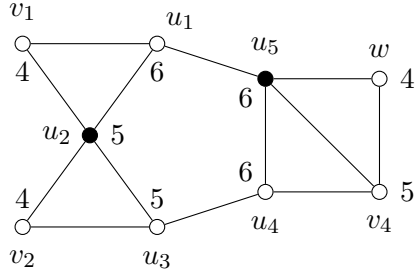
\begin{figure}[H]
\centering
\begin{tikzpicture}[scale=1.5]{thick}
\begin{scope}[every node/.style={circle,draw,minimum size=1pt,inner sep=2}]
    \node[label={above right:$u_1$},label={below:$6$}] (u1) at (0.05,1.31) {};
    \node[fill,label={left:$u_2$},label={right:$5$}] (u2) at (-0.54,0.5) {};
    \node[label={below:$u_3$},label={above:$5$}] (u3) at (0.05,-0.31) {};
    \node[label={below:$u_4$},label={above left:$6$}] (u4) at (1,0) {};
    \node[fill,label={above:$u_5$},label={below left:$6$}] (u5) at (1,1) {};
    
    \node[label={above:$v_1$},label={below:$4$}] (v1) at (-1.13,1.31) {};
    \node[label={below:$v_2$},label={above:$4$}] (v2) at (-1.13,-0.31) {};
    \node[label={below:$v_4$},label={right:$5$}] (v4) at (2,0) {};
    \node[label={above:$w$},label={right:$4$}] (w) at (2,1) {};
\end{scope}

\begin{scope}[every edge/.style={draw=black}]
    \path (u1) edge (u2);
    \path (u2) edge (u3);
    \path (u3) edge (u4);
    \path (u4) edge (u5);
    \path (u5) edge (u1);
    \path (u1) edge (v1);
    \path (u2) edge (v1);
    \path (v2) edge (u2);
    \path (u3) edge (v2);
    \path (u4) edge (v4);
    \path (u5) edge (v4);
    \path (u5) edge (w);
    \path (w) edge (v4);
\end{scope}
\end{tikzpicture}
\caption{Reducible configuration in \Cref{injg3:config9}.}
\label{fig:injg3:config9}
\end{figure}

\begin{proof}
Suppose by contradiction that $f$ is adjacent to at least two other $3$-faces. In such a case, the only possible configuration must be the one in \Cref{fig:injg3:config9} due to \Cref{injg3:config1}(iv) and \Cref{injg3:config3}. First, observe that the vertices in $S=\{u_1,u_2,u_3,u_4,u_5,v_1,v_2,v_4,w\}$ are pairwise distinct due to \Cref{injg3:config1}(ii, iv), \Cref{injg3:config2}, and \Cref{injg3:config3}. Color $G-\{u_1,u_2,u_3,u_4,u_5\}$ and uncolor $\{v_1,v_2,v_4,w\}$. The remaining list of colors for these vertices have size: $|L(u_1)|\geq 6$, $|L(u_2)|\geq 5$, $|L(u_3)|\geq 5$, $|L(u_4)|\geq 6$, $|L(u_5)|\geq 6$, $|L(v_1)|\geq 4$, $|L(v_2)|\geq 4$, $|L(v_4)|\geq 5$, and $|L(w)|\geq 4$. 

Observe that $v_4$ cannot be adjacent to $v_2$ nor $v_1$ due to \Cref{injg3:config2}. Thus, $u_2$ and $v_4$ cannot share a neighbor. If $L(u_2)\cap L(v_4)\neq \emptyset$, then we color $u_2$ and $v_4$ with the same color. Now, we color $w$, $v_1$, and $v_2$. Observe that each of the remaining list of colors for $u_1$, $u_3$, $u_4$, and $u_5$ has size at least 2 as they each see at most 4 different colors and if $u_3$ sees $w$ then $|L(u_3)|\geq 6$. Moreover, $u_1$ (resp. $u_3$) cannot see $u_5$ (resp. $u_4$) by \Cref{injg3:config1}(iv) (resp. \Cref{injg3:config3}). Since each of the remaining four vertices sees exactly two others, we can finish the coloring by the $2$-choosability of even cycles. Thus, we must have $L(u_2)\cap L(v_4)=\emptyset$.

We distinguish the following cases.

\textbf{Case 1:} $v_1$ shares a neighbor with $u_4$\\
Let $u$ be their common neighbor. We get $u\neq v_4$ (resp. $u_3$) by \Cref{injg3:config2}
(resp. \Cref{injg3:config1}(iv)). We also have $u\neq u_5$ since vertices of $S$ are all distinct. As a result, $v_2$ cannot share a neighbor with $u_5$ by planarity. Moreover, if $w$ sees $u_3$, then their common neighbor must be $u$ by planarity. However, $w$, $v_4$, $u_4$, $u$, and $u_5$ form the configuration in \Cref{injg3:config2}. Thus, $w$ cannot share a neighbor with $u_3$.

\begin{itemize}
\item If $L(u_3)\cap L(w)\neq \emptyset$, color $u_3$ and $w$ with the same color $c$. Let $L'(x)$ be the remaining list of colors for $x\in S\setminus\{u_3,w\}$.
\begin{itemize}
\item If $L'(u_5)\cap L'(v_2)\neq\emptyset$, then color $u_5$ and $v_2$ with the same color $c'$ and complete the coloring by \Cref{cor:Hall} as $L(u_2)\cap L(v_4) =\emptyset$ (as proven above). 
\item If $L'(u_5)\cap L'(v_2)=\emptyset$, then we complete the coloring of $S\setminus\{u_3,w\}$ by \Cref{cor:Hall}, which is possible since $|L'(u_2)\cup L'(v_4)|\geq 8$ and $|L'(u_5)\cup L'(v_2)|\geq 8$.
\end{itemize} 

\item If $L(u_3)\cap L(w)=\emptyset$, then $|L(u_3)\cup L(w)|\geq 9$. 
\begin{itemize}
\item If $L(u_5)\cap L(v_2)\neq\emptyset$, then color $u_5$ and $v_2$ with the same color $c$ and complete the coloring by \Cref{cor:Hall} as $L(u_2)\cap L(v_4) =\emptyset$ and $L(u_3)\cap L(w)=\emptyset$. 
\item If $L(u_5)\cap L(v_2)=\emptyset$, then we complete the coloring of $S$ by \Cref{cor:Hall}, which is possible since $|L(u_5)\cup L(v_2)|\geq 10$, $|L(u_3)\cup L(w)|\geq 9$, and $|L(u_2)\cup L(v_4)|\geq 10$.
\end{itemize}
\end{itemize}

\textbf{Case 2:} $v_2$ shares a neighbor with $u_5$.\\
Let $u$ be their common neighbor. We get $u\neq u_1$ (resp. $u_4$, $v_4$) by \Cref{injg3:config1}(iv)
(resp. \Cref{injg3:config3}, \Cref{injg3:config2}). So, $u=w$. Observe that in this case, $|L(u_1)|\geq 6$, $|L(u_2)|\geq 6$, $|L(u_3)|\geq 6$, $|L(u_4)|\geq 6$, $|L(u_5)|\geq 7$, $|L(v_1)|\geq 4$, $|L(v_2)|\geq 6$, $|L(v_4)|\geq 6$, and $|L(w)|\geq 6$.
\begin{itemize}
\item If $L(u_4)\cap L(v_1)\neq\emptyset$, then color $u_4$ and $v_1$ with the same color $c$ and complete the coloring by \Cref{cor:Hall} as $L(u_2)\cap L(v_4) =\emptyset$. 
\item If $L(u_4)\cap L(v_1)=\emptyset$, then we complete the coloring of $S$ by \Cref{cor:Hall}, which is possible since $|L(u_4)\cup L(v_1)|\geq 10$ and $|L(u_2)\cup L(v_4)|\geq 12$. 
\end{itemize}

\textbf{Case 3:} $v_1$ does not share a neighbor with $u_4$ and $v_2$ does not share a neighbor with $u_5$.
\begin{itemize}
\item If $L(u_4)\cap L(v_1)\neq\emptyset$, then color $u_4$ and $v_1$ with the same color $c$. Let $L'(x)$ be the remaining list of colors for $x\in S\setminus\{u_4,v_1\}$.
\begin{itemize}
\item If $L'(u_5)\cap L'(v_2)\neq\emptyset$, then color $u_5$ and $v_2$ with the same color $c'$ and complete the coloring by \Cref{cor:Hall} as $L(u_2)\cap L(v_4) =\emptyset$. 
\item If $L'(u_5)\cap L'(v_2)=\emptyset$, then we complete the coloring by \Cref{cor:Hall}, which is possible since $|L'(u_2)\cup L'(v_4)|\geq 8$ and $|L'(u_5)\cup L'(v_2)|\geq 8$.
\end{itemize} 

\item If $L(u_4)\cap L(v_1)=\emptyset$, then we complete the coloring of $S$ by \Cref{cor:Hall}, which is possible since $|L(u_4)\cup L(v_1)|\geq 10$ and $|L(u_2)\cup L(v_4)|\geq 12$. 
\begin{itemize}
\item If $L(u_5)\cap L(v_2)\neq\emptyset$, then color $u_5$ and $v_2$ with the same color $c$ and complete the coloring by \Cref{cor:Hall} as $L(u_2)\cap L(v_4) =\emptyset$ and $L(u_4)\cap L(v_1) =\emptyset$. 
\item If $L(u_5)\cap L(v_2)=\emptyset$, then we complete the coloring by \Cref{cor:Hall}, which is possible since $|L(u_2)\cup L(v_4)|\geq 10$, $|L(u_4)\cup L(v_1)|\geq 10$, and $|L(u_5)\cup L(v_2)|\geq 10$.
\end{itemize}
\end{itemize}
This concludes the proof.
\end{proof}

\begin{lemma}\label{injg3:config4}
A $5$-face in $G$ is not incident to a $3$-vertex and adjacent with three $3$-faces.
\end{lemma}

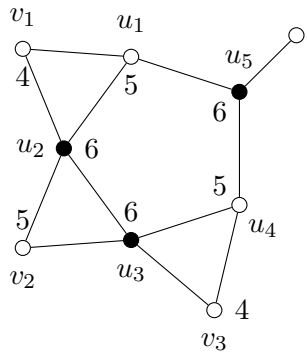
\begin{figure}[H]
\centering
\begin{tikzpicture}[scale=1.5]{thick}
\begin{scope}[every node/.style={circle,draw,minimum size=1pt,inner sep=2}]
    \node[label={above:$u_1$},label={below:$5$}] (u1) at (0.05,1.31) {};
    \node[fill,label={left:$u_2$},label={right:$6$}] (u2) at (-0.54,0.5) {};
    \node[fill,label={below:$u_3$},label={above:$6$}] (u3) at (0.05,-0.31) {};
    \node[label={below right:$u_4$},label={above left:$5$}] (u4) at (1,0) {};
    \node[fill,label={above :$u_5$},label={below left:$6$}] (u5) at (1,1) {};
    \node (5) at (1.5,1.5) {};
    
    \node[label={above:$v_1$},label={below:$4$}] (v1) at (-0.9,1.38) {};
    \node[label={below:$v_2$},label={above:$5$}] (v2) at (-0.9,-0.38) {};
    \node[label={below:$v_3$},label={right:$4$}] (v3) at (0.78,-0.93) {};
\end{scope}

\begin{scope}[every edge/.style={draw=black}]
    \path (u1) edge (u2);
    \path (u2) edge (u3);
    \path (u3) edge (u4);
    \path (u4) edge (u5);
    \path (u5) edge (u1);
    \path (u1) edge (v1);
    \path (u2) edge (v1);
    \path (u2) edge (v2);
    \path (u3) edge (v2);
    \path (u3) edge (v3);
    \path (u4) edge (v3);
    \path (u5) edge (5);
\end{scope}
\end{tikzpicture}
\caption{Reducible configuration in \Cref{injg3:config4}.}
\label{fig:injg3:config4}
\end{figure}

\begin{proof}
Suppose to the contrary that such a configuration exists in $G$. In which case, it must be the one in \Cref{fig:injg3:config4} due to \Cref{injg3:config1}(ii). First, observe that the vertices in $S=\{u_1,u_2,u_3,u_4,u_5,v_1,v_2,v_3\}$ are pairwise distinct due to \Cref{injg3:config1}(iv) and \Cref{injg3:config3}. 
Color $G-\{u_1,u_2,u_3,u_4,u_5\}$ and uncolor $\{v_1,v_2,v_3\}$. The remaning list of colors for these vertices have size: $|L(u_1)|\geq 5$, $|L(u_2)|\geq 6$, $|L(u_3)|\geq 6$, $|L(u_4)|\geq 5$, $|L(u_5)|\geq 6$, $|L(v_1)|\geq 4$, $|L(v_2)|\geq 5$, and $|L(v_3)|\geq 4$. Observe that $u_1$ and $u_5$ cannot share a neighbor due to \Cref{injg3:config1}(ii). The same holds for $u_4$ and $u_5$. 

Moreover, we claim that either $v_1$ and $u_4$, or $v_3$ and $u_1$ do not share a neighbor. Indeed, say $v_1$ and $u_4$ share a neighbor $w$, and $v_3$ and $u_1$ share a neighbor $w'$. Note that $w$ is distinct from $u_1$ by \Cref{injg3:config1}(ii), from $u_2$ and $u_3$ since vertices of $S$ are pairwise distinct, from $v_2$ by \Cref{injg3:config1}(iv), and from $v_3$ by \Cref{injg3:config2}. By symmetry and planarity, $w=w'$. However, this is impossible since $w$, $u_1$, $u_5$, $u_4$, and $v_1$ form the configuration in \Cref{injg3:config2}. 

Now, we can assume w.l.o.g. that $v_1$ and $u_4$ do not share a neighbor. In the following cases, we argue that $G$ is always colorable.
\begin{itemize}
\item If $L(v_1)\cap L(u_4)\neq \emptyset$, then color $v_1$ and $u_4$ with the same color $c$. Now, let $L'(x)$ be the remaining list of colors for $x\in S\setminus\{v_1,u_4\}$.
\begin{itemize}
\item If $L'(u_1)\cap L'(u_5)\neq \emptyset$, then color $u_1$ and $u_5$ with the same color (recall that they do not share a common neighbor) and finish by coloring $v_2$, $v_3$, $u_2$, and $u_3$ in this order.
\item If $L'(u_1)\cap L'(u_5)= \emptyset$, then $|L'(u_1)\cup L'(u_5)|\geq 9$. In this case, we conclude by coloring $S\setminus\{v_1,u_4\}$ due to \Cref{cor:Hall}.
\end{itemize}

\item If $L(v_1)\cap L(u_4)= \emptyset$, then $|L(v_1)\cup L(u_4)|\geq 9$. Now, we distinguish the two following cases.
\begin{itemize}
\item If $L(u_1)\cap L(u_5)\neq \emptyset$, then color $u_1$ and $u_5$ with the same color $c$. Now, let $L'(x)$ be the remaining list of colors for $x\in S\setminus\{u_1,u_5\}$. Since $|L(v_1)\cup L(u_4)|\geq 9$, we have $|L'(v_1)\cup L'(u_4)|\geq 7$. Thus, we conclude by coloring $S\setminus\{u_1,u_5\}$ due to \Cref{cor:Hall}.
\item If $L(u_1)\cap L(u_5)= \emptyset$, then $|L(u_1)\cup L(u_5)|\geq 11$. Since we also have $|L(v_1)\cup L(u_4)|\geq 9$, we can conclude by coloring $S$ due to \Cref{cor:Hall}.
\end{itemize}
\end{itemize}
This concludes the proof.
\end{proof}

\begin{lemma}\label{injg3:config5}
A $5$-face in $G$ is not adjacent with five $3$-faces.
\end{lemma}

\begin{figure}[H]
\centering
\begin{tikzpicture}[scale=1.5]{thick}
\begin{scope}[every node/.style={circle,draw,minimum size=1pt,inner sep=2}]
    \node[fill,label={above:$u_1$},label={below:$7$}] (u1) at (0.05,1.31) {};
    \node[fill,label={left:$u_2$},label={right:$7$}] (u2) at (-0.54,0.5) {};
    \node[fill,label={below:$u_3$},label={above:$7$}] (u3) at (0.05,-0.31) {};
    \node[fill,label={below right:$u_4$},label={above left:$7$}] (u4) at (1,0) {};
    \node[fill,label={above right:$u_5$},label={below left:$7$}] (u5) at (1,1) {};
    
    \node[label={above:$v_1$},label={below:$5$}] (v1) at (-0.9,1.38) {};
    \node[label={below:$v_2$},label={above:$5$}] (v2) at (-0.9,-0.38) {};
    \node[label={below:$v_3$},label={right:$5$}] (v3) at (0.78,-0.93) {};
    \node[label={right:$v_4$},label={below:$5$}] (v4) at (1.81,0.5) {};
    \node[label={above:$v_5$},label={right:$5$}] (v5) at (0.78,1.93) {};
\end{scope}

\begin{scope}[every edge/.style={draw=black}]
    \path (u1) edge (u2);
    \path (u2) edge (u3);
    \path (u3) edge (u4);
    \path (u4) edge (u5);
    \path (u5) edge (u1);
    \path (u1) edge (v1);
    \path (u2) edge (v1);
    \path (u2) edge (v2);
    \path (u3) edge (v2);
    \path (u3) edge (v3);
    \path (u4) edge (v3);
    \path (u4) edge (v4);
    \path (u5) edge (v4);
    \path (u1) edge (v5);
    \path (u5) edge (v5);
\end{scope}
\end{tikzpicture}
\caption{Reducible configuration in \Cref{injg3:config5}.}
\label{fig:injg3:config5}
\end{figure}
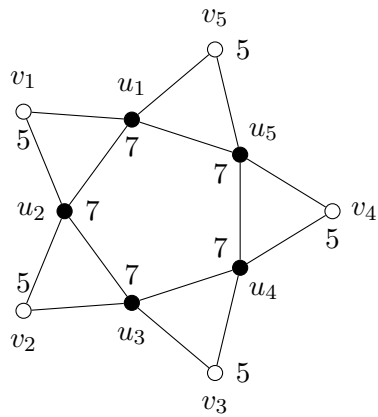

\begin{proof}
Suppose be contradiction that there exists a $5$-face $f=u_1u_2u_3u_4u_5$ that is adjacent to five $3$-faces like in \Cref{fig:injg3:config5}. First, observe that the vertices in $S=\{u_1,u_2,u_3,u_4,u_5,v_1,v_2,v_3,v_4,v_5\}$ are pairwise distinct due to \Cref{injg3:config1}(iv) and \Cref{injg3:config3}. 
Color $G-\{u_1,u_2,u_3,u_4,u_5\}$ and uncolor $\{v_1,v_2,v_3,v_4,v_5\}$. The remaning list of colors for these vertices have size: $|L(u_1)|\geq 7$, $|L(u_2)|\geq 7$, $|L(u_3)|\geq 7$, $|L(u_4)|\geq 7$, $|L(u_5)|\geq 7$, $|L(v_1)|\geq 5$, $|L(v_2)|\geq 5$, $|L(v_3)|\geq 5$, $|L(v_4)|\geq 5$, and $|L(v_5)|\geq 5$. Observe that, for $1\leq i\leq $, $v_i$ and $v_{i+2(\text{mod }5)}$ cannot be adjacent due to \Cref{injg3:config2}. By symmetry, the same holds for $v_i$ and $v_{i+3(\text{mod }5)}$. As a result, $v_i$ and $u_{i+3 (\text{mod }5)}$ cannot share a neighbor for $1\leq i\leq 5$. We have the following cases.
\begin{itemize}
\item If there exists $1\leq i\leq 5$ such that $L(v_i)\cap L(u_{i+3(\text{mod }5)})\neq \emptyset$, say $i=1$, then color $v_1$ and $u_4$ with the same color $c$. Now, let $L'(x)$ be the remaining list of colors for $x\in S\setminus\{v_1,u_4\}$.

\begin{itemize}
\item If there exists $2\leq i\leq 5$ such that $L'(v_i)\cap L'(u_{i+3(\text{mod }5)})\neq \emptyset$, say $i=2$, then color $v_2$ and $u_5$ with the same color $c'$ and let $L''(x)$ be the remaining list of colors for $x\in S\setminus\{v_1,v_2,u_4,u_5\}$. 
\begin{itemize}
\item If $L''(v_3)\cap L''(u_1)\neq \emptyset$, then color $v_3$ and $u_1$ with the same color $c''$ and finish by coloring $v_4$, $v_5$, $u_2$, and $u_3$.
\item If $L''(v_3)\cap L''(u_1) = \emptyset$, then $|L''(v_3)\cup L''(u_1)|\geq 8$. We complete the coloring of $S\setminus\{v_1,v_2,u_4,u_5\}$ by \Cref{cor:Hall}.
\end{itemize}

\item If, for every $2\leq i\leq 5$, $L'(v_i)\cap L'(u_{i+3(\text{mod }5)})= \emptyset$, then $|L'(v_i)\cup L'(u_{i+3(\text{mod }5)})|\geq 10$. In such a case, we can conclude by coloring $S\setminus\{v_1,u_4\}$ due to \Cref{cor:Hall}.
\end{itemize}

\item If, for every $1\leq i\leq 5$, $L(v_i)\cap L(u_{i+3(\text{mod }5)})= \emptyset$, then $|L(v_i)\cup L(u_{i+3(\text{mod }5)})|\geq 12$. In such a case, we can conclude by coloring $S$ due to \Cref{cor:Hall}.
\end{itemize}
It follows that a $5$-face in $G$ is adjacent with at most four $3$-faces, thus concluding the proof.
\end{proof}

For conciseness, we introduce a useful definition when considering $5$-faces.

\textbf{Bad faces:} We call a $5$-face \emph{bad} if it is adjacent to four $3$-faces (see \Cref{fig:injg3:bad}).

Observe that bad $5$-faces are adjacent to exactly four $3$-faces due to \Cref{injg3:config5}.

\begin{figure}[H]
\centering
\begin{tikzpicture}[scale=1.5,rotate=55]{thick}
\begin{scope}[every node/.style={circle,draw,minimum size=1pt,inner sep=2}]
    \node[fill] (u1) at (0.05,1.31) {};
    \node (u2) at (-0.54,0.5) {};
    \node (u3) at (0.05,-0.31) {};
    \node[fill] (u4) at (1,0) {};
    \node[fill] (u5) at (1,1) {};
    
    \node (v1) at (-0.9,1.38) {};
    \node (v3) at (0.78,-0.93) {};
    \node (v4) at (1.81,0.5) {};
    \node (v5) at (0.78,1.93) {};
\end{scope}

\begin{scope}[every edge/.style={draw=black}]
    \path (u1) edge (u2);
    \path (u2) edge (u3);
    \path (u3) edge (u4);
    \path (u4) edge (u5);
    \path (u5) edge (u1);
    \path (u1) edge (v1);
    \path (u2) edge (v1);
    \path (u3) edge (v3);
    \path (u4) edge (v3);
    \path (u4) edge (v4);
    \path (u5) edge (v4);
    \path (u1) edge (v5);
    \path (u5) edge (v5);
\end{scope}
\end{tikzpicture}
\caption{A bad $5$-face.}
\label{fig:injg3:bad}
\end{figure}
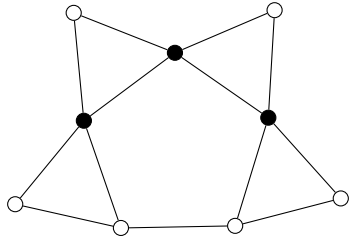

We now turn our attention to bad $5$-faces.

\begin{lemma}\label{injg3:config6}
A bad $5$-face in $G$ is not adjacent with a $4$-face.
\end{lemma}

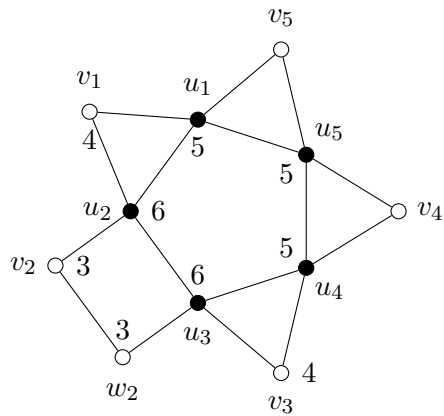
\begin{figure}[H]
\centering
\begin{tikzpicture}[scale=1.5]{thick}
\begin{scope}[every node/.style={circle,draw,minimum size=1pt,inner sep=2}]
    \node[fill,label={above:$u_1$},label={below:$5$}] (u1) at (0.05,1.31) {};
    \node[fill,label={left:$u_2$},label={right:$6$}] (u2) at (-0.54,0.5) {};
    \node[fill,label={below:$u_3$},label={above:$6$}] (u3) at (0.05,-0.31) {};
    \node[fill,label={below right:$u_4$},label={above left:$5$}] (u4) at (1,0) {};
    \node[fill,label={above right:$u_5$},label={below left:$5$}] (u5) at (1,1) {};
    
    \node[label={above:$v_1$},label={below:$4$}] (v1) at (-0.9,1.38) {};
    \node[label={left:$v_2$},label={right:$3$}] (v2) at (-1.2,0.02) {};
    \node[label={below:$w_2$},label={above:$3$}] (w2) at (-0.61,-0.79) {};
    \node[label={below:$v_3$},label={right:$4$}] (v3) at (0.78,-0.93) {};
    \node[label={right:$v_4$}] (v4) at (1.81,0.5) {};
    \node[label={above:$v_5$}] (v5) at (0.78,1.93) {};
\end{scope}

\begin{scope}[every edge/.style={draw=black}]
    \path (u1) edge (u2);
    \path (u2) edge (u3);
    \path (u3) edge (u4);
    \path (u4) edge (u5);
    \path (u5) edge (u1);
    \path (u1) edge (v1);
    \path (u2) edge (v1);
    \path (u2) edge (v2);
    \path (w2) edge (v2);
    \path (u3) edge (w2);
    \path (u3) edge (v3);
    \path (u4) edge (v3);
    \path (u4) edge (v4);
    \path (u5) edge (v4);
    \path (u1) edge (v5);
    \path (u5) edge (v5);
\end{scope}
\end{tikzpicture}
\caption{Reducible configuration in \Cref{injg3:config6}.}
\label{fig:injg3:config6}
\end{figure}

\begin{sloppypar}
\begin{proof}
Suppose by contradiction that there exists a $5$-face $f=u_1u_2u_3u_4u_5$ that is adjacent to four $3$-faces and one $4$-face like in \Cref{fig:injg3:config6}. First, observe that the vertices in $\{u_1,u_2,u_3,u_4,u_5,v_1,v_2,w_2,v_3,v_4,v_5\}$ are pairwise distinct due to \Cref{injg3:config1}(iv), \Cref{injg3:config2}, and \Cref{injg3:config3}. 
Color $G-\{u_1,u_2,u_3,u_4\}$ and uncolor $\{u_5,v_1,v_2,w_2,v_3\}$. Let $S=\{u_1,u_2,u_3,u_4,u_5,v_1,v_2,w_2,v_3\}$. The remaning list of colors for these vertices have size: $|L(u_1)|\geq 5$, $|L(u_2)|\geq 6$, $|L(u_3)|\geq 6$, $|L(u_4)|\geq 5$, $|L(u_5)|\geq 5$, $|L(v_1)|\geq 4$, $|L(v_2)|\geq 3$, $|L(w_2)|\geq 3$, and $|L(v_3)|\geq 4$. 

Observe that $v_1$ and $v_3$ (resp. $v_4$) cannot be adjacent due to \Cref{injg3:config2}. The same holds for $v_3$ and $v_5$. As a result, $v_1$ and $u_4$ (resp. $v_3$ and $u_1$) cannot share a neighbor. 

Similarly, $v_2$ (resp. $w_2$) cannot be adjacent to $v_5$ (resp. $v_4$) by \Cref{injg3:config2}. As a result, if $v_2$ (resp. $w_2$) share a neighbor with $u_5$, then $v_2$ must be adjacent to $v_4$ (resp. $v_5$). By planarity and by symmetry, we can assume w.l.o.g. that $u_5$ cannot share a neighbor with $v_2$. 

We distinguish the following cases.
\begin{itemize}
\item If $L(v_1)\cap L(u_4)\neq \emptyset$, then color $v_1$ and $u_4$ with the same color $c$. Now, let $L'(x)$ be the remaining list of colors for $x\in S\setminus\{v_1,u_4\}$.

\begin{itemize}
\item If $L'(v_3)\cap L'(u_1)\neq \emptyset$, then color $v_3$ and $u_1$ with the same color $c'$. Due to \Cref{injg3:config2}, $v_2$ and $w_2$ (resp. $u_2$ and $v_2$, $u_3$ and $w_2$) cannot share a neighbor. So, we can finish by coloring $v_2$, $w_2$, $u_5$, $u_2$ and $u_3$. 

\item If $L'(v_3)\cap L'(u_1)= \emptyset$, then $|L'(v_3)\cup L'(u_1)|\geq 7$. 
\begin{itemize}
\item If $L'(v_2)\cap L'(u_5)\neq \emptyset$, then color $v_2$ and $u_5$ with the same color $c'$ and finish by coloring $w_2$, $v_3$, $u_1$, $u_2$ and $u_3$.
\item If $L'(v_2)\cap L'(u_5) = \emptyset$, then $|L'(v_2)\cup L'(u_5)|\geq 6$. We complete the coloring of $S\setminus\{v_1,u_4\}$ by \Cref{cor:Hall}.\end{itemize}
\end{itemize}

\item If $L(v_1)\cap L(u_4)=\emptyset$, then by symmetry, $L(v_3)\cap L(u_1)=\emptyset$ due to the previous case. Thus, we have $|L(v_1)\cup L(u_4)|\geq 9$ and $|L(v_3)\cup L(u_1)|\geq 9$.
\begin{itemize}
\item If $L(v_2)\cap L(u_5)\neq \emptyset$, then color $v_2$ and $u_5$ with the same color $c$. Let $L'(x)$ be the remaining list of colors for $x\in S\setminus\{v_2,u_5\}$. Since $|L(v_1)\cup L(u_4)|\geq 9$ and $|L(v_3)\cup L(u_1)|\geq 9$, we have $|L(v_1)\cup L(u_4)|\geq 7$ and $|L(v_3)\cup L(u_1)|\geq 7$. Thus, we can the coloring of $S\setminus\{v_2,u_5\}$ by \Cref{cor:Hall}.
\item If $L(v_2)\cap L(u_5) = \emptyset$, then $|L(v_2)\cup L(u_5)|\geq 8$. We can color $S$ by \Cref{cor:Hall}.
\end{itemize}
\end{itemize}
This concludes the proof.
\end{proof}
\end{sloppypar}

\begin{lemma}\label{injg3:config7}
A bad $5$-face in $G$ is not adjacent with a $5$-face.
\end{lemma}

\begin{figure}[H]
\centering
\begin{tikzpicture}[scale=1.5]{thick}
\begin{scope}[every node/.style={circle,draw,minimum size=1pt,inner sep=2}]
    \node[fill,label={above:$u_1$},label={below:$7$}] (u1) at (1.95,1.31) {};
    \node[fill,label={above left:$u_2$},label={below right:$6$}] (u2) at (1,1) {};
    \node[fill,label={below left:$u_3$},label={above right:$6$}] (u3) at (1,0) {};
    \node[fill,label={below:$u_4$},label={above:$7$}] (u4) at (1.95,-0.31) {};
    \node[fill,label={right:$u_5$},label={left:$7$}] (u5) at (2.54,0.5) {};
    
    \node[label={above:$v_1$},label={below :$5$}] (v1) at (1.22,1.93) {};
    \node[label={below:$v_3$},label={above :$5$}] (v3) at (1.22,-0.93) {};
    \node[label={below:$v_4$},label={above:$5$}] (v4) at (2.9,-0.38) {};
    \node[label={above:$v_5$},label={below:$5$}] (v5) at (2.9,1.38) {};
    
    \node[label={above:$v_2$},label={below:$3$}] (v2) at (0.05,1.31) {};
    \node[label={below:$w_2$},label={above:$3$}] (w2) at (0.05,-0.31) {};
    \node (2) at (-0.54,0.5) {};
\end{scope}

\begin{scope}[every edge/.style={draw=black}]
    \path (u1) edge (u2);
    \path (u2) edge (u3);
    \path (u3) edge (u4);
    \path (u4) edge (u5);
    \path (u5) edge (u1);
    \path (u1) edge (v1);
    \path (u2) edge (v1);
    \path (u2) edge (v2);
    \path (v2) edge (2);
    \path (w2) edge (2);
    \path (u3) edge (w2);
    \path (u3) edge (v3);
    \path (u4) edge (v3);
    \path (u4) edge (v4);
    \path (u5) edge (v4);
    \path (u1) edge (v5);
    \path (u5) edge (v5);
\end{scope}
\end{tikzpicture}
\caption{Reducible configuration in \Cref{injg3:config7}.}
\label{fig:injg3:config7}
\end{figure}

\begin{proof}
Suppose be contradiction that there exists a bad $5$-face $f=u_1u_2u_3u_4u_5$ that is adjacent to a $5$-face as shown in \Cref{fig:injg3:config7}. First, observe that the vertices in $S=\{u_1,u_2,u_3,u_4,u_5,v_1,v_2,w_2,v_3,v_4,v_5\}$ are pairwise distinct due to \Cref{injg3:config1}(iv), \Cref{injg3:config2}, and \Cref{injg3:config3}. 
Color $G-\{u_1,u_2,u_3,u_4,u_5\}$ and uncolor $\{v_1,v_2,w_2,v_3,v_4,v_5\}$. The remaning list of colors for these vertices have size: $|L(u_1)|\geq 7$, $|L(u_2)|\geq 6$, $|L(u_3)|\geq 6$, $|L(u_4)|\geq 7$, $|L(u_5)|\geq 7$, $|L(v_1)|\geq 5$, $|L(v_2)|\geq 3$, $|L(w_2)|\geq 3$, $|L(v_3)|\geq 5$, $|L(v_4)|\geq 5$, and $|L(v_5)|\geq 5$. 
Observe that the following pairs of vertices cannot share a neighbor:
\begin{itemize}
\item $(u_1,v_3)$ and $(u_4,v_1)$. Otherwise, by symmetry, say $v_3$ sees $u_1$, then it would need to be adjacent to $v_1$ or $v_5$, which is impossible due to \Cref{injg3:config2}.
\item $(u_2,v_2)$ and $(u_3,w_2)$ due to \Cref{injg3:config3}.
\item $(u_2,u_3)$ since they neighbors are distinct vertices.
\end{itemize}

We distinguish the following cases:

\textbf{Case 1:} $L(u_2)\cap L(v_2)\neq\emptyset$.\\
Color $u_2$ and $v_2$ with the same color $c$. Now, let $L'(x)$ be the remaining list of colors for $x\in S\setminus\{u_2,v_2\}$.

\begin{itemize}
\item If $L'(w_2)\cap L'(u_3)\neq \emptyset$, then color $w_2$ and $u_3$ with the same color $c'$. Let $L''(x)$ be the remaining list of colors for $x\in S\setminus\{u_2,u_3,v_2,w_2\}$. We have the following:
$|L''(u_1)|\geq 5$, $|L''(u_4)|\geq 5$, $|L''(u_5)|\geq 5$, $|L''(v_1)|\geq 3$, and $|L''(v_3)|\geq 3$. Observe that $|L''(v_4)|\geq 4$ because if $v_4$ sees the color $c$, then $v_4$ sees either $v_2$ or $u_2$, in which case we had $|L(v_4)|\geq 6$. The same holds for $v_5$ and $c'$.
\begin{itemize}
\item If $L''(v_1)\cap L''(u_4)\neq \emptyset$ or $L''(v_3)\cap L''(u_1)\neq \emptyset$, say it holds for $L''(v_1)$ and $L''(u_4)$, then we color $v_1$ and $u_4$ with the same color $c''$ and finish by coloring $v_3$, $v_4$, $v_5$, $u_5$, and $u_1$.
\item If $L''(v_1)\cap L''(u_4)=\emptyset$ and $L''(v_3)\cap L''(u_1)=\emptyset$, then we complete the coloring by \Cref{cor:Hall}.
\end{itemize}
 
\item If $L'(w_2)\cap L'(u_3)=\emptyset$, then we have following cases.
\begin{itemize}
\item If $L'(v_1)\cap L'(u_4)\neq \emptyset$, then we color $v_1$ and $u_4$ with the same color $c'$. Let $L''(x)$ be the remaining list of colors for $x\in S\setminus\{u_2,u_4,v_2,v_1\}$. We have the following:
$|L''(u_1)|\geq 5$, $|L''(u_3)|\geq 4$, $|L''(u_5)|\geq 5$, $|L''(v_3)|\geq 3$, $|L''(v_5)|\geq 3$, and $|L''(w_2)|\geq 1$. Similar to the previous case, even if $v_4$ sees $c$, we get $|L''(v_4)|\geq 4$. Moreover, we have $|L''(u_3)\cup L''(w_2)|=|L''(u_3)|+|L''(w_2)|\geq 7$ since $L''(u_3)\cap L''(w_2)=\emptyset$.

\begin{itemize}
\item If $|L''(w_2)|=1$, then $|L''(u_3)|\geq 6$. Thus, we can coloring the remaining vertices in this order: $w_2$, $v_3$, $v_5$, $v_4$, $u_1$, $u_5$, and $u_3$.
\item If $|L''(w_2)|\geq 2$, then we can always color $w_2$ last. Now, we either complete the coloring by \Cref{cor:Hall}, or $L''(u_1)\cap L''(v_3)\neq \emptyset$, in which case, we color $u_1$ and $v_3$ with the same color and finish by coloring $v_5$, $v_4$, $u_3$, $u_5$, and $w_2$ in this order.
\end{itemize} 

\item If $L'(v_3)\cap L'(u_1)\neq\emptyset$, then we color $v_3$ and $u_1$ with the same color $c'$. Let $L''(x)$ be the remaining list of colors for $x\in S\setminus\{u_1,u_2,v_2,v_3\}$. We have the following:
$|L''(u_3)|\geq 4$, $|L''(u_4)|\geq 5$, $|L''(u_5)|\geq 5$, $|L''(v_1)|\geq 3$, $|L''(v_4)|\geq 4$, $|L''(v_5)|\geq 3$, and $|L''(w_2)|\geq 1$. Moreover, we have $|L''(u_3)\cup L''(w_2)|=|L''(u_3)+|L''(w_2)|\geq 7$ since $L''(u_3)\cap L''(w_2)=\emptyset$.
\begin{itemize}
\item If $|L''(w_2)|=1$, then $|L''(u_3)|\geq 6$. Thus, we can coloring the remaining vertices in this order: $w_2$, $v_1$, $v_5$, $v_4$, $u_4$, $u_5$, and $u_3$.
\item If $|L''(w_2)|\geq 2$, then we can always color $w_2$ last. Now, we either complete the coloring by \Cref{cor:Hall}, or $L''(u_4)\cap L''(v_1)\neq \emptyset$, in which case, we color $u_4$ and $v_1$ with the same color and finish by coloring $v_5$, $v_4$, $u_3$, $u_5$, and $w_2$ in this order.
\end{itemize}

\item If $L'(v_1)\cap L'(u_4)=\emptyset$ and $L'(v_3)\cap L'(u_1)=\emptyset$, then we have $|L'(v_1)\cup L'(u_4)|\geq 10$, $|L'(v_3)\cup L'(u_1)|\geq 10$, and $|L'(w_2)\cup L'(u_3)|\geq 8$. Thus, we can conclude by \Cref{cor:Hall}.
\end{itemize}
\end{itemize}

\textbf{Case 2:} $L(u_2)\cap L(v_2)=\emptyset$.\\
By symmetry, we also have $L(w_2)\cap L(u_3)=\emptyset$ or we would be back in \textbf{Case 1}.
\begin{itemize}
\item If $L(v_1)\cap L(u_4)\neq \emptyset$, then we color $v_1$ and $u_4$ with the same color $c$. Now, color $v_2$, $w_2$, $v_3$, and $v_4$ in this order. Note that $v_5$ has at least one color left, otherwise $|L(v_5)|=5$ and it would need to see every colored vertex in $S$, in which case $|L(v_5)|\geq 6$. So, we can also color $v_5$. Let $L'(x)$ be the remaining list of colors for $x\in \{u_1,u_2,u_2,u_5\}$. By the same arguments along with the fact that $L(v_2)\cap L(u_2)=\emptyset$ and $L(w_2)\cap L(u_3)=\emptyset$, we must have $|L'(u_1)|\geq 3$, $|L'(u_2)|\geq 2$, $|L'(u_3)|\geq 2$, and $|L'(u_5)|\geq 3$. Observe that $u_2$ does not share a neighbor with $u_3$. Thus, we can either complete the coloring by \Cref{cor:Hall}, or by coloring $u_2$ and $u_3$ with the same color and finish with $u_1$ and $u_5$.
\item If $L(v_1)\cap L(u_4) = \emptyset$, then by symmetry, we also have $L(v_3)\cap L(u_1) = \emptyset$ or we would be back to the previous case. In this case, we start by coloring $v_2$ and $w_2$. Let $L'(x)$ be the remaining list of colors for $x\in S\setminus\{v_2,w_2\}$.
\begin{itemize}
\item If $L'(u_2)\cap L'(u_3)\neq \emptyset$, then we color $u_2$ and $u_3$ with the same color $c'$. Let $L''(x)$ be the remaining list of colors for $x\in S\setminus\{u_2,u_3,v_2,w_2\}$. Observe that, by similar arguments as above, recounting the colors for $L''(x)$ gives us $|L''(u_1)|\geq 5$, $|L''(u_4)|\geq 5$, $|L''(u_5)|\geq 6$, $|L''(v_1)|\geq 3$, $|L''(v_3)|\geq 3$, $|L''(v_4)|\geq 4$, and $|L''(v_5)|\geq 4$. We can conclude by \Cref{cor:Hall} since $|L''(v_1)\cup L''(u_4)|\geq 8$ and $|L''(v_3)\cup L''(u_1)|\geq 8$ as $L(v_1)\cap L(u_4)= \emptyset$ and $L(v_3)\cap L(u_1)= \emptyset$.

\item If $L'(u_2)\cap L'(u_3)=\emptyset$, then we have $|L'(u_2)\cup L'(u_3)|\geq 10$, $|L'(v_1)\cup L'(u_4)|\geq 10$, and $|L'(v_3)\cup L'(u_1)|\geq 10$ as $L(v_1)\cap L(u_4)= \emptyset$ and $L(v_3)\cap L(u_1)= \emptyset$. We complete the coloring by \Cref{cor:Hall}.
\end{itemize} 
\end{itemize}
This concludes the proof.
\end{proof}

\subsection{Discharging}\label{injg3:discharging}

Knowing the reducible structures in $G$, we apply the following rules in the discharging procedure to get a contradiction with \Cref{equation}:

\begin{itemize}
\item[\ru0] Every $5^+$-face $f$ gives $\frac{1}{3}$ to each incident $3$-vertex.
\item[\ru1] Every $5^+$-face $f$ gives $\frac{1}{3}$ to each adjacent $3$-face that is not adjacent with any $3$-faces.
\item[\ru2] Every $5^+$-face $f$ gives $\frac{1}{2}$ to each adjacent $3$-face that is adjacent with another $3$-face.
\item[\ru3] Every $6^+$-face $f$ gives $\frac{1}{3}$ to each adjacent bad $5$-face.
\end{itemize}

\begin{figure}[H]
\begin{minipage}[b]{0.2\textwidth}
\centering
\begin{tikzpicture}[scale=0.5]{thick}
\begin{scope}[every node/.style={circle,draw,minimum size=1pt,inner sep=2}]
    \node (1) at (0,0) {};
    \node[fill] (2) at (2,1) {};
    \node (3) at (4,0) {};
    \node (4) at (2,3) {};
    
    \node[draw=none] (f0) at (0,2) {$f$};
\end{scope}

\begin{scope}[every edge/.style={draw=black}]
    \path (1) edge (2);
    \path (2) edge (3);
    \path (2) edge (4);
    \path[->] (f0) edge node[above] {$\frac13$} (2);
\end{scope}
\end{tikzpicture}
\caption{\ru0.}
\end{minipage}
\begin{minipage}[b]{0.2\textwidth}
\centering
\begin{tikzpicture}[scale=0.5]{thick}
\begin{scope}[every node/.style={circle,draw,minimum size=1pt,inner sep=2}]
    \node (1) at (0,0) {};
    \node[draw=none] (2) at (2,1) {};
    \node (3) at (4,0) {};
    \node (4) at (2,3) {};
    
    \node[draw=none] (f0) at (-1,2) {$f$};
\end{scope}

\begin{scope}[every edge/.style={draw=black}]
    \path (1) edge (3);
    \path (3) edge (4);
    \path (1) edge (4);
    \path[->] (f0) edge node[above] {$\frac13$} (2);
\end{scope}
\end{tikzpicture}
\caption{\ru1.}
\end{minipage}
\begin{minipage}[b]{0.2\textwidth}
\centering
\begin{tikzpicture}[scale=0.5]{thick}
\begin{scope}[every node/.style={circle,draw,minimum size=1pt,inner sep=2}]
    \node (1) at (0,0) {};
    \node[draw=none] (2) at (2,1) {};
    \node (3) at (4,0) {};
    \node (4) at (2,3) {};
    \node (5) at (2,-3) {};
    
    \node[draw=none] (f0) at (-1,2) {$f$};
\end{scope}

\begin{scope}[every edge/.style={draw=black}]
    \path (1) edge (3);
    \path (3) edge (4);
    \path (1) edge (4);
    \path (3) edge (5);
    \path (1) edge (5);
    \path[->] (f0) edge node[above] {$\frac12$} (2);
\end{scope}
\end{tikzpicture}
\caption{\ru2.}
\end{minipage}
\begin{minipage}[b]{0.4\textwidth}
\centering
\begin{tikzpicture}[scale=1.5]{thick}
\begin{scope}[every node/.style={circle,draw,minimum size=1pt,inner sep=2}]
    \node[fill] (u1) at (1.95,1.31) {};
    \node[fill] (u2) at (1,1) {};
    \node[fill] (u3) at (1,0) {};
    \node[fill] (u4) at (1.95,-0.31) {};
    \node[fill] (u5) at (2.54,0.5) {};
    
    \node (v1) at (1.22,1.93) {};
    \node (v3) at (1.22,-0.93) {};
    \node (v4) at (2.9,-0.38) {};
    \node (v5) at (2.9,1.38) {};
    
    \node (v2) at (0.05,1.31) {};
    \node (w2) at (0.05,-0.31) {};
    
    \node[draw=none] (f) at (0.1, 0.5) {$f$};
    \node[draw=none] (bad) at (1.5, 0.5) {};
\end{scope}

\begin{scope}[every edge/.style={draw=black}]
    \path (u1) edge (u2);
    \path (u2) edge (u3);
    \path (u3) edge (u4);
    \path (u4) edge (u5);
    \path (u5) edge (u1);
    \path (u1) edge (v1);
    \path (u2) edge (v1);
    \path (u2) edge (v2);
    \path (u3) edge (w2);
    \path (u3) edge (v3);
    \path (u4) edge (v3);
    \path (u4) edge (v4);
    \path (u5) edge (v4);
    \path (u1) edge (v5);
    \path (u5) edge (v5);
    \path[->] (f) edge node[above left] {$\frac13$} (bad); 
\end{scope}
\end{tikzpicture}
\caption{\ru3.}
\end{minipage}
\end{figure}

Now we prove \Cref{thm:injg3} using both discharging procedure and structural properties of $G$ proven in \Cref{injg3:structure} and the discharging rules stated above.

\begin{proof}[Proof of \Cref{thm:injg3}]
Let $G$ be a minimal counterexample to the theorem and let $\mu(u)$ be the initial charge assignment for the vertices and faces of $G$ with the charge $\mu(u)=d(u)-4$ for each vertex $u\in V(G)$, and $\mu(f)=d(f)-4$ for each face $f\in F(G)$. By \Cref{equation}, we have that the total sum of the charges is negative.

Let $\mu^*$ be the assigned charges after the discharging procedure. In what follows, we will prove that: $$\forall x \in V(G)\cup F(G), \mu^*(x)\geq 0.$$

We first prove that the final charge on each vertex is non-negative. Let $u$ be a vertex in $V(G)$. Note that $u$ has degree at least $3$ by \Cref{injg3:minimumDegree}. Recall also that $\Delta(G)=4$ and $\mu(u)=d(u)-4$, thus we now consider the following two cases.

\textbf{Case 1:} $d(u)=4$\\ 
Then, $u$ does not give any charge. So,
$$\mu^*(u)=\mu(u)=d(u)-4=0.$$

\textbf{Case 2:} $d(u)=3$\\ 
Then, $\mu(u)=d(u)-4=-1$. By \Cref{injg3:config1}(ii), $u$ is incident to only $5^+$-faces. Thus, $u$ receives three times $\frac13$ by \ru0. So, 
$$\mu^*(u)=-1+3\cdot\frac13=0.$$

Secondly, we prove that the final charge on each face is also non-negative. Let $f$ be a face in $F(G)$ and recall that $\mu(f)=d(f)-4$. We distinguish the following cases.

\textbf{Case 1:} $d(f)\geq 6$\\
We claim that $f$ never gives more than $\frac13 d(f)$ by \ru0-\ru3. Indeed, we argue that $f$ sends at most $\frac13$ per incident edge. Observe that we can also view \ru0 and \ru2 as $f$ giving charge along an incident edge as shown in \Cref{fig:injg3:6faces}. Consider an edge $uv$ incident to $f$, if $f$ gives $\frac16$ along $uv$ by \ru0, then it cannot give additional charge along $uv$ by any other rule since a $3$-vertex cannot be incident to a $3$-cycle by \Cref{injg3:config1}(ii), and it cannot be at distance 1 from two adjacent $3$-cycles by \Cref{injg3:config1}(iii). If $f$ gives $\frac13$ along $uv$ by \ru1, then it cannot give along $uv$ by \ru2 due to \Cref{injg3:config1}(iv), nor by \ru3 by definition. If $f$ gives (at most $\frac13$) along $uv$ by \ru2, then it cannot also give along $uv$ by \ru3 due to \Cref{injg3:config3}.
To conclude, we have
$$ \mu^*(u)\geq \mu(u)-\frac13 d(f)=d(f)-4-\frac13 d(f) = \frac23 d(f)-4\geq 0$$
since $d(f)\geq 6$.

\begin{figure}[H]
\centering
\begin{subfigure}[b]{0.32\textwidth}
\centering
\begin{tikzpicture}[scale=0.6]{thick}
\begin{scope}[every node/.style={circle,draw,minimum size=1pt,inner sep=2}]
    \node (1) at (0,0) {};
    \node[fill] (2) at (2,1) {};
    \node (3) at (4,0) {};
    \node (4) at (2,3) {};
    
    \node[draw=none] (f0) at (0,2) {$f$};
\end{scope}

\begin{scope}[every edge/.style={draw=black}]
    \path (1) edge (2);
    \path (2) edge (3);
    \path (2) edge (4);
    \draw[->] (f0) to[out=-90,in=180] (1,0.5) node[below] {$\frac16$} to[out=0,in=-90] (2);
    \draw[->] (f0) to[out=0,in=180] (2,2) node[right] {$\frac16$} to[out=0,in=45] (2);
\end{scope}
\end{tikzpicture}
\caption*{\ru0.}
\end{subfigure}
\begin{subfigure}[b]{0.32\textwidth}
\centering
\begin{tikzpicture}[scale=0.6]{thick}
\begin{scope}[every node/.style={circle,draw,minimum size=1pt,inner sep=2}]
    \node (0) at (-3,0) {};
    \node[fill] (1) at (0,0) {};
    \node[draw=none] (2) at (2,1) {};
    \node (3) at (4,0) {};
    \node (4) at (2,3) {};
    \node (5) at (2,-3) {};
    
    \node[draw=none] (f0) at (-1.5,1) {$f$};
\end{scope}

\begin{scope}[every edge/.style={draw=black}]
    \path (0) edge (3);
    \path (3) edge (4);
    \path (1) edge (4);
    \path (3) edge (5);
    \path (1) edge (5);
    \path[->] (f0) edge node[above] {$\frac13$} (2);
    \path[->] (f0) edge[out=-90,in=-90] node[below] {$\frac16$} (2);
\end{scope}
\end{tikzpicture}
\caption*{\ru2.}
\end{subfigure}
\caption{A $6^+$-face $f$ sending charges through incident edges.}
\label{fig:injg3:6faces}
\end{figure}
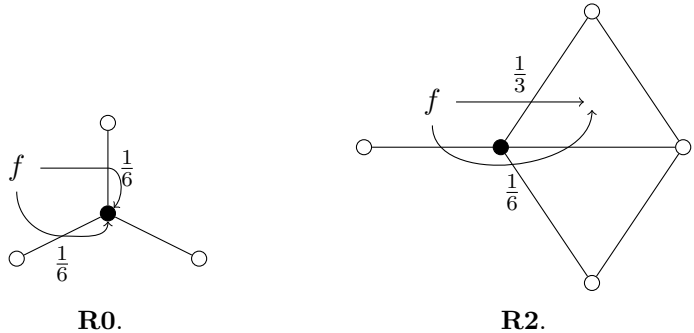

\textbf{Case 2:} $d(f)=5$\\
Recall that $\mu(f)=d(f)-4=1$. Let $i_0$, $i_1$, and $i_2$ be respectively the number of times $f$ gives charge by \ru0, \ru1, and \ru2. Note that \ru3 does not apply to $5$-faces. Observe now that we have the following inequalities:
\begin{itemize}
\item $i_1\leq 4$ by \Cref{injg3:config5}.
\item $i_0+i_2\leq 2$ by \Cref{injg3:config1} and \cref{injg3:config3}.
\item If $i_0\geq 1$, then $i_0+i_1\leq 3$ by \Cref{injg3:config1}(i, ii) and \Cref{injg3:config4}.
\end{itemize}

Recall that $f$ gives $\frac13 i_0 + \frac13 i_1 + \frac12 i_2$ by \ru0, \ru1, and \ru2.

If $i_2 = 2$, then $i_0=0$ since $i_0+i_2\leq 2$. We have $i_1=0$ due to \Cref{injg3:config1}(iv) and \Cref{injg3:config3}. Thus,
$$ \mu^*(f)\geq 1 - 2\cdot\frac12=0.$$ 

If $i_2 = 1$, then we have the two following cases since $i_0+i_2\leq 2$:
\begin{itemize}
\item If $i_0=1$, then $i_1=0$ due to \Cref{injg3:config1}(ii, iii, iv) and \Cref{injg3:config3}. Thus,
$$ \mu^*(f)\geq 1 - \frac13 - \frac12=\frac16.$$ 

\item If $i_0=0$, then $i_1\leq 1$ due to \Cref{injg3:config9}. Thus,
$$ \mu^*(f)\geq 1 - \frac13 - \frac12=\frac16.$$
\end{itemize}

If $i_2=0$, then we distinguish two cases:
\begin{itemize}
\item If $i_0\geq 1$, then recall $i_0+i_1\leq 3$. Thus,
$$ \mu^*(f)\geq 1 - 3\cdot\frac13 =0.$$
\item If $i_0=0$ and $i_1\leq 3$, then 
$$ \mu^*(f)\geq 1 - 3\cdot\frac13 =0.$$
\item If $i_0=0$ and $i_1=4$ (recall $i_1\leq 4$), then $f$ is adjacent to a $6^+$-face due to \Cref{injg3:config6} and \Cref{injg3:config7}. So $f$ receives $\frac13$ by \ru3. Thus, 
$$ \mu^*(f)\geq 1 - 4\cdot\frac13 +\frac13 =0.$$
\end{itemize}
Thus, we have that after the discharging procedure $5$-faces have non-negative charge.

\textbf{Case 3:} $d(f)=4$\\
Since $4$-faces do not give any charge, we have
$$ \mu^*(f)=\mu(f)=d(f)-4=0.$$

\textbf{Case 4:} $d(f)=3$\\
Recall that $\mu(f)=d(f)-4=-1$. If $f$ is adjacent only with $5^+$-faces, then $f$ receives three times $\frac13$ by \ru1. So,
$$ \mu^*(f)\geq -1+3\cdot\frac13=0.$$
On the other hand, if $f$ is adjacent to a $4^-$-face, then by \Cref{injg3:config2} and \Cref{injg3:config1}(iv) $f$ is adjacent with exactly one $3$-face. In such a case, it receives $\frac12$ from each of the two adjacent $5^+$-faces by \ru2. So,
$$ \mu^*(f)\geq -1+2\cdot\frac12=0.$$

Finally, it follows that after the discharging procedure the charge of every vertex and every face is non-negative and thus the final total sum is non-negative, a contradiction to \Cref{equation}.
\end{proof}

\section{Injective list-coloring of triangle-free planar graphs}\label{injg4}

Let $G$ be a minimal counterexample to \Cref{thm:injg4}. More precisely, $G$ has maximum degree 4, girth at least 4, and $\chi^i_\ell(G)\geq 10$.

\subsection{Structural properties of $G$}\label{injg4:structure}

Observe that since $g(G)\geq 4$, whenever two vertices are adjacent, they do not see each other (they do not share a common neighbor). Otherwise, $G$ would contain a $3$-cycle. As a result, an injective coloring of $G$ is also an exact square coloring as only vertices at distance exactly 2 see each other.

\begin{lemma}\label{injg4:minimumDegree}
The minimum degree of $G$ is at least 2.
\end{lemma}

\begin{proof}
If $G$ contains a 1-vertex $v$, then we can simply remove $v$ and color the resulting graph, which is possible by minimality of $G$. Then, we add $v$ back and extend the coloring, since $v$ shares a neighbor with at most $3$ other vertices and we have $9$ colors in total.
\end{proof}

Unlike the previous cases, we do not have enough colors to reduce a $2$-vertex directly. However, the presence of such a ``small'' vertex guarantees that its neighbors must have a larger neighborhood.

\begin{lemma}\label{injg4:counting}
If a $4$-vertex $u$ in $G$ is adjacent to a $2$-vertex, then $d^{\#2}(u)\geq 9$.
\end{lemma}

\begin{proof}
Suppose by contradiction that $u$ is a $4$-vertex that is adjacent to a $2$-vertex $v$ and $d^{\#2}(u)\leq 8$. Then, color $G-\{v\}$ by minimality and uncolor $u$. Vertex $u$ sees as many colors as $d^{\#2}(u)\leq 8$, so $u$ is colorable. Finish by coloring $v$ which sees only $d^{\#2}(v)\leq 6$ colors.
\end{proof}

\begin{lemma}\label{injg4:config1}
Graph $G$ cannot contain the following configurations:
\begin{itemize}
\item[(i)] Two adjacent $3^-$-vertices.
\item[(ii)] A $4$-vertex adjacent to two $2$-vertices.
\item[(iii)] A $4$-vertex adjacent to a $2$-vertex and two $3$-vertices.
\item[(iv)] A $2$-vertex incident to a $4$-cycle.
\item[(v)] A $3$-vertex incident to two $4$-cycles.
\item[(vi)] A $4$-vertex $u$ adjacent to a $2$-vertex and a $3$-vertex $v$, and $uv$ is incident to a $4$-cycle.
\end{itemize}
\end{lemma}

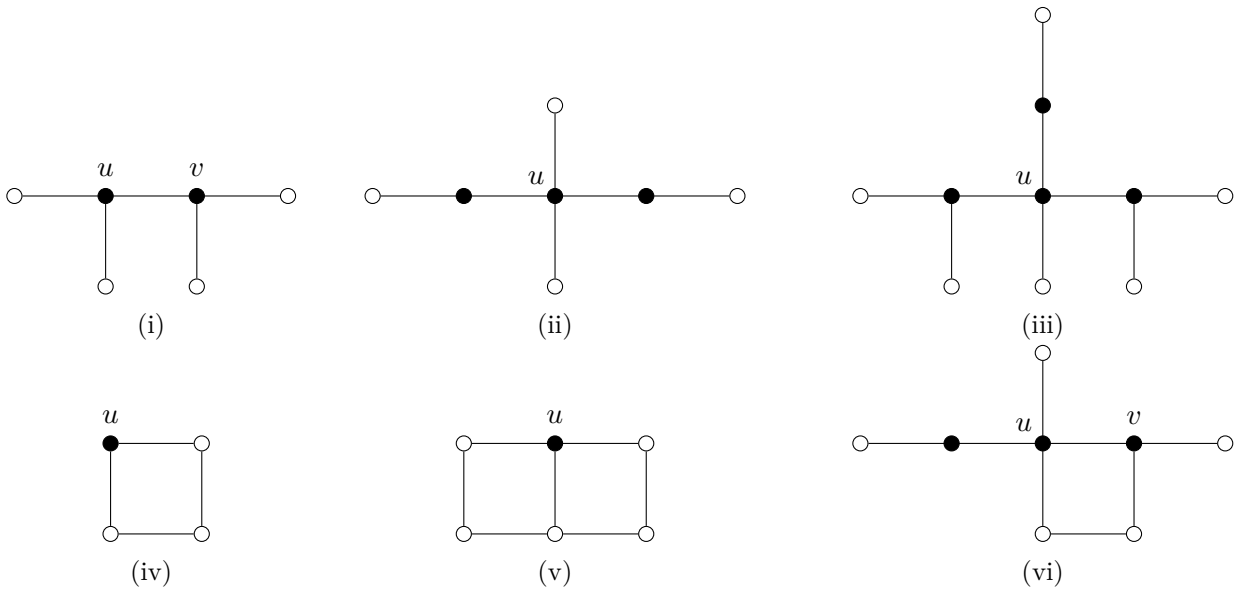
\begin{figure}[H]
\centering
\begin{subfigure}[b]{0.24\textwidth}
\centering
\begin{tikzpicture}[scale=0.6]{thick}
\begin{scope}[every node/.style={circle,draw,minimum size=1pt,inner sep=2}]
    \node (1) at (0,0) {};
    \node[fill,label={above:$u$}] (2) at (2,0) {};
    \node (20) at (2,-2) {};
    \node[fill,label={above:$v$}] (3) at (4,0) {};
    \node (30) at (4,-2) {};
    \node (4) at (6,0) {};
\end{scope}

\begin{scope}[every edge/.style={draw=black}]
    \path (1) edge (4);
    \path (2) edge (20);
    \path (3) edge (30);
\end{scope}
\end{tikzpicture}
\caption{}
\end{subfigure}
\begin{subfigure}[b]{0.37\textwidth}
\centering
\begin{tikzpicture}[scale=0.6]{thick}
\begin{scope}[every node/.style={circle,draw,minimum size=1pt,inner sep=2}]
    \node (0) at (-2,0) {};
    \node[fill] (1) at (0,0) {};
    \node[fill,label={above left:$u$}] (2) at (2,0) {};
    \node (20) at (2,-2) {};
    \node (21) at (2,2) {};
    \node[fill] (3) at (4,0) {};
    \node (4) at (6,0) {};
\end{scope}

\begin{scope}[every edge/.style={draw=black}]
    \path (0) edge (4);
    \path (21) edge (20);
\end{scope}
\end{tikzpicture}
\caption{}
\end{subfigure}
\begin{subfigure}[b]{0.37\textwidth}
\centering
\begin{tikzpicture}[scale=0.6]{thick}
\begin{scope}[every node/.style={circle,draw,minimum size=1pt,inner sep=2}]
    \node (0) at (-2,0) {};
    \node[fill] (1) at (0,0) {};
    \node (10) at (0,-2) {};
    \node[fill,label={above left:$u$}] (2) at (2,0) {};
    \node (20) at (2,-2) {};
    \node[fill] (21) at (2,2) {};
    \node (22) at (2,4) {};
    \node[fill] (3) at (4,0) {};
    \node (4) at (6,0) {};
    \node (30) at (4,-2) {};
\end{scope}

\begin{scope}[every edge/.style={draw=black}]
	\path (1) edge (10);    
    \path (0) edge (4);
    \path (22) edge (20);
    \path (3) edge (30);
\end{scope}
\end{tikzpicture}
\caption{}
\end{subfigure}

\begin{subfigure}[b]{0.24\textwidth}
\centering
\begin{tikzpicture}[scale=0.6]{thick}
\begin{scope}[every node/.style={circle,draw,minimum size=1pt,inner sep=2}]
    \node[fill,label={above:$u$}] (2) at (2,0) {};
    \node (20) at (2,-2) {};
    \node (3) at (4,0) {};
    \node (30) at (4,-2) {};
\end{scope}

\begin{scope}[every edge/.style={draw=black}]
    \path (2) edge (3);
    \path (2) edge (20);
    \path (3) edge (30);
    \path (20) edge (30);
\end{scope}
\end{tikzpicture}
\caption{}
\end{subfigure}
\begin{subfigure}[b]{0.37\textwidth}
\centering
\begin{tikzpicture}[scale=0.6]{thick}
\begin{scope}[every node/.style={circle,draw,minimum size=1pt,inner sep=2}]
    \node (1) at (0,0) {};
    \node (10) at (0,-2) {};
    \node[fill,label={above:$u$}] (2) at (2,0) {};
    \node (20) at (2,-2) {};
    \node (3) at (4,0) {};
    \node (30) at (4,-2) {};
\end{scope}

\begin{scope}[every edge/.style={draw=black}]
    \path (1) edge (3);
    \path (1) edge (10);
    \path (2) edge (20);
    \path (3) edge (30);
    \path (10) edge (20);
    \path (20) edge (30);
\end{scope}
\end{tikzpicture}
\caption{}
\end{subfigure}
\begin{subfigure}[b]{0.37\textwidth}
\centering
\begin{tikzpicture}[scale=0.6]{thick}
\begin{scope}[every node/.style={circle,draw,minimum size=1pt,inner sep=2}]
    \node (0) at (-2,0) {};
    \node[fill] (1) at (0,0) {};
    \node[fill,label={above left:$u$}] (2) at (2,0) {};
    \node (20) at (2,-2) {};
    \node (21) at (2,2) {};
    \node[fill,label={above:$v$}] (3) at (4,0) {};
    \node (4) at (6,0) {};
    \node (30) at (4,-2) {};
\end{scope}

\begin{scope}[every edge/.style={draw=black}]
    \path (0) edge (4);
    \path (21) edge (20);
    \path (3) edge (30);
    \path (20) edge (30);
\end{scope}
\end{tikzpicture}
\caption{}
\end{subfigure}
\caption{Reducible configurations in \Cref{injg4:config1}.}
\end{figure}

\begin{proof}
We separate the proof into four parts based on the configurations.
\begin{itemize}
\item[(i)] Suppose by contradiction that there exist two adjacent $3^-$-vertices $u$ and $v$. Color $G-\{uv\}$ by minimality. Uncolor $u$ and $v$. Observe that $d^{\#2}(u)\leq 8$. The same holds for $v$. Thus, $u$ and $v$ are colorable.
\item[(iv)] Suppose by contradiction that there exists a $2$-vertex $u$ incident to a $4$-cycle. Color $G-\{u\}$ by minimality. Observe that the two neighbors of $u$ will also have different colors in $G$ since they are at distance $2$ in $G-\{u\}$. Thus, we only need to color $u$ which sees only $d^{\#2}(u)\leq 5$ colors.
\item[(v)] Suppose by contradiction that there exists a $3$-vertex $u$ incident to two $4$-cycles. Let $e$ be the edge incident to $u$ that is incident to both cycles. Color $G-\{e\}$ by minimality and uncolor $u$. Observe that every pair of neighbors of $u$ are still at distance 2 in $G-\{e\}$. Thus, we only need to color $u$ which sees only $d^{\#2}(u)\leq 7$ colors.
\item[(ii),] (iii), and (vi) Observe that the $4$-vertex $u$ with the $2$-neighbor in these configurations always verifies $d^{\#2}(u)\leq 8$, which is impossible due to \Cref{injg4:counting}.
\end{itemize}
Thus, if $G$ contains any of the above configurations, then $\chi^i_\ell(G)\leq 9$, a contradiction.
\end{proof}

Before continuing with proving some more structural results, we first give some additional useful definitions and observations.

\textbf{Good and bad faces:} We call a $5$-face \emph{bad} if it is incident to a $2$-vertex and a $3$-vertex. Additionally, we call a $5^+$-face \emph{good}, if it is not a bad $5$-face.

The following observation is a direct consequence of \Cref{injg4:config1}(i).

\begin{observation}
A $2$-vertex and a $3$-vertex on a bad $5$-face $f$ in $G$ must be at distance $2$ and they are the only $3^-$-vertices on $f$.
\end{observation}

\begin{figure}[H]
\centering
\begin{tikzpicture}[scale=1.5,rotate=-18]{thick}
\begin{scope}[every node/.style={circle,draw,minimum size=1pt,inner sep=2}]
    \node (u1) at (0.05,1.31) {};
    \node[fill] (u2) at (-0.54,0.5) {};
    \node (u3) at (0.05,-0.31) {};
    \node (u4) at (1,0) {};
    \node[fill] (u5) at (1,1) {};
    
    \node (v2) at (-1.2,0.02) {};
\end{scope}

\begin{scope}[every edge/.style={draw=black}]
    \path (u1) edge (u2);
    \path (u2) edge (u3);
    \path (u3) edge (u4);
    \path (u4) edge (u5);
    \path (u5) edge (u1);
    \path (u2) edge (v2);
\end{scope}
\end{tikzpicture}
\caption{A bad face.}
\end{figure}
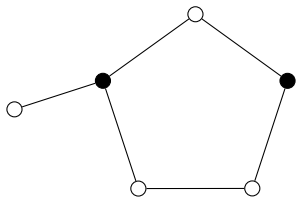

To further help us with the proofs, we now divide $3^-$-vertices into three different types. 

\textbf{Small, medium, and large $3^-$-vertices:} We call a $3^-$-vertex \emph{small}, if it is either a $2$-vertex or a $3$-vertex incident to a bad $5$-face and a $4$-face. A $3$-vertex is called \emph{medium}, if it is incident to either a bad $5$-face or a $4$-face. Finally, a $3$-vertex is called \emph{large}, if it is neither medium nor small. 

Due to \Cref{injg4:config1}(vi) we have the following observation.

\begin{observation}
A $4$-face in $G$, adjacent with a bad $5$-face $f$ and incident to a small $3$-vertex $v$, is not incident to the common neighbor of a vertex $v$ and the $2$-vertex on $f$. 
\end{observation}

\begin{figure}[H]
\begin{subfigure}[b]{0.49\textwidth}
\centering
\begin{tikzpicture}[scale=0.6]{thick}
\begin{scope}[every node/.style={circle,draw,minimum size=1pt,inner sep=2}]
    \node (1) at (0,0) {};
    \node[fill,label={above:$u$}] (2) at (2,0) {};
    \node (3) at (4,0) {};
\end{scope}

\begin{scope}[every edge/.style={draw=black}]
    \path (1) edge (3);
\end{scope}
\end{tikzpicture}
\begin{tikzpicture}[scale=1.5,rotate=-25]{thick}
\begin{scope}[every node/.style={circle,draw,minimum size=1pt,inner sep=2}]
    \node (u1) at (0.05,1.31) {};
    \node[fill,label={above:$u$}] (u2) at (-0.54,0.5) {};
    \node (u3) at (0.05,-0.31) {};
    \node (u4) at (1,0) {};
    \node[fill] (u5) at (1,1) {};
    
    \node (v2) at (-1.2,0.02) {};
    \node (w2) at (-0.61,-0.79) {};
\end{scope}

\begin{scope}[every edge/.style={draw=black}]
    \path (u1) edge (u2);
    \path (u2) edge (u3);
    \path (u3) edge (u4);
    \path (u4) edge (u5);
    \path (u5) edge (u1);
    \path (u2) edge (v2);
    \path (v2) edge (w2);
    \path (w2) edge (u3);
\end{scope}
\end{tikzpicture}
\caption{A small vertex $u$.}
\end{subfigure}
\quad
\begin{subfigure}[b]{0.49\textwidth}
\centering
\begin{tikzpicture}[scale=0.6]{thick}
\begin{scope}[every node/.style={circle,draw,minimum size=1pt,inner sep=2}]
    \node (1) at (0,0) {};
    \node[fill,label={above:$u$}] (2) at (2,0) {};
    \node (20) at (2,-2) {};
    \node (3) at (4,0) {};
    \node (30) at (4,-2) {};
\end{scope}

\begin{scope}[every edge/.style={draw=black}]
    \path (1) edge (3);
    \path (2) edge (20);
    \path (3) edge (30);
    \path (20) edge (30);
\end{scope}
\end{tikzpicture}
\quad
\begin{tikzpicture}[scale=1.5,rotate=-18]{thick}
\begin{scope}[every node/.style={circle,draw,minimum size=1pt,inner sep=2}]
    \node (u1) at (0.05,1.31) {};
    \node[fill,label={above:$u$},label={[label distance=+0.3cm]above left:$5^+$-face}] (u2) at (-0.54,0.5) {};
    \node (u3) at (0.05,-0.31) {};
    \node (u4) at (1,0) {};
    \node[fill] (u5) at (1,1) {};
    
    \node (v2) at (-1.2,0.02) {};
    \node[draw=none] (w2) at (-0.7,-0.1) {$5^+$-face};
\end{scope}

\begin{scope}[every edge/.style={draw=black}]
    \path (u1) edge (u2);
    \path (u2) edge (u3);
    \path (u3) edge (u4);
    \path (u4) edge (u5);
    \path (u5) edge (u1);
    \path (u2) edge (v2);
\end{scope}
\end{tikzpicture}
\caption{A medium vertex $u$.}
\end{subfigure}
\caption{Small and medium $3^-$-vertices.}
\end{figure}
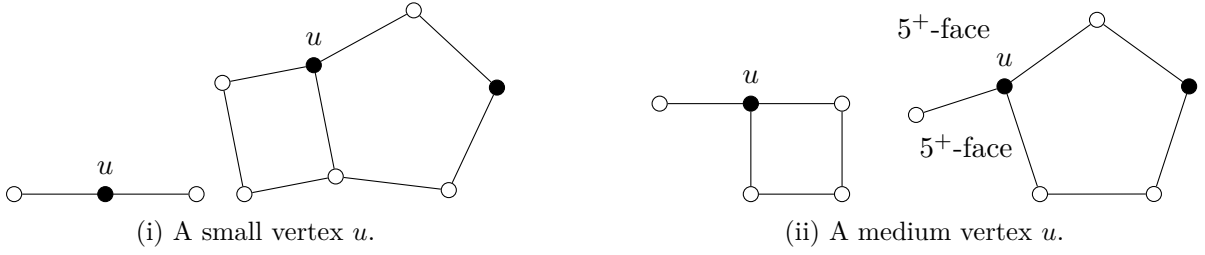

We are now ready to prove some structural properties regarding bad $5$-faces.

\begin{lemma}\label{injg4:config3}
Let $f=v_1v_2v_3v_4v_5$ be a bad $5$-face in $G$ where $v_1$ is the $3$-vertex and $v_3$ is the $2$-vertex. Let $f'=v'_1v'_2v'_3v'_4v'_5\neq f$ be another $5$-face incident to $v_1=v'_1$. Then, we have the following:
\begin{itemize}
\item If $f'$ is incident to $v_1v_2$, then $f'$ does not contain any other $3^-$-vertices (distinct from $v_1$).
\item If $f'$ is incident to $v_1v_5$, then $f'$ does not contain any other small vertices (distinct from $v_1$).
\end{itemize}
\end{lemma}

\begin{figure}[!htb]
\centering
\begin{subfigure}[b]{0.49\textwidth}
\centering
\begin{tikzpicture}[scale=1.5]{thick}
\begin{scope}[every node/.style={circle,draw,minimum size=1pt,inner sep=2}]
    \node[label={above:$v_5$}] (v2) at (0.05,1.31) {};
    \node[label={above:$v_4$}] (v3) at (-0.54,0.5) {};
    \node[fill,label={below:$v_3$},label={above:$4$}] (v4) at (0.05,-0.31) {};
    \node[label={below:$v_2$},label={above left:$1$}] (v5) at (1,0) {};
    \node[fill,label={above:$v_1$},label={below left:$2$}] (v1) at (1,1) {};
    \node[label={above:$v'_5$}] (v'2) at (1.95,1.31) {};
    \node[label={above:$v'_4$},label={left:$2$}] (v'3) at (2.54,0.5) {};
    \node[label={below:$v'_3$}] (v'4) at (1.95,-0.31) {};
\end{scope}

\begin{scope}[every edge/.style={draw=black}]
    \path (v1) edge (v2);
    \path (v2) edge (v3);
    \path (v3) edge (v4);
    \path (v4) edge (v5);
    \path (v5) edge (v1);
    \path (v5) edge (v'4);
    \path (v1) edge (v'2);
    \path (v'3) edge (v'2);
    \path (v'3) edge (v'4);
\end{scope}
\end{tikzpicture}
\caption*{Case 1.}
\end{subfigure}
\begin{subfigure}[b]{0.49\textwidth}
\centering
\begin{tikzpicture}[scale=1.5]{thick}
\begin{scope}[every node/.style={circle,draw,minimum size=1pt,inner sep=2}]
    \node[label={above:$v_2$}] (v2) at (0.05,1.31) {};
    \node[fill,label={above:$v_3$}] (v3) at (-0.54,0.5) {};
    \node[label={below:$v_4$}] (v4) at (0.05,-0.31) {};
    \node[label={below:$v_5$}] (v5) at (1,0) {};
    \node[fill,label={above:$v_1$}] (v1) at (1,1) {};
    \node[label={above:$v'_2$}] (v'2) at (1.95,1.31) {};
    \node[label={above:$v'_3$}] (v'3) at (2.54,0.5) {};
    \node[label={below:$v'_4$}] (v'4) at (1.95,-0.31) {};
\end{scope}

\begin{scope}[every edge/.style={draw=black}]
    \path (v1) edge (v2);
    \path (v2) edge (v3);
    \path (v3) edge (v4);
    \path (v4) edge (v5);
    \path (v5) edge (v1);
    \path (v5) edge (v'4);
    \path (v1) edge (v'2);
    \path (v'3) edge (v'2);
    \path (v'3) edge (v'4);
\end{scope}
\end{tikzpicture}
\caption*{Case 2.}
\end{subfigure}
\caption{Reducible configurations from \Cref{injg4:config3}.}
\end{figure}
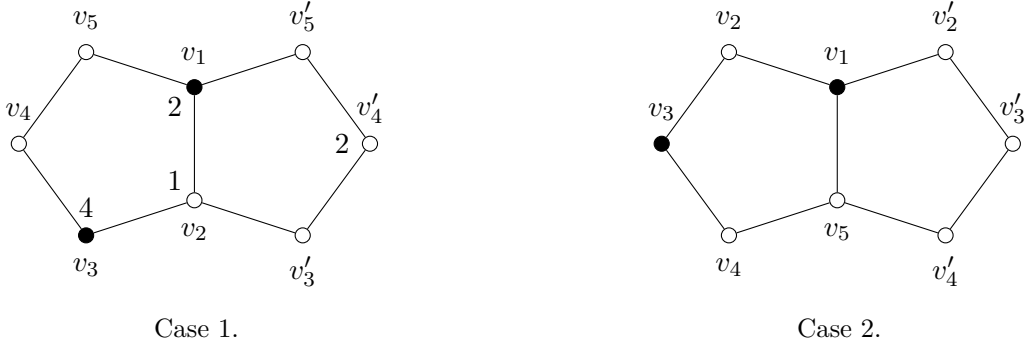

\begin{proof}
We assume w.l.o.g. that $v'_1=v_1$. Since $g(G)\geq 4$, every vertex of $f$ and $f'$ (except for the two common vertices that is $v_1$ and one of its neighbor) is distinct.

Suppose by contradiction that $f'$ contains (another) small vertex different from $v_1$.

\textbf{Case 1:} If $f'$ is incident to $v_1v_2$, say $v'_2=v_2$. First, observe that $d(v'_3)=4$ due to \Cref{injg4:config1}(iii) and $d(v'_5)=4$ due to \Cref{injg4:config1}(i). Thus, $v'_4$ must be a $3^-$-vertex. Color $G-\{v_3\}$ and uncolor $v_1$, $v_2$, and $v'_4$. Observe that $|L(v_1)|\geq 2$, $|L(v_2)|\geq 1$, $|L(v_3)|\geq 4$, $|L(v'_4)|\geq 2$. Therefore, we can color $v_2$, $v'_4$, $v_1$, and $v_3$ in this order.

\textbf{Case 2:} If $f'$ is incident to $v_1v_5$, say $v'_5=v_5$. By \Cref{injg4:config1}(i), $v'_2$ cannot be a small vertex, and at most one of $v'_3$ and $v'_4$ can be. Thus, we have the following two cases:
\begin{itemize}
\item If $v'_4$ is a $3^-$-vertex, then color $G-\{v_3\}$ and uncolor $v_1$, $v_4$, and $v'_4$. Observe that $|L(v_1)|\geq 3$, $|L(v_4)|\geq 1$, $|L(v_3)|\geq 4$, $|L(v'_4)|\geq 2$. Therefore, we can color $v_4$, $v'_4$, $v_1$, and $v_3$ in this order.
\item If $v'_3$ is a $3^-$-vertex, then recall that $v'_3$ is a small vertex. 
\begin{itemize}
\item If $v'_3$ is a small $3$-vertex, then it is incident to a bad $5$-face $f''\neq f'$ (since $f'$ is a good face) and a $4$-face. If $f''$ is incident to $v'_2v'_3$, then the $4$-face must be incident to $v'_3v'_4$. By \Cref{injg4:config1}(iii, vi), $f''$ cannot be incident to a $2$-vertex, which is a contradiction. Thus, $f''$ must be incident to $v'_3v'_4$. By \Cref{injg4:config1}(vi), the $2$-vertex incident to $f''$ must be adjacent to $v'_4$. However, in this case, we can use the same proof as in \textbf{Case 1} from the point of view of $v'_3$, $f''$, and $f'$ instead.
\item If $v'_3$ is a $2$-vertex, then we color $G-\{v_3\}$ and uncolor every vertex on $f$ and $f'$. Observe that the remaining list of colors for these vertices have size: $|L(v_1)|\geq 5$, $|L(v_2)|\geq 3$, $|L(v_3)|\geq 5$, $|L(v_4)|\geq 2$, $|L(v_5)|\geq 2$, $|L(v'_2)|\geq 3$, $|L(v'_3)|\geq 5$, and $|L(v'_4)|\geq 2$. Moreover, if we can color $v_2$, $v_4$, $v_5$, $v'_2$, and $v'_4$, then we can always finish by coloring $v_1$, $v'_3$, and $v_3$ in this order.

If $L(v_4)$ and $L(v_5)$ have a common color $c$, then we color them with $c$ and color $v'_4$, $v_2$, and $v'_2$ in this order. The same holds for $L(v'_4)$ and $L(v_5)$. As a result, $|L(v_4)\cup L(v_5)|\geq 4$ and $|L(v'_4)\cup L(v_5)|\geq 4$.

If $L(v_4)$ and $L(v'_2)$ have a common color $c$, then we proceed as follows. Suppose $v_4$ and $v'_2$ do not see each other, then we color them with $c$. Recall that $L(v'_4)\cap L(v_5)=\emptyset$. So, we can color $v'_4$, $v_2$, and $v_5$ in this order. The same holds for $L(v'_4)$ and $L(v_2)$. 
As a result, if $L(v_4)$ and $L(v'_2)$ share a color, then $v_4$ must see $v'_2$. The same holds for $v'_4$ and $v_2$. By planarity, $v_2$, $v_4$, $v'_2$, and $v'_4$ must share a common neighbor $u$. However, this is impossible since $v_3$ would be a $2$-vertex incident to the $4$-cycle $v_2v_3v_4u$ contradicting \Cref{injg4:config1}(iv). Finally, we must have $|L(v_4)\cup L(v'_2)|\geq 5$ or $|L(v'_4)\cup L(v_2)|\geq 5$.

Finally, we have $|L(v_4)\cup L(v_5)|\geq 4$, $|L(v'_4)\cup L(v_5)|\geq 4$, and at least one of the following two inequalities: $L(v_4)\cup L(v'_2)|\geq 5$ or $|L(v'_4)\cup L(v_2)|\geq 5$. Therefore, we can always color $v_2$, $v_4$, $v_5$, $v'_2$, and $v'_4$ by \Cref{cor:Hall}.
\end{itemize} 
\end{itemize}
Thus, we can conclude that $\chi^i_\ell(G)\leq 9$, a contradiction.
\end{proof}

Finally, we show that small vertices cannot be close to each other from the perspective of a face of size at least $6$.

\begin{definition}[Facial-distance]
Let $f=u_1u_2\dots u_{d(f)}$ be a face in $F(G)$, and let $u_i$ and $u_j$ be vertices incident to $f$. The facial-distance on $f$ between $u_i$ and $u_j$ is their distance on the cycle $u_1u_2\dots u_{d(f)}$ (which is $\min(i-j (\text{\emph{mod} }d(f)),j-i (\text{\emph{mod} }d(f)))$).
\end{definition}

\begin{lemma}\label{injg4:config4}
Two small vertices incident to a same $6^+$-face $f$ in $G$ are at facial-distance at least $3$ on $f$.
\end{lemma}

\begin{proof}
By \Cref{injg4:config1}(i), small vertices cannot be adjacent. By \Cref{injg4:config1}(ii), two $2$-vertices must be at distance at least $3$. We only need to check if a small $3$-vertex and a $2$-vertex, or two small $3$-vertices can be at facial-distance $2$ on $f$. Let $f=v_1v_2v_3v_4\dots$ is a $6^+$-face.

Suppose that $v_1$ is a $2$-vertex and $v_3$ a small $3$-vertex. Observe that $v_2v_3$ cannot be incident to a $4$-face by \Cref{injg4:config1}(vi), so $v_3v_4$ must be incident to a $4$-face and $v_2v_3$ is incident to a bad $5$-face (different from $f$ since $f$ is a $6^+$-face). However, due to \Cref{injg4:config1}(ii, vi), the bad $5$-face incident to $v_2v_3$ cannot be incident to a $2$-vertex, which is a contradiction.

Now, suppose that $v_1$ and $v_3$ are small $3$-vertices. They must both be incident to some $4$-faces and bad $5$-faces. If $v_3v_4$ is incident to a $4$-face, then $v_2v_3$ must be incident to a bad $5$-face. However, by \Cref{injg4:config1}(i, iii, vi), this $5$-face cannot be incident to any $2$-vertex. As a result, $v_2v_3$ must be incident to a $4$-face. By symmetry, $v_1v_2$ is also incident to a $4$-face. Additionally, $v_3v_4$ must be incident to a bad $5$-face. By \Cref{injg4:config1}(vi), the $2$-vertex $u$ on this bad $5$-face must be adjacent to $v_4$. Now, color $G-\{u\}$ and uncolor $v_4$ and $v_2$. Observe that $|L(v_4)|\geq 1$, $|L(v_2)|\geq 2$, and $|L(u)|\geq 3$. Thus, we can finish by coloring $v_4$, $v_2$, and $v_3$ in this order.
\end{proof}

\begin{figure}[H]
\centering
\begin{subfigure}[b]{0.49\textwidth}
\centering
\begin{tikzpicture}[scale=1.5]{thick}
\begin{scope}[every node/.style={circle,draw,minimum size=1pt,inner sep=2}]
    \node (u1) at (0.05,1.31) {};
    \node[fill,label={above:$v_3$}] (u2) at (-0.54,0.5) {};
    \node (u3) at (0.05,-0.31) {};
    \node (u4) at (1,0) {};
    \node[fill] (u5) at (1,1) {};
    
    \node[label={above:$v_2$},label={[label distance=+1cm]above:$f$}] (v2) at (-1.2,0.02) {};
    \node (w2) at (-0.61,-0.79) {};
    \node[fill,label={above:$v_1$}] (t) at (-1.86,0.5) {};
    \node (t1) at (-2.45,1.31) {};
\end{scope}

\begin{scope}[every edge/.style={draw=black}]
    \path (u1) edge (u2);
    \path (u2) edge (u3);
    \path (u3) edge (u4);
    \path (u4) edge (u5);
    \path (u5) edge (u1);
    \path (u2) edge (v2);
    \path (v2) edge (w2);
    \path (w2) edge (u3);
    \path (v2) edge (t);
    \path (t) edge (t1);
\end{scope}
\end{tikzpicture}
\end{subfigure}
\begin{subfigure}[b]{0.49\textwidth}
\centering
\begin{tikzpicture}[scale=1.5,rotate=-105]{thick}
\begin{scope}[every node/.style={circle,draw,minimum size=1pt,inner sep=2},yscale=-1,xscale=1]
    \node[label={above:$v_2$},label={[label distance=+1cm]above:$f$}] (u1) at (0.05,1.31) {};
    \node[fill,label={above:$v_3$}] (u2) at (-0.54,0.5) {};
    \node (u3) at (0.05,-0.31) {};
    \node (u4) at (1,0) {};
    \node[fill] (u5) at (1,1) {};
    
    \node[label={above:$v_4$}] (v2) at (-1.2,0.02) {};
    \node (w2) at (-0.61,-0.79) {};
    
    \node[fill,label={above:$v_1$}] (t) at (0.05,2.31) {};
    \node (t1) at (-0.54,3.12) {};
\end{scope}

\begin{scope}[every edge/.style={draw=black}]
    \path (u1) edge (u2);
    \path (u2) edge (u3);
    \path (u3) edge (u4);
    \path (u4) edge (u5);
    \path (u5) edge (u1);
    \path (u2) edge (v2);
    \path (v2) edge (w2);
    \path (w2) edge (u3);
    \path (u1) edge (t);
    \path (t) edge (t1);
\end{scope}
\end{tikzpicture}
\end{subfigure}

\begin{subfigure}[b]{0.49\textwidth}
\centering
\begin{tikzpicture}[scale=1.5,rotate=-108]{thick}
\begin{scope}[every node/.style={circle,draw,minimum size=1pt,inner sep=2},yscale=-1,xscale=1]
    \node[label={above:$v_2$},label={[label distance=+1cm]above:$f$}] (u1) at (0.05,1.31) {};
    \node[fill,label={above:$v_3$}] (u2) at (-0.54,0.5) {};
    \node (u3) at (0.05,-0.31) {};
    \node (u4) at (1,0) {};
    \node[fill] (u5) at (1,1) {};
    
    \node[label={above:$v_4$}] (v2) at (-1.2,0.02) {};
    \node (w2) at (-0.61,-0.79) {};
    
    \node[fill,label={above:$v_1$}] (t) at (0.05,2.31) {};
    \node (v) at (-0.1,3) {};
    \node (t1) at (1,2.62) {};
\end{scope}

\begin{scope}[every edge/.style={draw=black}]
    \path (u1) edge (u2);
    \path (u2) edge (u3);
    \path (u3) edge (u4);
    \path (u4) edge (u5);
    \path (u5) edge (u1);
    \path (u2) edge (v2);
    \path (v2) edge (w2);
    \path (w2) edge (u3);
    \path (u1) edge (t);
    \path (t) edge (v);
    \path (t) edge (t1);
\end{scope}
\end{tikzpicture}
\end{subfigure}
\begin{subfigure}[b]{0.49\textwidth}
\centering
\begin{tikzpicture}[scale=1.5]{thick}
\begin{scope}[every node/.style={circle,draw,minimum size=1pt,inner sep=2}]
    \node[label={above:$v_4$}] (u1) at (0.05,1.31) {};
    \node[fill,label={above:$v_3$}] (u2) at (-0.54,0.5) {};
    \node (u3) at (0.05,-0.31) {};
    \node (u4) at (1,0) {};
    \node[fill] (u5) at (1,1) {};
    
    \node[fill,label={above:$v_2$},label={[label distance=+1cm]above:$f$}] (v2) at (-1.2,0.02) {};
    \node (w2) at (-0.61,-0.79) {};
    \node (w3) at (-1.79,-0.79) {};
    \node (t1) at (-2.45,1.31) {};
    \node[fill,label={above:$v_1$}] (t2) at (-1.86,0.5) {};
    \node (t3) at (-2.45,-0.31) {};
    \node (t4) at (-3.4,0) {};
    \node[fill] (t5) at (-3.4,1) {};
\end{scope}

\begin{scope}[every edge/.style={draw=black}]
    \path (u1) edge (u2);
    \path (u2) edge (u3);
    \path (u3) edge (u4);
    \path (u4) edge (u5);
    \path (u5) edge (u1);
    \path (u2) edge (v2);
    \path (v2) edge (w2);
    \path (w2) edge (u3);
    \path (v2) edge (t2);
    \path (t2) edge (t1);
    \path (t2) edge (t3);
    \path (t3) edge (t4);
    \path (t4) edge (t5);
    \path (t5) edge (t1);
    \path (w3) edge (v2);
    \path (w3) edge (t3);
\end{scope}
\end{tikzpicture}
\end{subfigure}
\caption{Reducible configurations in \Cref{injg4:config4}.}
\end{figure}
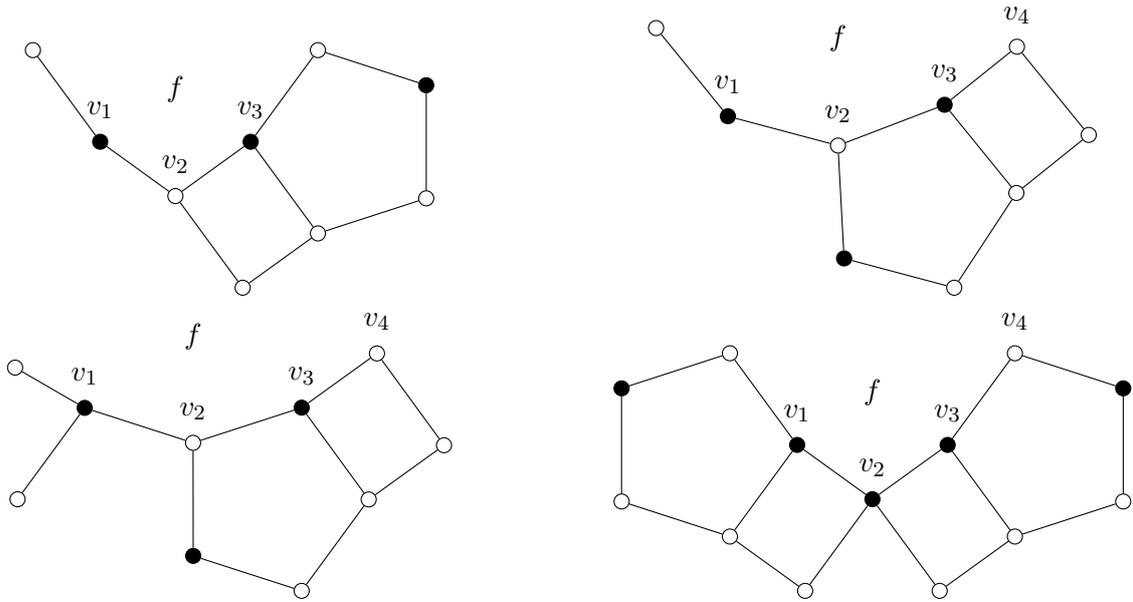

\subsection{Discharging procedure}\label{injg4:discharging}

To get a contradiction with \Cref{equation} we apply the following rules in the discharging procedure:

\begin{itemize}
\item[\ru0] Every $5^+$-face $f$ gives $1$ to each $2$-vertex.
\item[\ru1] Every good $5^+$-face $f$ gives $1$ to each small $3$-vertex.
\item[\ru2] Every good $5^+$-face $f$ gives $\frac12$ to each medium $3$-vertex.
\item[\ru3] Every $5^+$-face $f$ gives $\frac13$ to each large $3$-vertex.
\end{itemize}

\begin{figure}[H]
\begin{minipage}[b]{0.14\textwidth}
\centering
\begin{tikzpicture}[scale=0.5]{thick}
\begin{scope}[every node/.style={circle,draw,minimum size=1pt,inner sep=2}]
    \node (1) at (0,0) {};
    \node[fill,label={below:$u$}] (2) at (2,0) {};
    \node (3) at (4,0) {};
    \node[draw=none] (f) at (2,2) {$f$};
\end{scope}

\begin{scope}[every edge/.style={draw=black}]
    \path (1) edge (3);
    \path[->] (f) edge node[right] {1} (2);
\end{scope}
\end{tikzpicture}
\caption{\ru0.}
\end{minipage}
\begin{minipage}[b]{0.24\textwidth}
\centering
\begin{tikzpicture}[scale=1.5,rotate=-36]{thick}
\begin{scope}[every node/.style={circle,draw,minimum size=1pt,inner sep=2}]
    \node (u1) at (0.05,1.31) {};
    \node[fill,label={below left:$u$}] (u2) at (-0.54,0.5) {};
    \node (u3) at (0.05,-0.31) {};
    \node (u4) at (1,0) {};
    \node[fill] (u5) at (1,1) {};
    
    \node (v2) at (-1.2,0.02) {};
    \node (w2) at (-0.61,-0.79) {};
    \node[draw=none] (f) at (-1,1.1) {$f$};
\end{scope}

\begin{scope}[every edge/.style={draw=black}]
    \path (u1) edge (u2);
    \path (u2) edge (u3);
    \path (u3) edge (u4);
    \path (u4) edge (u5);
    \path (u5) edge (u1);
    \path (u2) edge (v2);
    \path (v2) edge (w2);
    \path (w2) edge (u3);
    \path[->] (f) edge node[right] {1} (u2);
\end{scope}
\end{tikzpicture}
\caption{\ru1.}
\end{minipage}
\begin{minipage}[b]{0.45\textwidth}
\centering
\begin{tikzpicture}[scale=0.6]{thick}
\begin{scope}[every node/.style={circle,draw,minimum size=1pt,inner sep=2}]
    \node (1) at (0,0) {};
    \node[fill,label={below left:$u$}] (2) at (2,0) {};
    \node (20) at (2,-2) {};
    \node (3) at (4,0) {};
    \node (30) at (4,-2) {};
    \node[draw=none] (f) at (2,2) {$f$};
\end{scope}

\begin{scope}[every edge/.style={draw=black}]
    \path (1) edge (3);
    \path (2) edge (20);
    \path (3) edge (30);
    \path (20) edge (30);
    \path[->] (f) edge node[right] {$\frac12$} (2);
\end{scope}
\end{tikzpicture}
\begin{tikzpicture}[scale=1.5,rotate=-36]{thick}
\begin{scope}[every node/.style={circle,draw,minimum size=1pt,inner sep=2}]
    \node (u1) at (0.05,1.31) {};
    \node[fill,label={below left:$u$}] (u2) at (-0.54,0.5) {};
    \node (u3) at (0.05,-0.31) {};
    \node (u4) at (1,0) {};
    \node[fill] (u5) at (1,1) {};
    
    \node (v2) at (-1.2,0.02) {};
    \node[draw=none] (f) at (-1,1.1) {$f$};
\end{scope}

\begin{scope}[every edge/.style={draw=black}]
    \path (u1) edge (u2);
    \path (u2) edge (u3);
    \path (u3) edge (u4);
    \path (u4) edge (u5);
    \path (u5) edge (u1);
    \path (u2) edge (v2);
    \path[->] (f) edge node[right] {$\frac12$} (u2);
\end{scope}
\end{tikzpicture}
\caption{\ru2.}
\end{minipage}
\begin{minipage}[b]{0.15\textwidth}
\centering
\begin{tikzpicture}[scale=0.5]{thick}
\begin{scope}[every node/.style={circle,draw,minimum size=1pt,inner sep=2}]
    \node (1) at (0,0) {};
    \node[fill] (2) at (2,1) {};
    \node (3) at (4,0) {};
    \node (4) at (2,3) {};
    
    \node[draw=none] (f0) at (0,2) {$f$};
\end{scope}

\begin{scope}[every edge/.style={draw=black}]
    \path (1) edge (2);
    \path (2) edge (3);
    \path (2) edge (4);
    \path[->] (f0) edge node[above] {$\frac13$} (2);
\end{scope}
\end{tikzpicture}
\caption{\ru3.}
\end{minipage}
\end{figure}

Now we can proceed by proving \Cref{thm:injg4} using the discharging procedure together with the structural properties of $G$ proven in \Cref{injg4:structure} and the discharging rules stated above.

\begin{proof}[Proof of \Cref{thm:injg4}]
Let $G$ be a minimal counterexample to the theorem and let $\mu(u)$ be the initial charge assignment for the vertices and faces of $G$ with the charge $\mu(u)=d(u)-4$ for each vertex $u\in V(G)$, and $\mu(f)=d(f)-4$ for each face $f\in F(G)$. By \Cref{equation}, we have that the total sum of the charges is negative.

Let $\mu^*$ be the assigned charges after the discharging procedure. In what follows, we will prove that: $$\forall x \in V(G)\cup F(G), \mu^*(x)\geq 0.$$

Let $u$ be a vertex in $V(G)$. Vertex $u$ has degree at least $2$ by \Cref{injg4:minimumDegree}. Recall that $\Delta(G)=4$ and $\mu(u)=d(u)-4$.

\textbf{Case 1:} If $d(u)=4$, then $u$ does not give any charge. So,
$$\mu^*(u)=\mu(u)=d(u)-4=0.$$

\textbf{Case 2:} If $d(u)=3$, then $\mu(u)=d(u)-4=-1$ and we have the following cases:
\begin{itemize}
\item If $u$ is a small $3$-vertex, then $u$ is incident to one good $5^+$-face due to \Cref{injg4:config1}(v) and \Cref{injg4:config3}. By \ru1, we have
$$ \mu^*(u)\geq -1 + 1 = 0.$$
\item If $u$ is a medium $3$-vertex, then it is incident to two good $5^+$-faces due to \Cref{injg4:config1}(v) and \Cref{injg4:config3}. By \ru2, we have
$$ \mu^*(u)\geq -1 + 2\cdot\frac12 = 0.$$
\item If $u$ is a large $3$-vertex, then by definition, $u$ is incident to only $5^+$-face. Thus, by \ru3, we have
$$ \mu^*(u)\geq -1 + 3\cdot\frac13 = 0.$$
\end{itemize}

\textbf{Case 3:} If $d(u)=2$, then $u$ has to be incident to only $5^+$-faces due to \Cref{injg4:config1}(iv). By \ru0, we have
$$ \mu^*(u)\geq d(u)-4+2\cdot 1 = 0.$$

Let $f$ be a face in $F(G)$. Recall that $\mu(f)=d(f)-4$ and $d(f)\geq 4$ since $g(G)\geq 4$. Let $i_0$, $i_1$, $i_2$, and $i_3$ be respectively the number of times $f$ gives charge by \ru0, \ru1, \ru2, and \ru3. We distinguish the following cases.\\
\textbf{Case 1:} $d(f)\geq 7$\\
Let $u$ and $v$ be two small vertices on $f$. By \Cref{injg4:config4}, $u$ and $v$ must be at facial-distance at least 3 on $f$. As a result, the neighbors of $u$ and $v$ on $f$ are distinct. Moreover, due to \Cref{injg4:config1}(i), those neighbors are $4$-vertices. Thus, we also have $i_2+i_3\leq d(f)-3(i_0+i_1)$. Due to \Cref{injg4:config1}(i), we also have $i_2+i_3\leq \frac{1}{2}d(f)$. Consequently, $i_2+i_3\leq \min(d(f)-3(i_0+i_1),\frac{1}{2}d(f))$.

We claim that $f$ gives at most $\frac{5}{12}d(f)$ charge away. Indeed, recall that $f$ gives $i_0+i_1+\frac{1}{2}i_2+\frac{1}{3}i_3$ by \ru0, \ru1, \ru2, and \ru3. By the above inequalities,
\begin{itemize}
\item if $d(f)-3(i_0+i_1)\leq \frac{1}{2}d(f)$, then $i_0+i_1\geq \frac{1}{6}d(f)$. Moreover, we get 
\begin{align*}
i_0+i_1+\frac{1}{2}i_2+\frac{1}{3}i_3 & \leq i_0+i_1 +\frac12(i_2+i_3)\\
& \leq i_0+i_1+\frac12(d(f)-3(i_0+i_1))\\
& \geq \frac12(d(f)-(i_0+i_1))\\
& \geq \frac{1}{2}(d(f)-\frac16 d(f))\\
& \geq \frac{5}{12}d(f)
\end{align*}
\item if $d(f)-3(i_0+i_1)>\frac{1}{2}d(f)$, then $i_0+i_1 < \frac{1}{6}d(f)$. Moreover, we get 
\begin{align*}
i_0+i_1+\frac{1}{2}i_2+\frac{1}{3}i_3 & \leq i_0+i_1 +\frac12(i_2+i_3)\\
& \leq \frac16 d(f)+\frac12\cdot\frac12 d(f)\\
& \geq \frac{5}{12}d(f)
\end{align*}
\end{itemize}
To conclude, we have 
$$ \mu^*(f)\geq d(f)-4-\frac{5}{12}d(f)\geq \frac{7}{12}d(f)-4\geq 0$$
since $d(f)\geq 7$.

\textbf{Case 2:} $d(f)=6$\\
Similar to the previous case, two small vertices cannot be at facial-distance 2 on $f$ by \Cref{injg4:config4}. As a result, we get $i_0+i_1\leq 2$. Moreover, by \Cref{injg4:config1}(i), two $3^-$-vertices cannot be adjacent, so we get $i_0+i_1+i_2+i_3\leq 3 \Leftrightarrow i_2+i_3\leq 3-(i_0+i_1)$. Now, we distinguish the following cases.
\begin{itemize}
\item If $i_0+i_1=2$, then observe that, since small vertices cannot share neighbors on $f$ and that their neighbors are all $4$-vertices, we have exactly two $3^-$-vertices on $f$. In other words, $i_2+i_3=0$. Recall that $\mu(f)=d(f)-4=2$ and that $f$ gives $i_0+i_1+\frac12 i_2 + \frac13 i_3$ by \ru0, \ru1, \ru2, and \ru3. Thus,
$$ \mu^*(f)\geq 2 - (i_0+i_1+\frac12 i_2 + \frac13 i_3) = 2-2+0=0.$$
\item If $i_0+i_1\leq 1$, then we get
$$ \mu^*(f)\geq 2 - (i_0+i_1+\frac12 i_2 + \frac13 i_3) \geq 2 - (i_0+i_1+\frac12 (3-(i_0+i_1)))\geq \frac12 - \frac12(i_0+i_1)\geq 0.$$ 
\end{itemize}

\textbf{Case 3:} $d(f)=5$\\
Recall that $\mu(f)=d(f)-4=1$. Observe that we have the following inequalities.
\begin{itemize}
\item $i_0+i_1+i_2+i_3\leq 2$ since there are no adjacent $3^-$-vertices by \Cref{injg4:config1}(i).
\item $i_0\leq 1$ due to \Cref{injg4:config1}(ii).
\item $i_1\leq 1$ due to \Cref{injg4:config3}. 
\end{itemize}
Recall that $f$ gives $i_0+i_1+\frac12 i_2 + \frac13 i_3$.
\begin{itemize}
\item If $i_0=1$, then either $f$ is incident to a $3$-vertex, in which case, it is a bad $5$-face and \ru1, \ru2, \ru3 do not apply (by definition of a bad face), or it is not incident to any $3$-vertex. In both cases, $i_1+i_2+i_3=0$. So,
$$ \mu^*(f)\geq 1 - (i_0+i_1+\frac12 i_2 + \frac13 i_3)\geq 1 - (1 + 0)=0.$$
\item If $i_0=0$ and $i_1=1$, then $f$ cannot be incident to any other (than the small $3$-vertex) $3^-$-vertices due to \Cref{injg4:config3}. As a result, $i_2+i_3=0$. So, $$ \mu^*(f)\geq 1 - (i_0+i_1+\frac12 i_2 + \frac13 i_3)\geq 1 - (1 + 0)=0.$$
\item If $i_0=i_1=0$, then 
$$ \mu^*(f)\geq 1 - (i_0+i_1+\frac12 i_2 + \frac13 i_3)\geq 1 - (0 + \frac12\cdot 2)=0.$$
\end{itemize}

\textbf{Case 4:} $d(f)=4$\\
Recall that $\mu(f)=d(f)-4=0$. Since $f$ does not give any charge, we have
$$ \mu^*(f)=\mu(f)=0.$$

We started with a negative total charge, but after the discharging procedure, which preserved the total sum, we end up with a non-negative total sum, a contradiction with \Cref{equation}. In other words, there exist no counter-examples to \Cref{thm:injg4}.
\end{proof}

\section{Exact square list-coloring of planar graphs}\label{exact}

Let $G$ be a minimal counterexample to \Cref{thm:exact}. More precisely, $G$ has maximum degree 4 and $\chi^{\#2}_\ell(G)\geq 11$.

\subsection{Structural properties of $G$}\label{exact:structure}

Unlike injective coloring where vertices incident to the same triangle need different colors, in exact square coloring, two vertices see each if and only if the distance between them is exactly 2.

\begin{lemma}\label{exact:minimumDegree}
The minimum degree of $G$ is at least 2.
\end{lemma}

\begin{proof}
Suppose that $v$ is a $1$-vertex in $G$. We now color $G-\{v\}$ by minimality and then we can color $v$, since $v$ sees at most three vertices and we have $10$ colors.
\end{proof}

Similarly as in the previous section, a $2$-vertex is not reducible by itself, but it provides a nice counting argument to prove that ``smaller'' objects must be further away.

\begin{lemma}\label{exact:counting}
If a $4^-$-vertex $u$ is adjacent with a $3^-$-vertex $v$, then $d^{\#2}(u)\geq 10$. Moreover, if $v$ is a $2$-vertex, then every vertex in $N^{\#2}(u)$ is a $4$-vertex.
\end{lemma}

\begin{proof}
Suppose to the contrary and let $u$ be a $4^-$-vertex with $d^{\#2}(u)< 10$ adjacent with a $3^-$-vertex $v$. Note that if $d(u)\le 3$, then the condition $d^{\#2}(u)< 10$ is always satisfied. Thus, we have $d^{\#2}(v)< 10$. Now we color $G-\{uv\}$ by minimality and uncolor the vertices $u$ and $v$. Next, observe that since $d^{\#2}(u)< 10$ and $d^{\#2}(v)< 10$, we have that $|L(u)|\ge 1$ and $|L(v)|\ge 1$. We can thus finish the coloring by first coloring $u$ and then $v$.
Assume now that $v$ is a $2$-vertex. From above we have that $u$ must be a $4$-vertex adjacent with three other $4$-vertices. Thus, we have that $d^{\#2}(u) = 10$. Suppose that there exists a $3^-$-vertex $w\in N^{\#2}(u)$. In that case we color $G-{uv}$ by minimality and uncolor the vertices $u$, $v$, and $w$. Since $w\in N^{\#2}(u)$, $d(v) = 2$, and $d(w)\le 3$, we have that the remaining list of colors for these vertices have size: $|L(u)|\ge 1$, $|L(v)|\ge 4$, and $|L(w)|\ge 2$. We can then finish by coloring $u$, $w$, and $v$ in this order.
\end{proof}

As a consequence of \Cref{exact:counting}, we obtain the configurations in \Cref{exact:config1} directly.

\begin{lemma}\label{exact:config1}
Graph $G$ cannot contain the following configurations:
\begin{itemize}
\item[(i)] Two adjacent $3^-$-vertices.
\item[(ii)] A $4$-vertex adjacent with a $2$-vertex and a $3^-$-vertex.
\item[(iii)] A $4$-vertex adjacent with three $3$-vertices. 
\item[(iv)] A $3^-$-vertex incident to a $3$-cycle.
\item[(v)] A $3^-$-vertex at distance $1$ from a $3$-cycle.
\item[(vi)] A $2$-vertex incident to a $4$-cycle.
\item[(vii)] A $3$-vertex incident to two adjacent $4$-cycles.
\end{itemize}
\end{lemma}

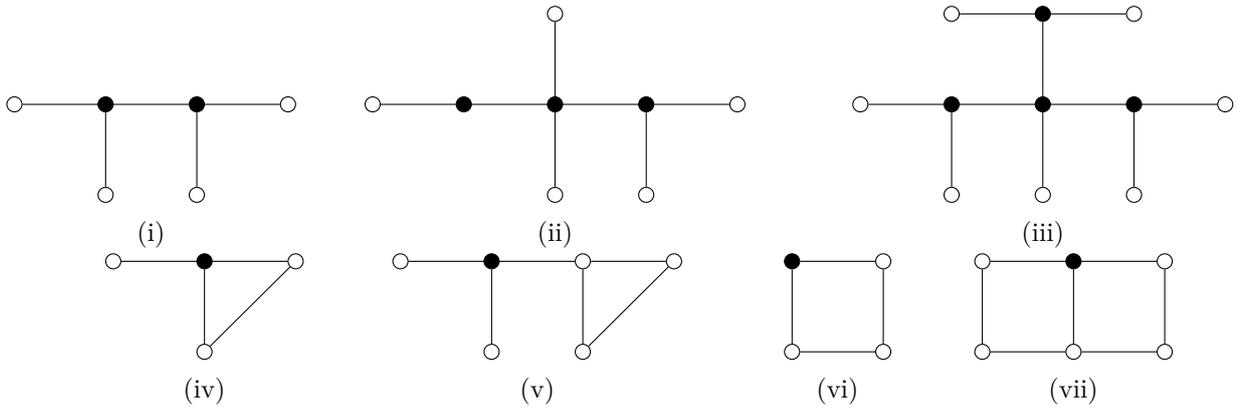
\begin{figure}[H]
\centering
\begin{subfigure}[b]{0.24\textwidth}
\centering
\begin{tikzpicture}[scale=0.6]{thick}
\begin{scope}[every node/.style={circle,draw,minimum size=1pt,inner sep=2}]
    \node (1) at (0,0) {};
    \node[fill] (2) at (2,0) {};
    \node (20) at (2,-2) {};
    \node[fill] (3) at (4,0) {};
    \node (30) at (4,-2) {};
    \node (4) at (6,0) {};
\end{scope}

\begin{scope}[every edge/.style={draw=black}]
    \path (1) edge (4);
    \path (2) edge (20);
    \path (3) edge (30);
\end{scope}
\end{tikzpicture}
\caption{}
\end{subfigure}
\begin{subfigure}[b]{0.37\textwidth}
\centering
\begin{tikzpicture}[scale=0.6]{thick}
\begin{scope}[every node/.style={circle,draw,minimum size=1pt,inner sep=2}]
    \node (0) at (-2,0) {};
    \node[fill] (1) at (0,0) {};
    \node[fill] (2) at (2,0) {};
    \node (20) at (2,-2) {};
    \node (21) at (2,2) {};
    \node[fill] (3) at (4,0) {};
    \node (4) at (6,0) {};
    \node (30) at (4,-2) {};
\end{scope}

\begin{scope}[every edge/.style={draw=black}]
    \path (0) edge (4);
    \path (21) edge (20);
    \path (3) edge (30);
\end{scope}
\end{tikzpicture}
\caption{}
\end{subfigure}
\begin{subfigure}[b]{0.37\textwidth}
\centering
\begin{tikzpicture}[scale=0.6]{thick}
\begin{scope}[every node/.style={circle,draw,minimum size=1pt,inner sep=2}]
    \node (0) at (-2,0) {};
    \node[fill] (1) at (0,0) {};
    \node[fill] (2) at (2,0) {};
    \node (20) at (2,-2) {};
    \node[fill] (21) at (2,2) {};
    \node[fill] (3) at (4,0) {};
    \node (4) at (6,0) {};
    \node (30) at (4,-2) {};
    \node (10) at (0,-2) {};
    \node (11) at (0,2) {};
    \node (31) at (4,2) {};
\end{scope}

\begin{scope}[every edge/.style={draw=black}]
    \path (0) edge (4);
    \path (21) edge (20);
    \path (3) edge (30);
    \path (1) edge (10);
    \path (11) edge (31);
\end{scope}
\end{tikzpicture}
\caption{}
\end{subfigure}

\begin{subfigure}[b]{0.2\textwidth}
\centering
\begin{tikzpicture}[scale=0.6]{thick}
\begin{scope}[every node/.style={circle,draw,minimum size=1pt,inner sep=2}]
    \node (1) at (0,0) {};
    \node[fill] (2) at (2,0) {};
    \node (20) at (2,-2) {};
    \node (3) at (4,0) {};
\end{scope}

\begin{scope}[every edge/.style={draw=black}]
    \path (1) edge (3);
    \path (2) edge (20);
    \path (20) edge (3);
\end{scope}
\end{tikzpicture}
\caption{}
\end{subfigure}
\begin{subfigure}[b]{0.3\textwidth}
\centering
\begin{tikzpicture}[scale=0.6]{thick}
\begin{scope}[every node/.style={circle,draw,minimum size=1pt,inner sep=2}]
    \node (1) at (0,0) {};
    \node[fill] (2) at (2,0) {};
    \node (20) at (2,-2) {};
    \node (3) at (4,0) {};
    \node (30) at (4,-2) {};
    \node (4) at (6,0) {};
\end{scope}

\begin{scope}[every edge/.style={draw=black}]
    \path (1) edge (3);
    \path (3) edge (4);
    \path (2) edge (20);
    \path (3) edge (30);
    \path (30) edge (4);
\end{scope}
\end{tikzpicture}
\caption{}
\end{subfigure}
\begin{subfigure}[b]{0.15\textwidth}
\centering
\begin{tikzpicture}[scale=0.6]{thick}
\begin{scope}[every node/.style={circle,draw,minimum size=1pt,inner sep=2}]
    \node[fill] (2) at (2,0) {};
    \node (20) at (2,-2) {};
    \node (3) at (4,0) {};
    \node (30) at (4,-2) {};
\end{scope}

\begin{scope}[every edge/.style={draw=black}]
    \path (2) edge (3);
    \path (2) edge (20);
    \path (3) edge (30);
    \path (20) edge (30);
\end{scope}
\end{tikzpicture}
\caption{}
\end{subfigure}
\begin{subfigure}[b]{0.2\textwidth}
\centering
\begin{tikzpicture}[scale=0.6]{thick}
\begin{scope}[every node/.style={circle,draw,minimum size=1pt,inner sep=2}]
    \node (1) at (0,0) {};
    \node (10) at (0,-2) {};
    \node[fill] (2) at (2,0) {};
    \node (20) at (2,-2) {};
    \node (3) at (4,0) {};
    \node (30) at (4,-2) {};
\end{scope}

\begin{scope}[every edge/.style={draw=black}]
    \path (1) edge (3);
    \path (1) edge (10);
    \path (2) edge (20);
    \path (3) edge (30);
    \path (10) edge (20);
    \path (20) edge (30);
\end{scope}
\end{tikzpicture}
\caption{}
\end{subfigure}
\caption{Reducible configurations in \Cref{exact:config1}.}
\end{figure}

We finish the study of $G$ structural properties by looking at small cycles.

\begin{figure}[!htb]
\centering
\begin{minipage}[b]{0.49\textwidth}
\centering
\begin{tikzpicture}[scale=0.6]{thick}
\begin{scope}[every node/.style={circle,draw,minimum size=1pt,inner sep=2}]
    \node[label={above:$u$},label={below right:1}] (u1) at (2,0) {};
    \node[label={below:$v$},label={above right:1}] (u2) at (2,-2) {};
    \node[label={below:$x$}] (u3) at (4,-2) {};
    \node[label={above:$y$}] (u4) at (4,0) {};
    
    \node[label={left:$w$}] (v1) at (0,-1) {};
\end{scope}

\begin{scope}[every edge/.style={draw=black}]
    \path (u1) edge (u4);
    \path (u1) edge (u2);
    \path (u2) edge (u3);
    \path (u3) edge (u4);
    \path (u1) edge (v1);
    \path (v1) edge (u2);
\end{scope}
\end{tikzpicture}
\quad
\begin{tikzpicture}[scale=0.6]{thick}
\begin{scope}[every node/.style={circle,draw,minimum size=1pt,inner sep=2}]
    \node[label={above:$u$},label={right:1}] (u1) at (2,0) {};
    \node[label={below:$v$},label={right:1}] (u2) at (2,-2) {};
    \node[label={above:$x$}] (u4) at (4,-1) {};
    
    \node[label={left:$w$}] (v1) at (0,-1) {};
\end{scope}

\begin{scope}[every edge/.style={draw=black}]
    \path (u1) edge (u4);
    \path (u1) edge (u2);
    \path (u2) edge (u4);
    \path (u1) edge (v1);
    \path (v1) edge (u2);
\end{scope}
\end{tikzpicture}
\caption{Reducible configurations in \Cref{exact:config2}.}
\end{minipage}
\begin{minipage}[b]{0.49\textwidth}
\centering
\begin{tikzpicture}[scale=0.6]{thick}
\begin{scope}[every node/.style={circle,draw,minimum size=1pt,inner sep=2}]
    \node[label={above: $w$},label={left: $2$}] (u) at (2,0) {};
    \node[label={below:$v$},label={left: $2$}] (v) at (2,-2) {};
    \node[fill,label={below:$u$},label={above: $2$}] (x) at (4,-1) {};
    \node[label={above: $y$},label={right: $2$}] (y) at (6,0) {};
    \node[label={below: $x$},label={right: $2$}] (z) at (6,-2) {};
\end{scope}

\begin{scope}[every edge/.style={draw=black}]
    \path (u) edge (v);
    \path (u) edge (x);
    \path (x) edge (v);
    \path (x) edge (y);
    \path (x) edge (z);
    \path (z) edge (y);
\end{scope}
\end{tikzpicture}
\caption{Reducible configuration in \Cref{exact:config3}.}
\end{minipage}

\begin{minipage}[b]{0.5\textwidth}
\centering
\begin{tikzpicture}[scale=1.5,rotate=90]{thick}
\begin{scope}[every node/.style={circle,draw,minimum size=1pt,inner sep=2}]
    \node[label={left:$u_3$},label={right:$2$}] (u1) at (1.95,1.31) {};
    \node[label={below left:$u_4$},label={above right:$2$}] (u2) at (1,1) {};
    \node[label={below right:$u_5$},label={above left:$2$}] (u3) at (1,0) {};
    \node[fill,label={right:$u_1$},label={left:$6$}] (u4) at (1.95,-0.31) {};
    \node[label={above:$u_2$},label={below:$2$}] (u5) at (2.54,0.5) {};
    
    \node (v1) at (1.22,1.93) {};
    
\end{scope}

\begin{scope}[every edge/.style={draw=black}]
    \path (u1) edge (u2);
    \path (u2) edge (u3);
    \path (u3) edge (u4);
    \path (u4) edge (u5);
    \path (u5) edge (u1);
    \path (u1) edge (v1);
    \path (u2) edge (v1);
\end{scope}
\end{tikzpicture}
\caption{Reducible configuration in \Cref{exact:config4}.}
\end{minipage}
\end{figure}

\begin{lemma}\label{exact:config2}
A $3$-cycle in $G$ is not adjacent with a $4^-$-cycle.
\end{lemma}

\begin{proof}
Suppose to the contrary and let $C = uvw$ be a $3$-cycle in $G$. Let $C' = uvxy$ be a $4^-$-cycle in $G$. Note that possibly $x = y$, in which case $C'$ is a $3$-cycle. We now color $G-{uv}$ by minimality and uncolor the vertices $u$ and $v$. Observe that whether $x = y$, or $x\not = y$, we have that $|L(u)|\ge 1$ and $|L(v)|\ge 1$. Therefore, we can finish the coloring by first coloring $u$ and then $v$.
\end{proof}

\begin{lemma}\label{exact:config3}
A $3$-cycle in $G$ is not incident to another $3$-cycle.
\end{lemma}

\begin{proof}
Suppose to the contrary and let $C = uvw$ be a $3$-cycle in $G$. Let $C' = uxy$ be another $3$-cycle in $G$. Note that possibly the vertices $v$, $w$, $x$, and $y$ are all distinct, otherwise we are done by \Cref{exact:config2}. We now color $G-{u}$ by minimality and uncolor the vertices $v$, $w$, $x$, and $y$. Observe that the remaining list of colors for vertices $u$, $v$, $w$, $x$, and $y$ have size: $|L(u)|\ge 2$, $|L(v)|\ge 2$, $|L(w)|\ge 2$, $|L(x)|\ge 2$, and $|L(y)|\ge 2$. Since each of the vertices $v$, $w$, $x$, and $y$ sees at most two others, we can color these vertices by the $2$-choosability of even cycles and finally coloring $u$ with one of the two remaining colors.
\end{proof}

\begin{lemma}\label{exact:config4}
A $5$-cycle in $G$ incident to a $2$-vertex is not adjacent to a $3$-cycle.
\end{lemma}

\begin{proof}
Suppose to the contrary and let $C = u_1u_2u_3u_4u_5$ be a $5$-cycle with $d(u_1) = 2$. Due to \Cref{exact:config1}(iv, v) we have that the adjacent $3$-cycle must be incident with the edge $u_3u_4$. We now color $G-{u_3u_4}$ by minimality and uncolor all the vertices of $C$. Note that the remaining list of colors for these vertices have size: $|L(u_1)|\ge 6$, $|L(u_2)|\ge 2$, $|L(u_3)|\ge 2$, $|L(u_4)|\ge 2$, and $|L(u_5)|\ge 2$. Observe that $u_2$ does not see $u_3$, $u_3$ does not see $u_4$ and $u_4$ does not see $u_5$. Thus, we can color the vertices $u_2$, $u_3$, $u_4$, and $u_5$ by the $2$-choosability of paths and finally coloring $u_1$ with one of the four remaining colors.
\end{proof}

\subsection{Discharging procedure}\label{exact:discharging}

To get a contradiction with \Cref{equation} we apply the following rules in the discharging procedure:

\begin{itemize}
\item[\ru0] Every $5^+$-face $f$ gives $1$ to each incident $2$-vertex.
\item[\ru1] Every $5^+$-face $f$ gives $\frac{1}{3}$ to each incident $3$-vertex that is not incident to a $4$-face.
\item[\ru2] Every $5^+$-face $f$ gives $\frac{1}{2}$ to each incident $3$-vertex that is incident to a $4$-face.
\item[\ru3] Every $5^+$-face $f$ gives $\frac{1}{3}$ to each adjacent $3$-face.
\end{itemize}

\begin{figure}[H]
\centering
\begin{minipage}[b]{0.24\textwidth}
\centering
\begin{tikzpicture}[scale=0.5]{thick}
\begin{scope}[every node/.style={circle,draw,minimum size=1pt,inner sep=2}]
    \node (1) at (0,0) {};
    \node[fill,label={below:$u$}] (2) at (2,0) {};
    \node (3) at (4,0) {};
    \node[draw=none] (f) at (2,2) {$f$};
\end{scope}

\begin{scope}[every edge/.style={draw=black}]
    \path (1) edge (3);
    \path[->] (f) edge node[right] {1} (2);
\end{scope}
\end{tikzpicture}
\caption{\ru0.}
\end{minipage}
\begin{minipage}[b]{0.24\textwidth}
\centering
\begin{tikzpicture}[scale=0.6]{thick}
\begin{scope}[every node/.style={circle,draw,minimum size=1pt,inner sep=2}]
    \node (1) at (0,0) {};
    \node[fill,label={below left:$u$}] (2) at (2,0) {};
    \node (20) at (2,-2) {};
    \node (3) at (4,0) {};
    \node[draw=none] (f) at (2,2) {$f$};
\end{scope}

\begin{scope}[every edge/.style={draw=black}]
    \path (1) edge (3);
    \path (2) edge (20);
    \path[->] (f) edge node[right] {$\frac13$} (2);
\end{scope}
\end{tikzpicture}
\caption{\ru1.}
\end{minipage}
\begin{minipage}[b]{0.24\textwidth}
\centering
\begin{tikzpicture}[scale=0.6]{thick}
\begin{scope}[every node/.style={circle,draw,minimum size=1pt,inner sep=2}]
    \node (1) at (0,0) {};
    \node[fill,label={below left:$u$}] (2) at (2,0) {};
    \node (20) at (2,-2) {};
    \node (3) at (4,0) {};
    \node (30) at (4,-2) {};
    \node[draw=none] (f) at (2,2) {$f$};
\end{scope}

\begin{scope}[every edge/.style={draw=black}]
    \path (1) edge (3);
    \path (2) edge (20);
    \path (3) edge (30);
    \path (20) edge (30);
    \path[->] (f) edge node[right] {$\frac12$} (2);
\end{scope}
\end{tikzpicture}
\caption{\ru2.}
\end{minipage}
\begin{minipage}[b]{0.24\textwidth}
\centering
\begin{tikzpicture}[scale=0.5]{thick}
\begin{scope}[every node/.style={circle,draw,minimum size=1pt,inner sep=2}]
    \node (1) at (0,0) {};
    \node[draw=none] (2) at (2,1) {};
    \node (3) at (4,0) {};
    \node (4) at (2,3) {};
    
    \node[draw=none] (f0) at (-1,2) {$f$};
\end{scope}

\begin{scope}[every edge/.style={draw=black}]
    \path (1) edge (3);
    \path (3) edge (4);
    \path (1) edge (4);
    \path[->] (f0) edge node[above] {$\frac13$} (2);
\end{scope}
\end{tikzpicture}
\caption{\ru3.}
\end{minipage}
\end{figure}

Now we can proceed by proving \Cref{thm:exact} using the discharging procedure together with the structural properties of $G$ proven in \Cref{exact:structure} and the discharging rules stated above.

\begin{proof}[Proof of \Cref{thm:exact}]

Let $G$ be a minimal counterexample to the theorem and let $\mu(u)$ be the initial charge assignment for the vertices and faces of $G$ with the charge $\mu(u)=d(u)-4$ for each vertex $u\in V(G)$, and $\mu(f)=d(f)-4$ for each face $f\in F(G)$. By \Cref{equation}, we have that the total sum of the charges is negative.

Let $\mu^*$ be the assigned charges after the discharging procedure. In what follows, we will prove that: $$\forall x \in V(G)\cup F(G), \mu^*(x)\geq 0.$$

Let $u$ be a vertex in $V(G)$. Vertex $u$ has degree at least $2$ by \Cref{exact:minimumDegree}. Recall that $\Delta(G)=4$ and $\mu(u)=d(u)-4$.

\textbf{Case 1:} If $d(u)=4$, then $u$ does not give any charge. So,

$$\mu^*(u)=\mu(u)=d(u)-4=0.$$

\textbf{Case 2:} If $d(u)=3$, then $\mu(u)=d(u)-4=-1$ and due to \Cref{exact:config1}(iv, vii) we have the following two cases:
\begin{itemize}
\item If $u$ is not incident to any $4$-face, then it receives $\frac{1}{3}$ from each of the three incident $5^+$-faces by \ru1. So,
$$ \mu^*(u) = -1 + 3\cdot\frac13 = 0.$$
\item If $u$ is incident to a $4$-face, then it has exactly one incident $4$-face and two incident $5^+$-faces due to \Cref{exact:config1}(vii). Therefore, $u$ receives $\frac12$ from each incident $5^+$-face by \ru2. So,
$$ \mu^*(u) = -1 + 2\cdot\frac12 = 0.$$
\end{itemize}

\textbf{Case 3:} If $d(u)=2$, then $\mu(u)=d(u)-4=-2$ and due to \Cref{exact:config1}(iv, vi) $u$ is incident with two $5^+$-faces. Thus, $u$ receives $1$ from each incident face by \ru0. So,
$$ \mu^*(u) = -2 + 2\cdot 1 = 0.$$

Let $f$ be a face in $F(G)$. Recall that $\mu(f)=d(f)-4$. Let $i_0$, $i_1$, $i_2$, and $i_3$ be respectively the number of times $f$ gives charge by \ru0, \ru1, \ru2, and \ru3. We distinguish the following cases.\\

\textbf{Case 1:} $d(f)\geq 6$\\
We claim that $f$ never gives more than $\frac13 d(f)$ by \ru0-\ru3. Indeed, we argue that $f$ sends at most $\frac13$ per incident edge. Observe that we can also view \ru0, \ru1, and \ru2 as $f$ giving charge along incident edges of $f$ as shown in \Cref{fig:exact:discharging}. Consider now an edge $uv$ incident to $f$. If $f$ gives $\frac14$ by \ru0, then it cannot give any additional charge along $uv$ since a $2$-vertex is at distance at least 2 from any $3$-face by \Cref{exact:config1}(iv) and is also at distance at least $4$ from any other $3^-$-vertex by \Cref{exact:counting}. If $f$ gives either $\frac16$ (respectively $\frac14$) along $uv$ by \ru1 (respectively \ru2), then $f$ cannot give any additional charge along $uv$. Indeed, as argued above, $f$ cannot give charge along $uv$ by \ru0 and also $f$ cannot give additional charge along $uv$ by \ru1-\ru3 due to \Cref{exact:config1}(i, iv, v). Finally, if $f$ gives $\frac13$ along $uv$ by \ru3, then, due to \Cref{exact:config1}(iv, v) and \Cref{exact:config2}, $f$ cannot give any additional charge along $uv$. To conclude, we have
$$ \mu^*(u)\ge \mu(u)-\frac13 d(f) = d(f)-4-\frac13 d(f) = \frac23 d(f)-4\ge 0$$
since $d(f)\ge 6$.\\

\begin{figure}[H]
\centering
\begin{subfigure}[b]{0.33\textwidth}
\centering
\begin{tikzpicture}[scale=0.7]{thick}
\begin{scope}[every node/.style={circle,draw,minimum size=1pt,inner sep=2}]
    \node (0) at (-2,0) {};
    \node (1) at (0,0) {};
    \node[fill] (2) at (2,0) {};
    \node (3) at (4,0) {};
    \node (4) at (6,0) {};
    \node[draw=none] (f) at (2,1.5) {$f$};
\end{scope}

\begin{scope}[every edge/.style={draw=black}]
    \path (0) edge (1);
    \path (1) edge (3);
    \path (3) edge (4);
    \draw[->] (f) to[out=0,in=90] (3,0) node[above right] {$\frac14$} to[out=-90,in=-45] (2);
    \draw[->] (f) to[out=180,in=90] (1,0) node[above left] {$\frac14$} to[out=-90,in=225] (2);
    \draw[->] (f) to[out=0,in=90] (5,0) node[above right] {$\frac14$} to[out=-90,in=-45] (2);
    \draw[->] (f) to[out=180,in=90] (-1,0) node[above left] {$\frac14$} to[out=-90,in=225] (2);
\end{scope}
\end{tikzpicture}
\caption*{\ru0.}
\end{subfigure}
\begin{subfigure}[b]{0.32\textwidth}
\centering
\begin{tikzpicture}[scale=0.7]{thick}
\begin{scope}[every node/.style={circle,draw,minimum size=1pt,inner sep=2}]
    \node (1) at (0,0) {};
    \node[fill] (2) at (2,0) {};
    \node (20) at (2,-2) {};
    \node (3) at (4,0) {};
        \node[draw=none] (f) at (2,1.5) {$f$};
\end{scope}

\begin{scope}[every edge/.style={draw=black}]
    \path (1) edge (3);
    \path (2) edge (20);
    \draw[->] (f) to[out=0,in=90] (3,0) node[above right] {$\frac16$} to[out=-90,in=-45] (2);
    \draw[->] (f) to[out=180,in=90] (1,0) node[above left] {$\frac16$} to[out=-90,in=225] (2);
\end{scope}
\end{tikzpicture}
\caption*{\ru1.}
\end{subfigure}
\begin{subfigure}[b]{0.32\textwidth}
\centering
\begin{tikzpicture}[scale=0.7]{thick}
\begin{scope}[every node/.style={circle,draw,minimum size=1pt,inner sep=2}]
    \node (1) at (0,0) {};
    \node[fill] (2) at (2,0) {};
    \node (20) at (2,-2) {};
    \node (3) at (4,0) {};
    \node (30) at (4,-2) {};
    \node[draw=none] (f) at (2,1.5) {$f$};
\end{scope}

\begin{scope}[every edge/.style={draw=black}]
    \path (1) edge (3);
    \path (2) edge (20);
    \path (3) edge (30);
    \path (20) edge (30);
    \draw[->] (f) to[out=0,in=90] (3,0) node[above right] {$\frac14$} to[out=-90,in=-45] (2);
    \draw[->] (f) to[out=180,in=90] (1,0) node[above left] {$\frac14$} to[out=-90,in=225] (2);
\end{scope}
\end{tikzpicture}
\caption*{\ru2.}
\end{subfigure}
\caption{\label{fig:exact:discharging}}
\end{figure}

\textbf{Case 2:} $d(f)=5$\\
Recall that $\mu(f)=d(f)-4=1$. Observe that we have the following inequalities regarding the values $i_0$, $i_1$, $i_2$, and $i_3$:
\begin{itemize}
\item $i_0\leq 1$, due to \Cref{exact:config1}(i, ii).
\item $i_1 + i_2 + i_3\leq 2$. Indeed, by \Cref{exact:config1}(i) we have that $f$ is incident to at most two $3$-vertices. Moreover, by \Cref{exact:config3} we have that $f$ is adjacent with at most two $3$-faces. Next, if $f$ is incident to exactly one $3$-vertex, then by \Cref{exact:config1}(iv, v) we have that $f$ is adjacent with at most one $3$-face. Finally, if $f$ is incident to exactly two $3$-vertices, then again by \Cref{exact:config1}(iv, v) we have that $f$ is not adjacent to any $3$-faces.
\end{itemize}

Recall that $f$ gives $i_0 + \frac13 i_1 + \frac12 i_2 + \frac13 i_3$ by \ru0, \ru1, \ru2, and \ru3.\\

If $i_0 = 1$, then, by \Cref{exact:config1}(i, ii, iv, v) and \Cref{exact:config4}, we have that $i_1 + i_2 + i_3 = 0$. So,
$$ \mu^*(f)\geq 1 - 1=0.$$ 

If $i_0 = 0$, then since $i_1 + i_2 + i_3\leq 2$,
\begin{align*}
\mu^*(f)& \geq \mu(f)-\frac{1}{3}i_i-\frac{1}{2}i_2-\frac{1}{3}i_3\\
& \geq 1-\frac{1}{2}(i_1 + i_2 + i_3)\\
& \geq 1-2\cdot\frac{1}{2} = 0.
\end{align*}

\textbf{Case 3:} $d(f)=4$\\
Recall that $\mu(f)=d(f)-4=0$. Since $f$ does not give any charge, we have
$$ \mu^*(f)=\mu(f)=0.$$

\textbf{Case 4:} $d(f)=3$\\
Recall that $\mu(f)=d(f)-4=-1$. Due to \Cref{exact:config2} we have that $f$ receives $\frac13$ from each adjacent face by \ru3. So,
$$ \mu^*(f)=-1 + 3\cdot\frac13=0.$$

To conclude, we started with a charge assignment with a negative total sum, but after the discharging procedure, which preserved that sum, we end up with a non-negative one, which is a contradiction. In other words, there exists no counter-example to \Cref{thm:exact}.
\end{proof}

\section{Conclusion}\label{conclusion}

Almost all results in $2$-distance (and also injective) coloring of planar graphs (of high girth) are proved using the discharging method. Many of these proofs use planarity only to bound the number of edges with respect to the number of vertices in any induced subgraph, which can also be achieved by bounding the maximum average degree. The \emph{maximum average degree}, denoted by $\mad(G)$, is taken as the maximum over all subgraphs $H$ of $G$ of average degree of $H$. Note that for planar graphs of girth $g$ we have the inequality $\mad(G) < \frac{2g}{g-2}$. This allows one to easily translate the condition on maximum average degree to the condition on girth in the case of planar graphs. Thus, we can ask if our results can be extended to non-planar graphs $G$ which have maximum average degree $\mad(G)<\frac{2\cdot 4}{4-2}=4$.

To answer this question we need to look at finite projective planes. For a prime power $q$, a projective plane $PG(2,q)$, of order $q$, consists of $q^2 + q + 1$ points and $q^2 + q + 1$ lines such that any two points belong to exactly one line and any two lines intersect in exactly one point. Moreover, each point belongs to exactly $q+1$ lines and each line contains exactly $q+1$ points. Let $IG(q)$ be the incidence graph of a projective plane $PG(2,q)$ (i.e., a bipartite graph whose vertices are the points and lines of $PG(2,q)$ and an edge $uv\in E(IG(q))$ if and only if a line $u$ contains a point $v$). In~\cite{HahnKraSirSott02}, Hahn \textit{et al.} proved the following result.

\begin{theorem}[Hahn \textit{et al.}~\cite{HahnKraSirSott02}]\label{HahnIG}
Let $G$ be a connected graph of maximum degree $\Delta\ge 3$. Then, $\chi^{i}(G) = \Delta^2 - \Delta + 1$ if and only if there exists a projective plane of order $q = \Delta - 1$ and $G$ is isomorphic to $IG(q)$.
\end{theorem}

A similar result, but for the exact square coloring was shown by Focaud \textit{et al.} in~\cite{FouEtAl}.

\begin{theorem}[Foucaud \textit{et al.}~\cite{FouEtAl}]\label{FocaudIG}
Let $G$ be a connected graph of maximum degree $\Delta\ge 3$. Then, $\chi^{\#2}(G) = \Delta^2 - \Delta + 1$ if and only if there exists a projective plane of order $q = \Delta - 1$ and $G$ is isomorphic to $IG(q)$.
\end{theorem}

\begin{figure}[H]
    \centering
    \includegraphics[scale=0.8]{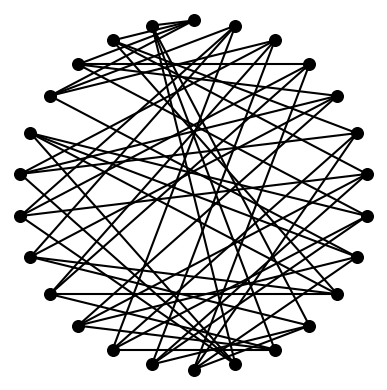}
    \caption{A $4$-regular bipartite graph $IG(3)$.}
    \label{fig:concl:bipartite}
\end{figure}

None of our results on injective and exact square coloring can be extended to non-planar graphs $G$ with maximum average degree $\mad(G)<4$ even in their non-list version. Indeed, note that for the graph $G = IG(3)$ as seen in \Cref{fig:concl:bipartite} we have $\chi^i(G)=\chi^{\#2}(G) = 13$ by \Cref{HahnIG,FocaudIG}. Moreover, removing any vertex from $G$ yields a graph $G'$ with $\mad(G')<4$ and $\chi^i(G') = \chi^{\#2}(G') = 12$, which still proves that our results are ``optimal'' in the sense that planarity is needed not only for sparseness. It is easy to see that $\chi^{\#2}_\ell(G)\leq \chi^i_\ell(G)\leq 12$ for every graph $G$ with $\Delta(G)=4$ and $\mad(G)<4$ and this bound is tight for the graph $G'$ as explained above.

\begin{theorem}\label{thm:mad}
For every graph $G$ with $\Delta(G)=4$ and $\mad(G)<4$, $\chi^i_\ell(G)\leq 12$. 
\end{theorem}

Indeed, if a counter-example $G$ to \Cref{thm:mad} exists, then it would have minimum degree 4 (the proof is similar to \Cref{injg3:minimumDegree}), which is a contradiction since $G$ would be a $4$-regular graph but $\mad(G)<4$. 

For 2-distance coloring, our result is also not extendable to graphs with $\mad(G)<4$ due to the graph in \Cref{fig:concl:mad2d} for which $\chi^2(G) = 13$. In 2016, Cranston and Rabern~\cite{CraRab16} showed the following result.

\begin{theorem}[Cranston and Rabern~\cite{CraRab16}]
If $G$ is a connected graph with maximum degree $\Delta \ge 3$ and $G$ is not
the Peterson graph, the Hoffman-Singleton graph, or a Moore graph with $\Delta = 57$, then $\chi^2_{\ell}(G)\le  \Delta^2 - 1$.
\end{theorem}

Thus, when $\Delta(G) = 4$, we immediately get that $\chi^2_{\ell}(G)\le 15$. We believe that the following conjecture is true. 

\begin{conjecture}
For every graph $G$ with $\Delta(G) = 4$ and $\mad(G)<4$, $\chi^2(G)\le 13$.
\end{conjecture}

\begin{figure}[!htb]
    \centering
    \includegraphics[scale=0.8]{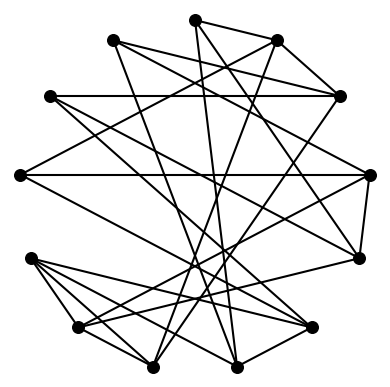}
    \caption{A graph $G$ on 13 vertices with $\Delta(G)=4$, $\mad(G)<4$, and diameter 2.}
    \label{fig:concl:mad2d}
\end{figure}

We are also interested in planar graphs with girth 4. Due to constructions in \cite{lm21bis} and \cite{wegner} we have some lower bounds for the 2-distance coloring of planar graphs with girth 4. We believe these lower bounds to be tight and thus we conjecture the following.

\begin{conjecture}
For every planar graph $G$ with girth at least 4 and maximum degree $\Delta$, $\chi^2(G)\le \Delta+3$ for $3\leq \Delta\leq 5$ and $\chi^2(G)\le\lfloor\frac32\Delta\rfloor$ for $\Delta\geq 6$.
\end{conjecture}

Similarly, due to constructions in~\cite{LuzSkreTan09}, we conjecture the following for injective and exact square coloring which is the same in the case of planar graphs with girth 4.

\begin{conjecture}
For every planar graph $G$ with girth at least 4 and maximum degree $\Delta$, $\chi^i(G)\le 4$ for $\Delta=3$, $\chi^i(G)=\chi^{\#2}(G)\le \Delta+2$ for $4\leq \Delta\leq 5$ and $\chi^i(G)=\chi^{\#2}(G)\le \lfloor\frac32\Delta\rfloor$ for $\Delta\geq 6$.
\end{conjecture}

Finally, unlike in the injective coloring, in the exact square coloring we are only concerned by conflicts between vertices at distance exactly 2, hence the presence of triangles does not create any conflicts. Thus, we believe that a similar conjecture holds for exact square coloring.

\begin{conjecture}
For every planar graph $G$ with maximum degree $\Delta$, $\chi^{\#2}(G)=\lfloor\frac32\Delta\rfloor$.
\end{conjecture}

\section*{Acknowledgements}
Both authors were supported by the research grant PHC PROTEUS 2020, project EColoGra N° 44236RL. The first author was also supported by the INS2I project ACDG.
The second author acknowledges support by the Young Researchers Grant of the Slovenian Research Agency and the partial support by the program P1--0383 and the project J1--3002.

\bibliographystyle{plain}
\bibliography{References.bib}

\end{document}